\documentclass[english]{article}
\usepackage{amsmath}
\usepackage{amsthm}
\usepackage{amssymb}
\usepackage	{amsrefs}
\usepackage[all]{xy}
\usepackage{enumitem}
\usepackage{color}
\usepackage[hidelinks]{hyperref}
\usepackage{fancyhdr}

\setlist[enumerate,1]{label={(\roman*)}}

\makeatletter
\numberwithin{equation}{section}
\theoremstyle{plain}
\newtheorem{thm}{\protect\theoremname}[section]
  \theoremstyle{definition}
  \newtheorem{example}[thm]{\protect\examplename}
  \theoremstyle{definition}
  \newtheorem{defn}[thm]{\protect\definitionname}
  \theoremstyle{plain}
  \newtheorem{prop}[thm]{\protect\propositionname}
  \theoremstyle{remark}
  \newtheorem{rem}[thm]{\protect\remarkname}
  \theoremstyle{plain}
  \newtheorem{lem}[thm]{\protect\lemmaname}

\newcommand{\myarr}[2]{\ar[#1]^-{#2}}
\SelectTips{cm}{}

\makeatother

  \providecommand{\definitionname}{Definition}
  \providecommand{\examplename}{Example}
  \providecommand{\lemmaname}{Lemma}
  \providecommand{\propositionname}{Proposition}
  \providecommand{\remarkname}{Remark}
\providecommand{\theoremname}{Theorem}

\newcommand\shorttitle{ABELIAN GROUPS ARE POLYNOMIALLY STABLE}
\newcommand\authors{OREN BECKER AND JONATHAN MOSHEIFF}

\begin{document}

\global\long\def\NN{\mathbb{{N}}}
\global\long\def\ZZ{\mathbb{{Z}}}
\global\long\def\QQ{\mathbb{{Q}}}
\global\long\def\RR{\mathbb{{R}}}
\global\long\def\FF{\mathbb{{F}}}
\global\long\def\GL{\mathbb{\text{\mathrm{{GL}}}}}
\global\long\def\frakg{\mathfrak{g}}
\global\long\def\id{\mathrm{{id}}}
\global\long\def\Sym{\mathrm{{Sym}}}
\global\long\def\BS{\mathrm{{BS}}}
\global\long\def\Sub{\mathrm{{Sub}}}
\global\long\def\IRS{\mathrm{{IRS}}}
\global\long\def\RS{\mathrm{{RS}}}
\global\long\def\calX{\mathcal{{X}}}
\global\long\def\calG{\mathcal{{G}}}
\global\long\def\calC{\mathcal{{C}}}
\global\long\def\calD{\mathcal{{D}}}
\global\long\def\calM{\mathcal{{M}}}
\global\long\def\calL{\mathcal{{L}}}
\global\long\def\calB{\mathcal{{B}}}
\global\long\def\calE{\mathcal{{E}}}
\global\long\def\calT{\mathcal{{T}}}
\global\long\def\gen{\mathrm{{gen}}}
\global\long\def\ball{\mathrm{{ball}}}
\global\long\def\stat{\mathrm{{stat}}}
\global\long\def\Prob{\mathrm{{Prob}}}
\global\long\def\Stab{\mathrm{{Stab}}}
\global\long\def\Ker{\mathrm{{Ker}}}
\global\long\def\prodname{\mathrm{{prod}}}
\global\long\def\std{\mathrm{{std}}}
\global\long\def\BAD{\mathrm{{BAD}}}
\global\long\def\Tor{\mathrm{{Tor}}}
\global\long\def\rank{\mathrm{{rank}}}
\global\long\def\Box{\mathrm{{Box}}}
\global\long\def\cost{\mathrm{{cost}}}
\global\long\def\calR{{\mathcal{R}}}
\global\long\def\SNS{\mathrm{{SNS}}}
\global\long\def\eff{\mathrm{{eff}}}
\global\long\def\scar{\mathrm{{scar}}}
\global\long\def\lla{{\langle\!\langle}}
\global\long\def\rra{{\rangle\!\rangle}}
\global\long\def\Tietze{\mathrm{{Tietze}}}
\global\long\def\spn{\mathrm{{span}}}
\global\long\def\SR{\mathrm{{SR}}}
\global\long\def\Eq{\mathrm{{Eq}}}
\global\long\def\Ima{\mathrm{{Im}}}
\global\long\def\domain{\mathrm{{domain}}}
\global\long\def\defname{\mathrm{{def}}}

\title{Abelian groups are polynomially stable}
\author{
	Oren Becker
	\thanks{Department of Mathematics, Hebrew University, Jerusalem, Israel. e-mail: oren.becker@mail.huji.ac.il. Supported by the ERC.}
	\and
	Jonathan Mosheiff
	\thanks{Department of Computer Science, Hebrew University, Jerusalem, Israel. e-mail: yonatanm@cs.huji.ac.il. Supported by the Adams Fellowship Program of the Israel Academy of Sciences and Humanities.}}
\date{}

\maketitle
\begin{abstract}
In recent years, there has been a considerable amount of interest in stability of equations and their corresponding groups. Here, we initiate the systematic study of the quantitative aspect of this theory. We develop a novel method, inspired by the Ornstein-Weiss quasi-tiling technique, to prove that abelian groups are polynomially stable with respect to permutations, under the normalized Hamming metrics on the groups $\Sym(n)$. In particular, this means that there exists $D\geq1$ such that for $A,B\in \Sym(n)$, if $AB$ is $\delta$-close to $BA$, then $A$ and $B$ are $\epsilon$-close to a commuting pair of permutations, where $\epsilon\leq O\left(\delta^{1/D}\right)$. We also observe a property-testing reformulation of this result, yielding efficient testers for certain permutation properties.
\end{abstract}

\section{Introduction}

We begin with an informal presentation of a general framework that
originated in \cite{GR09}. Fix the normalized
Hamming metric over $\Sym\left(n\right)$ as a measure of proximity
between permutations. Let $\left(\sigma_{1},\dotsc,\sigma_{s}\right)$
be permutations in $\Sym\left(n\right)$. Suppose that $\left(\sigma_{1},\dotsc,\sigma_{s}\right)$
``almost'' satisfy a given finite system of equations $E$. Is it
necessarily the case that we can ``slightly'' modify each $\sigma_{i}$
to some $\tau_{i}\in\Sym\left(n\right)$, so that $\left(\tau_{1},\dotsc,\tau_{s}\right)$
satisfy $E$ \emph{exactly}? The answer depends on the system $E$.
In the case where $E$ is comprised of the single equation $XY=YX$,
it was shown in \cite{AP15} that the answer is positive. Simply put,
``almost commuting permutations are close to commuting permutations'',
where $\sigma_{1},\sigma_{2}$ are said to ``almost commute'' if
the permutations $\sigma_{1}\sigma_{2}$ and $\sigma_{2}\sigma_{1}$
are close. On the other hand, for $E$ that is the single equation
$XY^{2}=Y^{3}X$ the answer is negative \cite{GR09}.
In the language of \cite{GR09}, ``$XY=YX$ is stable, whereas
$XY^{2}=Y^{3}X$ is not''. 

The main novelty of this paper is that we provide the first results
in this area which are \emph{quantitative} and \emph{algorithmic}.
Let us illustrate this for the system $E=\left\{ XY=YX\right\} $.
We seek statements of the form ``if the distance between $\sigma_{1}\sigma_{2}$
and $\sigma_{2}\sigma_{1}$ is $\delta$, then there are two commuting
permutations $\tau_{1},\tau_{2}$ such that $\tau_{i}$ is at most
$f(\delta)$ away from $\sigma_{i}$ for $i=1,2$''. We stress that
$f=f_{E}$, called the \emph{stability rate} of $E$, depends solely
on $\delta$ but not $\sigma_{1},\sigma_{2}$ or $n$. Previously,
it was only shown that $\lim_{\delta\rightarrow0}f\left(\delta\right)=0$,
but nothing was said about the rate at which $f$ tends to zero with
$\delta$. This is precisely the kind of results that we achieve here.
By stating that our results are algorithmic, we mean that they provide
an explicit transformation of $\sigma_{1},\sigma_{2}$ into $\tau_{1},\tau_{2}$. 

More generally, one may consider the set of equations 
\[
E_{\text{comm}}^{d}=\left\{ X_{i}X_{j}=X_{j}X_{i}\mid1\leq i<j\leq d\right\} 
\]
and its stability rate function $f_{d}$. The paper \cite{AP15} shows
that $E_{\text{comm}}^{d}$ is stable, or, in terms of the framework of the present paper, that $\lim_{\delta\to0}f_{d}\left(\delta\right)=0$
for all $d\in\NN$. We give a stronger result by proving, in an algorithmic
manner, that 
\[
f_{d}\left(\delta\right)\le O\left(\delta^{\frac{1}{D}}\right)\,\,\text{,}
\]
where $D=D(d)\ge d$ is an explicitly given constant. We say that
``$E_{\text{comm}}^{d}$ is polynomially stable (with degree at most
$D$)''. Our proof employs novel elementary methods, and does not
rely on previous results concerning stability. Furthermore, we also
prove that 
\[
f_{d}\left(\delta\right)\ge\Omega\left(\delta^{\frac{1}{d}}\right)\,\,\text{.}
\]
It remains an open problem to close the gap between the bounds on
$f_{d}\left(\delta\right)$. 

Following \cite{GR09}, the basic observation
is that the stability of a set $E$ of equations is best studied in
terms of the group presented by taking $E$ as relators, and that
stability is a \emph{group invariant} \cite{AP15}. We show that the stability
rate is a group invariant as well, and then use properties of the
group $\ZZ^{d}$ to study the quantitative stability of the equations
$E_{\text{comm}}^{d}$. More specifically, our proof of the aforementioned
upper bound is based on a tiling procedure, in the spirit of Ornstein-Weiss
quasi-tiling for amenable groups \cite{ow} (in the setting of group actions which are not necessarily free). The fact that we work with abelian groups (rather than more general amenable groups) enables
a highly efficient tiling procedure via an original application of \emph{reduction
theory} of lattices in $\ZZ^{d}$. Our use of reduction theory is made through a theorem of Lagarias, Lenstra and Schnorr \cite{LLS90} regarding Korkin-Zolotarev bases \cite{Lag95}.

Finally, we examine a connection to the topic of \emph{property testing}
in computer science by rephrasing some of the above notions in terms of a certain {\em canonical testing algorithm} (this connection, in the non-quantitative setting, is the subject of \cite{BLtesting}). In particular, we show that our result yields an {\em efficient} algorithm to test whether a given tuple of permutations in $\Sym\left(n\right)$ satisfies the equations $E_\text{comm}^d$.

The rest of the introduction is organized as follows. Section \ref{subsec:ReviewOfStability}	
paraphrases definitions and results from the theory of stability in
permutations. In Section \ref{subsec:QuantitativeStabilityIntro},
we define the new, more delicate, notion of quantitative stability\emph{,
}and state our main theorems. Section \ref{subsec:PropertyTesting}
explores the connection between stability and property testing\emph{,}
and Section \ref{subsec:PreviousWork} discusses previous work\emph{.} 

\subsection{\label{subsec:ReviewOfStability}A review of stability in permutations}

For a set $S$, we write $\mathbb{F}_{S}$ for the free group based
on $S$. We also write $S^{-1}\subseteq\FF_{S}$ for the set of inverses
of elements of $S$ and $S^{\pm}=S\cup S^{-1}$. Fix a finite\emph{
set of variables} $S=\left\{ s_{1},\dotsc,s_{m}\right\} $. An \emph{$S$-assignment
}is a function $\Phi:S\to\Sym(n)$ for some positive integer $n$,
i.e., we assign permutations to the variables, all of which are in
$\Sym(n)$ for the same $n$. We naturally extend $\Phi$ to the domain
$\mathbb{F}_{S}$ via $\Phi(s^{-1})=\Phi(s)^{-1}$ if $s\in S$, and
$\Phi(w_{1}\cdots w_{t})=\Phi(w_{1})\cdots\Phi(w_{t})$ if $w_{1},\ldots,w_{t}\in S^{\pm}$.
In other words, an $S$-assignment $\Phi$ is also regarded as a homomorphism
from $\mathbb{F}_{S}$ to a finite symmetric group. Hence, $\Phi$
naturally describes a group action of $\mathbb{F}_{S}$ on the finite
set $[n]=\left\{1,\ldots,n\right\}$.

An $S$-\emph{equation-set} (or equation-set over the set $S$ of
variables) is a finite\footnote{See \cite{BLT18} for an extension of these definitions to infinite
equation-sets.} set $E\subseteq\FF_{S}$. An assignment $\Phi$ is said to be an
\emph{$E$-solution} if $\Phi(w)$ is the identity permutation $1=1_{\Sym(n)}$
for every $w\in E$. Equivalently, $\Phi$ is an $E$-solution if,
when regarded as a homomorphism $\mathbb{F}_{S}\to\Sym(n)$, it factors
through the quotient map $\pi:\mathbb{F}_{S}\to\mathbb{F}_{S}/\lla E\rra$.
Namely, if $\Phi=g\circ\pi$ for some homomorphism $g:\mathbb{F}_{S}/\lla E\rra\to\Sym(n)$.
Here and throughout, $\lla E\rra$ denotes the normal closure of a
subset $E$ of a group (the containing group will always be clear
from context).
\begin{example}
To relate these definitions to the canonical example of almost commuting
pairs of permutation, let $S=\{s_{1},s_{2}\}$ and $E=\{s_{1}s_{2}s_{1}^{-1}s_{2}^{-1}\}$,
and consider an $S$-assignment $\Phi$. Note that $\Phi(s_{1})$
and $\Phi(s_{2})$ commute if and only if $\Phi$ is an $E$-solution.
It may be helpful to keep this example in mind when reading up to
the end of Section \ref{subsec:QuantitativeStabilityIntro}. 
\end{example}
\begin{defn}
For $n\in\NN$, the \emph{normalized Hamming distance} $d_{n}$ on
$\Sym\left(n\right)$ is defined by $d_{n}\left(\sigma_{1},\sigma_{2}\right)=\frac{1}{n}\left|\left\{ x\in\left[n\right]\mid\sigma_{1}\left(x\right)\neq\sigma_{2}\left(x\right)\right\} \right|$
for $\sigma_{1},\sigma_{2}\in\Sym\left(n\right)$. Also, if $\Phi,\Psi$
are assignments over some finite set of variables $S$, we define
$d_{n}(\Phi,\Psi)=\sum_{s\in S}d_{n}(\Phi(s),\Psi(s))$. 
\end{defn}
Note that $d_{n}$ is bi-invariant, namely, $d_{n}(\tau\cdot\sigma_{1}\cdot\upsilon,\tau\cdot\sigma_{2}\cdot\upsilon)=d_{n}(\sigma_{1},\sigma_{2})$
for every $\tau,\upsilon\in\Sym(n)$. In particular, $d_{n}(\sigma_{1},\sigma_{2})=d_{n}(\sigma_{1}\cdot\sigma_{2}^{-1},1)$.

Given an $S$-equation-set $E$ and an $S$-assignment $\Phi$, one
may ask how close $\Phi$ is to being an $E$-solution. This notion
can be interpreted in two different ways. Locally, one can count the
points in $[n]=\{1,\ldots,n\}$ on which $\Phi$ violates the equation-set
$E$. Taking a global view, one measures how much work, that is, change
of permutation entries, is needed to ``correct'' $\Phi$ into an
$E$-solution. We proceed to define the \emph{local defect} and the
\emph{global defect} of an assignment.
\begin{defn}
Fix a finite set of variables $S$
and an $S$-equation-set $E$, and let $\Phi:S\to\Sym(n)$ be an $S$-assignment.
\begin{enumerate}
\item The \emph{local defect} of $\Phi$ with respect to $E$ is 
\[
L_{E}(\Phi)=\sum_{w\in E}d_{n}\left(\Phi(w),1\right)=\frac{1}{n}\sum_{w\in E}\left|\{x\in[n]\mid\Phi(w)(x)\ne x\}\right|.
\]
\item The \emph{global defect} of $\Phi$ with respect to $E$ is 
\[
G_{E}(\Phi)=\min\left|\left\{ d_{n}(\Phi,\Psi)\mid\Psi:S\to\Sym(n)\text{ is an \ensuremath{E}-solution}\right\} \right|.
\]
\end{enumerate}
\end{defn}
Note that $L_{E}(\Phi)$ and $G_{E}(\Phi)$ are in the ranges $[0,\left|E\right|]$
and $[0,|S|],$ respectively. Also, the conditions (i) $\Phi$ is
an $E$-solution, (ii) $L_{E}(\Phi)=0$ and (iii) $G_{E}(\Phi)=0$,
are all equivalent.

We wish to study the possible divergence between these two types of
defect. It is not hard to see (Lemma \ref{lem:local-defect-of-close-actions}) that $G_{E}(\Phi)\ge\Omega(L_{E}(\Phi))$.
Indeed, changing a single entry in one permutation $\Phi\left(s\right)$,
resolves, at best, the violation of $E$ at a constant number of points
$x\in\left[n\right]$. In this work we seek a reverse inequality,
i.e., an upper bound on the global defect in terms of the local defect.
\begin{defn}
Let $E$ be an $S$-equation-set. We define its \emph{stability rate
}$\SR_{E}:(0,\left|E\right|]\to[0,\infty)$ by \label{def:EquationSetStabilityRate}
\begin{align*}
\SR_{E}(\delta)=\sup\{G_{E}(\Phi)\mid&\Phi:S\to\Sym(n),\text{ where }n\in\NN \text{ and }L_{E}(\Phi)\le\delta\}\,\,\text{.}
\end{align*}
\end{defn}
Note that $\SR_{E}(\delta)$ is a monotone nondecreasing function. 
\begin{defn}\cite{GR09} If $\lim_{\delta\to0}\SR_{E}(\delta)=0$ then $E$
is said to be \emph{stable (in permutations)}. 
\end{defn}
The theory of equation-sets in permutations is related to group theory
via a key observation, given below as Proposition \ref{prop:StabilityIsGroupInvariant}.
\begin{defn}
Let $E_{1}$ and $E_{2}$ be equation-sets over the respective finite
sets of variables $S_{1}$ and $S_{2}$. If the groups $\mathbb{F}_{S_{1}}/\lla E_{1}\rra$
and $\mathbb{F}_{S_{2}}/\lla E_{2}\rra$ are isomorphic then $E_{1}$
and $E_{2}$ are called \emph{equivalent}.
\end{defn}
\begin{prop}{\cite{AP15}}
Let $E_{1}$ and $E_{2}$ be equivalent equation-sets.
Then $E_{1}$ is stable if and only if $E_{2}$ is stable. \label{prop:StabilityIsGroupInvariant}
\end{prop}
Proposition \ref{prop:StabilityIsGroupInvariant} allows us to regard
stability as a group invariant:
\begin{defn}\cite{AP15}
Let $\Gamma$ be a finitely presented group. That is, $\Gamma\cong\mathbb{\mathbb{F}}_{S}/\lla E\rra$
for some finite sets $S$ and $E\subseteq\mathbb{\mathbb{F}}_{S}$.
Then, $\Gamma$ is called \emph{stable} if $E$ is a stable equation-set.
\end{defn}
This definition enables us to apply the properties of a group in order
to study the stability of its defining set of equations. See Section
\ref{subsec:PreviousWork} for previous results obtained by this method.
The following result is of particular interest to us in the context
of this work.
\begin{thm}
\cite{AP15} Every finitely generated abelian group is stable.\label{thm:AbelianGroupsAreStable}
\end{thm}
The proof of Theorem \ref{thm:AbelianGroupsAreStable} in \cite{AP15}
is not algorithmic. That is, it does not describe an explicit transformation that maps an assignment with small local defect to a nearby $E$-solution.
With some work, it is possible in principle to extract such an algorithm
from the idea presented in \cite{AP15}, by looking further into the
proof a theorem of Elek and Szabo \cite{ES11} used there. However,
such an approach would perform poorly in the quantitative sense described
below.

\subsection{Quantitative stability\label{subsec:QuantitativeStabilityIntro}}

We turn to discuss the main new definitions introduced in this work,
which deals with quantitative stability. The basis for quantitative
stability is a stronger version of Proposition \ref{prop:StabilityIsGroupInvariant},
given in Proposition \ref{prop:StabilityRateIsAGroupInvariant} below.
Namely, we show that stability can be refined, yet remain a group
invariant, by considering the \emph{rate} at which $\SR_{E}$ converges to
$0$ as $\delta\to0$. We now make this claim precise.
\begin{defn}
Let $F_{1},F_{2}:(0,\left|E\right|]\to[0,\infty)$ be monotone nondecreasing
functions. Write $F_{1}\sim F_{2}$ if $F_{1}(\delta)\le F_{2}(C\delta)+C\delta$
and $F_{2}(\delta)\le F_{1}(C\delta)+C\delta$ for some $C>0$. Let
$[F_{1}]$ denote the class of $F_{1}$ with regard to this equivalence
relation. \label{def:FunctionEquivalence}
\end{defn}
The reason for the introduction of the $C\delta$ summand in the above
definition is elaborated upon in Remark \ref{rem:whats-this-Cdelta}.
\begin{prop}
For every equivalent pair of equation-sets $E_{1}$ and $E_{2}$ we
have $\SR_{E_{1}}\sim\SR_{E_{2}}$. \label{prop:StabilityRateIsAGroupInvariant}
\end{prop}
We prove Proposition \ref{prop:StabilityRateIsAGroupInvariant} in
Section \ref{sec:StabilityRateIsAGroupInvariant}. This proposition allows us
to define the \emph{stability rate} of a group through the stability rate of a corresponding equation-set (Definition \ref{def:EquationSetStabilityRate}).
\begin{defn}
Let $E$ be an $S$-equation-set. The \emph{stability rate }$\SR_{\Gamma}$
of the group $\Gamma=\mathbb{\mathbb{F}}_{S}/\lla E\rra$ is the equivalence class
$[\SR_{E}]$. \label{def:GroupStabilityRate}
\end{defn}
Our goal in this paper is to show that for an abelian group $\Gamma$,
not only does the stability rate converge to $0$ as $\delta\to0$,
but this convergence is \emph{fast}. This claim can be made precise
as follows.
\begin{defn}
\label{def:degree}Let $F:(0,\left|E\right|]\to[0,\infty)$. We define
the \emph{degree} of $F$ by
\begin{equation}
\deg(F)=\inf\left\{ k\ge1\mid F(\delta)\le\underset{\delta\to0}{O}\left(\delta^{\frac{1}{k}}\right)\right\} .\label{eq:FunctionDegree}
\end{equation}
It is possible that $\deg(F)=\infty$. Also, let $\deg([F])=\deg(F)$.
\end{defn}
\begin{rem}
Note that $\deg([F])$ is well-defined. Indeed, if $F_{1}\sim F_{2}$
and $F_{1}(\delta)\le\underset{\delta\to0}{O}(\delta^{\frac{1}{m}})$
for $m\ge1$, then $F_{2}(\delta)\le F_{1}(C\delta)+C\delta\le\underset{\delta\to0}{O}(\delta^{\frac{1}{m}})$
as well.
\end{rem}
\begin{defn}
In the notation of Definition \ref{def:GroupStabilityRate}, the \emph{degree
of polynomial stability} of both $E$ and $\Gamma$ is defined to
be $D=\deg([\SR_{E}])=\deg\left(\SR_{\Gamma}\right)$. If $D<\infty$,
we say that $E$ and $\Gamma$ are \emph{polynomially stable}. 
\end{defn}
We can now state our main theorem.
\begin{thm}
Every finitely generated abelian group is polynomially stable. \label{thm:AbelianGroupsArePolynomiallyStable}
\end{thm}
Notably, Theorem \ref{thm:AbelianGroupsArePolynomiallyStable} applies
to \emph{commutator equation-sets}, namely equation-sets of the form
\begin{equation}
E_{\text{comm}}^{d}=\left\{ s_{i}s_{j}s_{i}^{-1}s_{j}^{-1}\mid1\le i<j\le d\right\} \label{eq:E_comm}
\end{equation}
over the variables $\{s_{1},\ldots,s_{d}\}$ where $d\ge0$. Indeed,
note that $\mathbb{\mathbb{F}}_{d}/\lla E_{\text{comm}}^{d}\rra\cong\mathbb{Z}^{d}$,
so $E_{\text{comm}}^{d}$ is polynomially stable. In particular, $d=2$
yields the example of commuting \emph{pairs }of permutations from
the beginning of the introduction. 

We complement Theorem \ref{thm:AbelianGroupsArePolynomiallyStable}
with a lower bound on the degree of polynomial stability.
\begin{thm}
For all $d\in\NN$, the group $\mathbb{Z}^{d}$ has degree of polynomial
stability at least $d$. \label{thm:DegreeLowerBound}
\end{thm}
We note that, having shown that an abelian group $\Gamma$ is polynomially
stable, we are left with the more delicate question of its \emph{degree
of polynomial stability}. Restricting attention to a free abelian
group $\ZZ^{d}$, the proof of Theorem \ref{thm:AbelianGroupsArePolynomiallyStable}
yields an upper bound on this degree that grows exponentially in $d$ (see Equation (\ref{eq:DBound})).
A remaining open problem is to close the large gap between this upper
bound, and the lower bound of Theorem \ref{thm:DegreeLowerBound}.

We diverge from the practice of providing an outline of the proofs
in the introduction since such an outline requires a geometric formulation
of stability, which is developed in Section \ref{sec:stability-and-graphs}.
Theorem \ref{thm:AbelianGroupsArePolynomiallyStable} is proved in Section \ref{sec:AbelianGroupsArePolynomiallyStable}, and its proof is outlined in Section \ref{subsec:ProofPlan}. A reader who is interested
in the quickest route to understanding this outline may skip Section
\ref{sec:StabilityRateIsAGroupInvariant}. Theorem \ref{thm:DegreeLowerBound}
is the subject of Section \ref{sec:LowerBound}. 

\subsection{Application to property testing \label{subsec:PropertyTesting}}

The notions of stability and quantitative stability have a natural
interpretation in terms of \emph{property testing} (For more on this
connection see \cite{BLtesting}. For background on property testing see
\cite{Gol17,Ron10}). 
\begin{defn}
\label{def:TestingAlgorithm}Fix a nonempty equation-set $E$ over
the finite set of variables $S$. A \emph{tester }for $E$ is an algorithm
which takes an assignment $\Phi:S\to\Sym(n)$ as input, queries the
permutations $\left\{ \Phi\left(s\right)\right\} _{s\in S}$ at a
constant (in particular, independent of $n$) number of entries among
$1,\dotsc,n$, and decides whether $\Phi$ is an $E$-solution. The
algorithm must satisfy the following:
\begin{enumerate}[label=(\Roman*)]
\item If $\Phi$ is an $E$-solution, the algorithm always \emph{accepts}.
\item For some fixed function $\delta:\left(0,\infty\right)\times\NN\rightarrow(0,1]$,
if $\Phi$ is not an $E$-solution, then the algorithm \emph{rejects}
with probability at least $\delta\left(G_{E}\left(\Phi\right),n\right).$
This function $\delta$ is called the \emph{detection probability
}of the tester.\label{enu:TestingAlgorithmDetectionProb}
\end{enumerate}
The precise term in the literature for this notion is an ``adaptive
proximity-oblivious tester with one-sided error and constant query
complexity'' (see, \cite{Gol17} Definition 1.7).
\end{defn}
The \emph{canonical tester }${\mathcal{N}}_{E}$ for $E$ samples $x\in[n]$
and $w\in E$ uniformly at random from their respective sets. It accepts
if $\Phi(w)(x)=x$ and rejects otherwise. Clearly, if $\Phi$ is an
$E$-solution then ${\mathcal{N}}_{E}$ always accepts. If $\Phi$ is not
an $E$-solution, then ${\mathcal{N}}_{E}$ rejects with probability $\frac{1}{\left|E\right|}\cdot L_{E}(\Phi)$.
Generally, it is desirable for the detection probability function
$\delta\left(\epsilon,n\right)$ to be \emph{uniform}, i.e., depend
only on $\epsilon=G_{E}\left(\Phi\right)$ and not on $n$. In order to show that this is the case for $\mathcal{N}_E$, we must consider the function
\begin{align*}
\delta(\epsilon) & =\inf\left\{ \frac{1}{\left|E\right|}\cdot L_{E}(\Phi)\mid\text{\ensuremath{\Phi}}\text{ is an }S\text{-assignment and }G_{E}(\Phi)\ge\epsilon\right\} \\
 & =\inf\left\{ \delta\mid\SR_{E}(\delta)\ge\epsilon\right\} \,\,\text{.}
\end{align*}
The tester $\mathcal{N}_E$ satisfies Condition \ref{def:TestingAlgorithm}\ref{enu:TestingAlgorithmDetectionProb} if and only if $\delta\left(\epsilon\right)$ is positive for every $\epsilon>0$.
This is equivalent to the condition that $\lim_{\delta\to0}\SR_{E}(\delta)=0$.
Hence, ${\mathcal{N}}_{E}$ admits a uniform detection probability function
if and only if $E$ is stable.

Furthermore, the smaller the stability rate of $E$, the larger $\delta$
is. In particular, if $E$ is polynomially stable with degree of polynomial
stability $D$, then $\delta(\epsilon)$ is bounded from below by
$\Omega\left(\epsilon^{D}\right)$. Therefore, due to Theorem \ref{thm:AbelianGroupsArePolynomiallyStable},
the canonical tester ${\mathcal{N}}_{E}$ has detection probability polynomial
in $\epsilon$ and uniform in $n$, whenever $\mathbb{\mathbb{F}}_{S}/\lla E\rra$
is abelian. 

In a somewhat weaker formulation of property testing (see \cite{Gol17},
Definition 1.6), the tester is only required to distinguish between
the cases $G_{E}(\Phi)=0$ and $G_{E}(\Phi)>\epsilon$, where $\epsilon>0$
is given as input. In this formulation, the detection probability
is required to be larger than some constant, say $\frac{1}{2}$, and
one seeks to minimize the number of queries. Note that, for every
$E$, the canonical tester ${\mathcal{N}}_{E}$ can be used to build a tester
$\tilde{\mathcal{N}}_{E}$, satisfying this weaker formulation: Given
$\epsilon>0$, and an input $\Phi$, the tester $\tilde{\mathcal{N}}_{E}$
runs ${\mathcal{N}}_{E}$ on $\Phi$ repeatedly for $\log_{1-\delta(\epsilon,n)}\frac{1}{2}=\Theta(\frac{1}{\delta(\epsilon,n)})$
independent iterations, and accepts only if ${\mathcal{N}}_E$ accepts in all iterations. Hence, the resulting tester $\tilde{\mathcal{N}}_{E}$ performs
$\Theta\left(\frac{1}{\delta(\epsilon,n)}\right)$ queries. In particular,
Theorem \ref{thm:AbelianGroupsArePolynomiallyStable} shows that for
$E$ such that $\mathbb{\mathbb{F}}_{S}/\lla E\rra$ is abelian, the
tester $\tilde{\mathcal{N}}_{E}$ is \emph{efficient}. Namely, it has
constant detection probability, and its number of queries is polynomial
in $\frac{1}{\epsilon}$ and does not depend on $n$. No such tester
was previously known.

\subsection{Previous work\label{subsec:PreviousWork}}

The general question of whether almost-solutions are close to solutions,
in various contexts, was suggested by Ulam (see \cite{Ula60}, Chapter
VI). The most studied question of this sort is whether almost-commuting
\emph{matrices} are close to commuting matrices, and the answer depends
on the chosen matrix norm and on which kind of matrices is considered
(e.g. self-adjoint, unitary, etc.). See the introduction of \cite{AP15}
for a short survey, and \cite{dglt,HS1,HS2,ESS} for some newer works.
In this context, some quantitative results are already known \cite{hastings,FK}.
The question of (non-quantitative) stability in \emph{permutations},
under the normalized Hamming metric, was initiated in \cite{GR09}
and developed further in \cite{AP15}. The former paper proves that finite groups are stable
(see our Proposition \ref{prop:finite-groups} for a quantitative
version), and the latter proves that abelian groups are stable. Both
papers provide examples of non-stable groups as well, and relate stability
in permutations to the notion of sofic groups. These results are generalized
in \cite{BLT18}, which provides a characterization of stability in
permutations, among amenable groups, in terms of their invariant random
subgroups. On the other side of the spectrum (compared to amenable
groups), \cite{BLpropertyT} proves that infinite groups with Property
$\text{\ensuremath{\left(\text{T}\right)}}$ are never stable in permutations,
and consequently suggests some weaker forms of stability.

\section{The stability rate is a group invariant \label{sec:StabilityRateIsAGroupInvariant}}

In this section we prove Proposition \ref{prop:StabilityRateIsAGroupInvariant}.
We start with Lemma \ref{lem:d_n(Phi(w),Psi(w))}, which formalizes
the claim that if two $S$-assignments $\Phi$ and $\Psi$ are close
to each other, then so are the permutations $\Phi(w)$ and $\Psi(w)$,
provided that $w\in\mathbb{F}_{S}$ is a short word. It is helpful
to consider $w=w_{1}\cdots w_{t}$ ($w_{i}\in S^{\pm})$ as a sequence
of directions, namely, $\Phi(w)(x)$ is the endpoint of the path that
starts at $x$, moves in the direction $w_{t}$ to $\Phi(w_{t})(x)$,
continues to $\Phi(w_{t-1}w_{t})(x)$, and so on. The immediate observation
behind Lemma \ref{lem:d_n(Phi(w),Psi(w))}, is that as long as this
path moves only along ``nice'' edges, i.e., edges on which $\Phi$
and $\Psi$ agree, it is guaranteed that $\Phi(w)(x)=\Psi(w)(x)$.
A similar idea is then used to prove Lemma \ref{lem:d_n(Phi(w),1)}
as well.
\begin{lem}
\label{lem:d_n(Phi(w),Psi(w))}Let $E$ be an $S$-equation-set, let
$\Phi,\Psi:S\to\Sym(n)$ be $S$-assignments, and let $w=w_{1}\cdots w_{t}\in\mathbb{F}_{S}$
where $w_{i}\in S^{\pm}$. Then, $d_{n}(\Phi(w),\Psi(w))\le t\cdot d_{n}(\Phi,\Psi)$.
\end{lem}
\begin{proof}
Let $\bar{w}_{i}$ denote the suffix $w_{i}\cdots w_{t}$. By
the bi-invariance of $d_{n}$ and the triangle inequality,
\begin{align*}
d_{n}\left(\Phi(w),\Psi(w)\right) & =d_{n}\left((\Psi(w))^{-1}\cdot\Phi(w),1\right)\\
 & \le\sum_{i=1}^{t}d_{n}\left((\Psi(\bar{w}_{i}))^{-1}\cdot\Phi(\bar{w}_{i}),\Psi(\bar{w}_{i+1}))^{-1}\cdot\Phi(\bar{w}_{i+1})\right)\\
 & =\sum_{i=1}^{t}d_{n}\left((\Psi(w_{i}))^{-1}\cdot\Phi(w_{i}),1\right)=\sum_{i=1}^{t}d_{n}\left(\Phi(w_{i}),\Psi(w_{i})\right).
\end{align*}
Note that each term of the right hand side is at most $d_{n}(\Phi,\Psi)$.
Here, if $w_{i}$ is the inverse of a generator, we used the fact
that $d_{n}(\Phi(w_{i}),\Psi(w_{i}))=d_{n}\left(\Phi\left(w_{i}^{-1}\right),\Psi\left(w_{i}^{-1}\right)\right)$.
The lemma follows.
\end{proof}
\begin{lem}
\label{lem:d_n(Phi(w),1)}Fix an $S$-equation-set $E\subseteq \FF_S$, and a word
$w\in\lla E\rra$, written as $w=u_{1}q_{1}u_{1}^{-1}u_{2}q_{2}u_{2}^{-1}\cdots u_{t}q_{t}u_{t}^{-1}$,
where $u_{i}\in\mathbb{F}_{S}$ and each $q_{i}$ is an element
of $E$ or its inverse. Then, for every $S$-assignment $\Phi:S\rightarrow\Sym\left(n\right)$,
\[
d_{n}(\Phi(w),1)\le L_{E}(\Phi)\cdot t\,\,\text{.}
\]
\end{lem}
\begin{proof}
We have
\begin{align*}
d_{n}(\Phi(w),1) & =d_{n}\left(\Phi\left(u_{1}q_{1}u_{1}^{-1}u_{2}q_{2}u_{2}^{-1}\cdots u_{t}q_{t}u_{t}^{-1}\right),1\right)\\
 & \leq\sum_{i=1}^{t}d_{n}\left(\Phi\left(u_{i}q_{i}u_{i}^{-1}\cdots u_{t}q_{t}u_{t}^{-1}\right),\Phi\left(u_{i+1}q_{i+1}u_{i+1}^{-1}\cdots u_{t}q_{t}u_{t}^{-1}\right)\right)\\
 & =\sum_{i=1}^{t}d_{n}\left(\Phi\left(u_{i}q_{i}u_{i}^{-1}\right),1\right)\\
 & =\sum_{i=1}^{t}d_{n}\left(\Phi\left(q_{i}\right),1\right)\\
 & \leq\sum_{i=1}^{t}\sum_{w\in E}d_{n}\left(\Phi\left(w\right),1\right)\\
 & =t\cdot L_{E}\left(\Phi\right)\,\,\text{.}
\end{align*}
\end{proof}
The following corollary of Lemma \ref{lem:d_n(Phi(w),1)} will also
be useful.
\begin{lem}
\label{lem:d_n(Phi lambda, Phi)}Fix an $S$-equation-set $E$ and
a homomorphism $\lambda:\mathbb{F}_{S}\to\mathbb{F}_{S}$ such that
$w$ and $\lambda(w)$ belong to the same coset in $\mathbb{F}_{S}/\lla E\rra$
for each $w\in\FF_{S}$. Then, there exists a positive $c=c(S,E,\lambda)$
such that for every $S$-assignment $\Phi:S\rightarrow\Sym\left(n\right)$,
we have $d_{n}(\Phi,\Phi\circ\lambda)\le c\cdot L_{E}(\Phi)$, where
we regard $\Phi\circ\lambda$ as an $S$-assignment by restricting
its domain from $\FF_{S}$ to $S$.
\end{lem}
\begin{proof}
Note that
\[
d_{n}(\Phi,\Phi\circ\lambda)=\sum_{s\in S}d_{n}(\Phi(s),\Phi(\lambda(s)))=\sum_{s\in S}d_{n}(1,\Phi(\lambda(s)\cdot s^{-1})\,\,\text{.}
\]
Since $\lambda(s)\cdot s^{-1}\in\lla E\rra$, it follows from Lemma
\ref{lem:d_n(Phi(w),1)} that the $s-$term of this sum is at most
$O\left(L_{E}(\Phi)\right)$, where the implied constant depends on
$s,$ $E$ and $\lambda$. Hence, $d_{n}(\Phi,\Phi\circ\lambda)\le O\left(L_{E}(\Phi)\right)$.
\end{proof}
We turn to prove Proposition \ref{prop:StabilityRateIsAGroupInvariant}.
\begin{proof}
[Proof of Proposition \ref{prop:StabilityRateIsAGroupInvariant}]In
the course of the proof, when a function whose domain is $\FF_{S_{1}}$or
$\FF_{S_{2}}$ appears where a function whose domain is $S_{1}$ or
$S_{2}$ is expected, the function should be regarded as its respective
restriction (for example, when measuring distances between assignments).

Let $E_{1}$ and $E_{2}$ be equivalent equation-sets over the respective
finite sets of variables $S_{1}$ and $S_{2}$. By symmetry, it is
enough to prove that $\SR_{E_{1}}(\delta)\le\SR_{E_{2}}(C\delta)+C\delta$
for some $C=C(E_{1},E_{2})$. Equivalently, we need to show that $G_{E_{1}}(\Phi_{1})\le\SR_{E_{2}}(C\delta)+C\delta$
for any given $\delta>0$ and $S_{1}$-assignment $\Phi_{1}:S_{1}\to\Sym(n)$
with $L_{E_{1}}(\Phi_{1})\le\delta$. 

Our strategy is to ``translate'' the $S_{1}$-assignment $\Phi_{1}$
into an $S_{2}$-assignment $\Phi_{2}$, find an $E_{2}$-solution $\Psi_{2}$
which is close to $\Phi_{2}$, and pull back $\Psi_{2}$ to an $E_{1}$-solution
$\Psi_{1}$. We apply Lemma \ref{lem:d_n(Phi(w),1)} to bound $L_{E_{2}}(\Phi_{2})$,
and then use Lemmas \ref{lem:d_n(Phi(w),Psi(w))} and \ref{lem:d_n(Phi lambda, Phi)}
to control the distance between $\Phi_{1}$ and $\Psi_{1}$, yielding
an upper bound on $G_{E_{1}}(\Phi_{1})$.

We define the machinery needed to map $S_{1}$-assignments to $S_{2}$-assignments
and vice versa. Since $E_{1}$ and $E_{2}$ are equivalent equation-sets,
there exists a group isomorphism $\theta:\mathbb{F}_{S_{1}}/\lla E_{1}\rra\to\mathbb{F}_{S_{2}}/\lla E_{2}\rra$.
Let $\pi_{1}$ denote the quotient map $\mathbb{F}_{S_{1}}\to\mathbb{F}_{S_{1}}/\lla E_{1}\rra$,
and likewise for $\pi_{2}$. Fix a homomorphism $\lambda_{2}:\mathbb{F}_{S_{2}}\to\mathbb{F}_{S_{1}}$
such that $\pi_{2}=\theta\circ\pi_{1}\circ\lambda_{2}$. In other
words, we choose $\lambda_{2}$ so that the composition of the following
chain of morphisms equals $\pi_{2}$:
\[
\mathbb{F}_{S_{2}}\xrightarrow{\lambda_{2}}\mathbb{F}_{S_{1}}\xrightarrow{\pi_{1}}\mathbb{F}_{S_{1}}/\lla E_{1}\rra\xrightarrow{\theta}\mathbb{F}_{S_{2}}/\lla E_{2}\rra\,\,\text{.}
\]
Note that such $\lambda_{2}$ exists since $\theta\circ\pi_{1}$ is
surjective. Similarly, fix a homomorphism $\lambda_{1}:\mathbb{F}_{S_{1}}\to\mathbb{F}_{S_{2}}$
satisfying $\pi_{1}=\theta^{-1}\circ\pi_{2}\circ\lambda_{1}$. From
now on, in our use of asymptotic $O\left(\cdot\right)$-notation,
we allow the implied constant to depend on $S_{1}$, $S_{2}$, $E_{1}$
and $E_{2}$. Since $\lambda_{2}$ and $\lambda_{1}$ have been fixed
solely in terms of these four objects, the implied
constant is allowed to depend on them as well.

Let $\delta>0$ and $n\in\NN$. Let $\Phi_{1}:S_{1}\rightarrow\Sym\left(n\right)$
be an $S_{1}$-assignment such that $L_{E_{1}}\left(\Phi_{1}\right)\leq\delta$.
Define the $S_{2}$-assignment $\Phi_{2}=\Phi_{1}\circ\lambda_{2}:S_{2}\to\Sym(n)$.
We seek to bound its local defect:
\[
L_{E_{2}}(\Phi_{2})=\sum_{w\in E_{2}}d_{n}(\Phi_{2}(w),1)=\sum_{w\in E_{2}}d_{n}(\Phi_{1}(\lambda_{2}(w)),1)\,\,\text{.}
\]
Since $\lambda_{2}(w)\in\lla E_{1}\rra$ for every $w\in E_{2}$,
it follows from Lemma \ref{lem:d_n(Phi(w),1)}, applied to $\Phi_{1}$
and the word $\lambda_{2}\left(w\right)$, that the $w$-term in the above
sum is at most $O(L_{E_{1}}(\Phi_{1}))$. Consequently, 
\[
L_{E_{2}}(\Phi_{2})\le C_{1}\cdot L_{E_{1}}(\Phi_{1})\le C_{1}\cdot\delta
\]
for some positive $C_{1}=C_{1}(E_{1},E_{2})$. Hence, there is an
$E_{2}$-solution $\Psi_{2}:S_2\to\Sym(n)$ with 
\[
d_{n}(\Phi_{2},\Psi_{2})\le\SR_{E_{2}}(L_{E_{2}}(\Phi_{2}))\le\SR_{E_{2}}(C_{1}\cdot\delta)\,\,\text{.}
\]

Let $\Psi_{1}$ be the $S_{1}$-assignment $\Psi_{2}\circ\lambda_{1}:S_{1}\to\Sym(n)$.
Note that $\Psi_{1}$ is an $E_{1}$-solution. Thus, $G_{E_{1}}(\Phi_{1})\le d_{n}(\Phi_{1},\Psi_{1})$.
We proceed to bound this distance. By the triangle inequality, 
\begin{align*}
d_{n}(\Phi_{1},\Psi_{1}) & \le d_{n}(\Phi_{1},\Phi_{2}\circ\lambda_{1})+d_{n}(\Phi_{2}\circ\lambda_{1},\Psi_{1})\\
 & =d_{n}(\Phi_{1},\Phi_{1}\circ\lambda_{2}\circ\lambda_{1})+d_{n}(\Phi_{2}\circ\lambda_{1},\Psi_{2}\circ\lambda_{1}).
\end{align*}
We turn to bound both terms of the right hand side. For the first
term, note that 
\[
\pi_{1}=\theta^{-1}\circ\pi_{2}\circ\lambda_{1}=\theta^{-1}\circ\theta\circ\pi_{1}\circ\lambda_{2}\circ\lambda_{1}=\pi_{1}\circ\lambda_{2}\circ\lambda_{1}\,\,,
\]
and so $\lambda_{2}\circ\lambda_{1}$ satisfies the requirements of
Lemma \ref{lem:d_n(Phi lambda, Phi)}. Hence, due to this lemma, 
\[
d_{n}(\Phi_{1},\Phi_{1}\circ\lambda_{2}\circ\lambda_{1})\le O\left(L_{E_{1}}(\Phi_{1})\right)\le O(\delta)\,\,\text{.}
\]
 Turning to the second term, 
\[
d_{n}(\Phi_{2}\circ\lambda_{1},\Psi_{2}\circ\lambda_{1})=\sum_{s_{2}\in S_{2}}d_{n}(\Phi_{2}(\lambda_{1}(s_{2})),\Psi_{2}(\lambda_{1}(s_{2}))).
\]
By Lemma \ref{lem:d_n(Phi(w),Psi(w))}, applied to $\Phi_{2}$, $\Psi_{2}$
and the word $\lambda_{1}\left(s_{2}\right)$, the $s_{2}$-term of this sum
is upper bounded by $O(d_{n}(\Phi_{2},\Psi_{2}))$, and so 
\[
d_{n}(\Phi_{2}\circ\lambda_{1},\Psi_{2}\circ\lambda_{1})\le O(d_{n}(\Phi_{2},\Psi_{2}))\le O\left(\SR_{E_{2}}(C_{1}\cdot\delta)\right)\,\,\text{.}
\]
 We conclude that 

\[
G_{E}(\Phi_{1})\le d_{n}(\Phi_{1},\Psi_{1})\le O\left(\SR_{E_{2}}(C_{1}\cdot\delta)+\delta\right)\,\,\text{.}
\]
\end{proof}
\begin{rem} 
\label{rem:whats-this-Cdelta}In Proposition \ref{prop:SR-at-least linear},
we show that $\SR_{E}(\delta)\ge\Omega(\delta)$ for every equation-set $E$ which is not empty and not $\left\{1\right\}$. When $E$ is $\emptyset$ or $\left\{1\right\},$ however, it is clear that $\SR_{E}\equiv0$.
This is worth noting, since the free group $\mathbb{\mathbb{F}}_{S}$
can be defined by either a trivial equation-set over $S$, or by a
certain nontrivial equation-set over some larger finite set of variables.
Two nuances in Definitions \ref{def:FunctionEquivalence} and \ref{def:degree}
ensure that the stability rate and degree of $\mathbb{\mathbb{F}}_{S}$
are well-defined despite this phenomenon. The first is the addition
of the term $C\delta$ to the inequalities in Definition \ref{def:FunctionEquivalence},
and the second is the restriction $k\ge1$ in Equation (\ref{eq:FunctionDegree})
in Definition \ref{def:degree}.
\end{rem}

\section{Stability and graphs of actions\label{sec:stability-and-graphs}}

This section reformulates the notions of stability and stability rate
in terms of group actions, and provides basic tools arising from this
point of view. Before we begin, a small clarification regarding terminology
is in order: When a group homomorphism $\theta:\Lambda_{2}\rightarrow\Lambda_{1}$
is fixed and understood from the context, we regard any given $\Lambda_{1}$-set
$X$ as a $\Lambda_{2}$-set as well via $\theta$, i.e., for $g_{2}\in\Lambda_{2}$
and $x\in X$, we let $g_{2}\cdot x=\theta\left(g_{2}\right)\cdot x$.
In most cases in the sequel, $\Lambda_{1}$ is a group generated by
a finite set $S$, $\Lambda_{2}=\FF_{S}$ is a free group on $S$,
and $\theta$ is the natural quotient map $\FF_{S}\rightarrow\Lambda_{1}$.
So, when a $\Lambda_{1}$-set $X$ appears where an $\FF_{S}$-set
is expected, we treat $X$ as an $\FF_{S}$-set in this manner. In some
other cases, the role of $\Lambda_{2}$ is taken by a free abelian
group, rather than a free group.

\subsection{\label{subsec:stability-in-terms-of-actions}Stability in terms of
group actions}

Throughout Section \ref{subsec:stability-in-terms-of-actions}, we
fix an equation-set $E$ over the finite set of variables $S$, and
denote $\Gamma=\mathbb{F}_{S}/\lla E\rra$. As mentioned in Section
\ref{subsec:ReviewOfStability}, an $S$-assignment $\Phi:S\to\Sym(n)$
can also be regarded as a group action of $\mathbb{F}_{S}$ on $[n]$.
We write $\mathbb{F}_{S}(\Phi)$ for the $\mathbb{F}_{S}$-set whose
set of points is $[n]$, with the group action given by $s\cdot x=\Phi(s)(x)$.
We now expand upon this view, rephrasing the definition of local and
global defect in terms of group actions, along the same lines as in
Section 3.2 of \cite{BLT18}. This will enable us to prove our main
theorems using a geometric approach, focusing on the geometry of an
edge-labeled graph representing $\mathbb{F}_{S}(\Phi)$.
\begin{defn}
\label{def:LocalDefectOfAction}Let $X$ be a finite $\FF_{S}$-set.
The \emph{local defect} of $X$ with respect to $E$ is
\[
L_{E}\left(X\right)=\frac{1}{\left|X\right|}\cdot|\{(x,w)\in X\times E\mid w\cdot x\ne x\}|\,\,\text{.}
\]
\end{defn}
It follows directly from the definitions that $L_{E}\left(\Phi\right)=L_{E}\left(\FF_{S}\left(\Phi\right)\right)$
for an assignment $\Phi:S\rightarrow\Sym\left(n\right)$.

We turn to the global defect. Let us first characterize $E$-solutions
through actions. It is not hard to see that $\Phi:S\rightarrow\Sym\left(n\right)$
is an $E$-solution if and only if the action $\FF_{S}\curvearrowright\FF_{S}\left(\Phi\right)$
factors through the group $\Gamma$, that is, if every two elements
$u,v\in\mathbb{F}_{S}$, belonging to the same coset in $\mathbb{F}_{S}/\lla E\rra$,
act on $[n]$ in the same manner. In this case, $\Gamma$ itself acts
on $\FF_{S}\left(\Phi\right)$.

We proceed to define a metric of similarity between $\mathbb{F}_{S}$-sets.
\begin{defn} \label{def:DistanceBetweenFSets}
Let $X$ and $Y$ be finite $\FF_{S}$-sets, where $\left|X\right|=\left|Y\right|=n$.
For a function $f:X\rightarrow Y$, define
\[
\|f\|_{S}=\sum_{s\in S}\frac{1}{n}\cdot\left|\left\{ \left(s,x\right)\in S\times X\mid f\left(s\cdot x\right)\neq s\cdot f\left(x\right)\right\} \right|\,\,\text{.}
\]
Furthermore, define
\[
d_{S}\left(X,Y\right)=\min\left\{ \|f\|_{S}\mid f:X\rightarrow Y\,\,\text{is a bijection}\right\} \,\,\text{.}
\]
\end{defn}
Note that $d_{S}(X,Y)=0$, if and only if $X$ and $Y$ are isomorphic
as $\mathbb{F}_{S}$-sets. Also, for $S$-assignments $\Phi,\Psi:S\rightarrow\Sym\left(n\right)$
we have $d_{n}\left(\Phi,\Psi\right)=\|\id_{\Phi,\Psi}\|_{S}$, where
$\id_{\Phi,\Psi}:\FF_{S}\left(\Phi\right)\rightarrow\FF_{S}\left(\Psi\right)$
is the identity map $\left[n\right]\rightarrow\left[n\right]$. 

We use this metric to express the notion of global defect.
\begin{defn}
\label{def:global-defect-of-action}Let $X$ be a finite $\FF_{S}$-set.
The \emph{global defect} of $X$ with respect to $E$ is
\[
G_{E}\left(X\right)=\min\{d_{S}(X,Y)\mid\text{\ensuremath{Y} is a \ensuremath{\Gamma}-set and \ensuremath{\left|Y\right|=\left|X\right|}}\}\,\,\text{.}
\]
\end{defn}
\begin{prop}
Let $\Phi:S\to\Sym(n)$ be an $S$-assignment. Then,
\[
G_{E}(\Phi)=G_{E}\left(\FF_{S}\left(\Phi\right)\right)\,\,\text{.}
\]
\end{prop}
\begin{proof}
Let $\Psi:S\to\Sym(n)$ be an $E$-solution which minimizes $d_{n}\left(\Phi,\Psi\right)$.
Let $Y$ be a $\Gamma$-set, $\left|Y\right|=n$, which minimizes
$d_{S}\left(\FF_{S}\left(\Phi\right),Y\right)$. We need to show that
$d_{n}\left(\Phi,\Psi\right)=d_{S}\left(\FF\left(\Phi\right),Y\right)$.
Indeed, on one hand,
\[
d_{S}\left(\FF_{S}\left(\Phi\right),Y\right)\leq d_{S}\left(\FF_{S}\left(\Phi\right),\FF_{S}\left(\Psi\right)\right)\leq\|\id_{\Phi,\Psi}\|_{S}=d_{n}\left(\Phi,\Psi\right)\,\,,
\]
where the first inequality follows from the defining property of $Y$,
and the second from the definition of $d_{S}$. On the other hand,
take a bijection $f:\FF_{S}\left(\Phi\right)\rightarrow Y$ for which
$d_{S}\left(\FF_{S}\left(\Phi\right),Y\right)=\|f\|_{S}$, and define
an $S$-assignment $\Theta:S\rightarrow\Sym\left(n\right)$ by $\Theta\left(s\right)\left(x\right)=f^{-1}\left(s\cdot f\left(x\right)\right)$.
Then, $\Theta$ is an $E$-solution because $Y$ is a $\Gamma$-set,
and so 
\[
d_{n}\left(\Phi,\Psi\right)\leq d_{n}\left(\Phi,\Theta\right)=\|\id_{\Phi,\Theta}\|_{S}=\|f\|_{S}=d_{S}\left(\FF_{S}\left(\Phi\right),Y\right)\,\,\text{.}
\]
\end{proof}
The above discussion enables us to define the stability rate of $E$
in terms of group actions, as recorded below:
\begin{prop}
\label{prop:stability-rate-in-terms-of-actions}The stability rate
of $E$ is given by
\[
\SR_{E}\left(\delta\right)=\sup\{G_{E}(X)\mid X\text{ is a finite }\FF_{S}-\text{set and }L_{E}(X)\le\delta\}\,\,.
\]
\end{prop}
\begin{defn}
\label{def:E-abiding-points}For an $\FF_{S}$-set $X$, define the
set of \emph{$E$-abiding points} in $X$ as
\[
X_{E}=\left\{ x\in X\mid\forall w\in E\,\,\,\,w\cdot x=x\right\} \,\,\text{.}
\]
\end{defn}
Note that for an $\FF_{S}$-set $X$, we have 
\begin{equation}
\frac{\left|X\setminus X_{E}\right|}{\left|X\right|}\leq L_{E}\left(X\right)\leq\left|E\right|\cdot\frac{\left|X\setminus X_{E}\right|}{\left|X\right|}\,\,\text{.}\label{eq:local-defect-vs-X_E}
\end{equation}
.

\subsection{\label{subsec:graphs-of-actions}Graphs of actions}

As mentioned, it will be useful to represent a group action as a labeled
graph. Throughout Section \ref{subsec:graphs-of-actions}, we fix
a finite set $S$ and a group $\Lambda$ generated by $S$. We have
a natural surjection $\FF_{S}\to\Lambda$ which enables us to regard
a given action of $\Lambda$ as an action of $\FF_{S}$. The \emph{action
graph} of a finite $\Lambda$-set $X$ (with regard to the set of
generators $S$) is the edge-labeled directed graph over the vertex
set $X$, which has a directed edge labeled $s$ from $x$ to $s\cdot x$
for each $x\in X$ and $s\in S$. Note that the action graph of $X$
remains the same if we choose to treat $X$ as an $\FF_{S}$-set rather
a $\Lambda$-set.

In the context of graphs of actions, it is often useful to consider
pointed sets $\left(X,x_{0}\right)$, i.e., a set $X$ together with
a distinguished point $x_{0}\in X$. The role of $X$ will always
be taken by a $\Lambda$-set or a subset of a $\Lambda$-set. We use
the notation $f:\left(X,x_{0}\right)\rightarrow\left(Y,y_{0}\right)$
for a map $f:X\rightarrow Y$ which sends $x_{0}\mapsto y_{0}$.

It is helpful to consider the following definitions with the role
of $\Lambda$ taken by the free group $\mathbb{F}_{S}$ itself, or
with $\Lambda=\mathbb{Z}^{\left|S\right|}$. We proceed to define
isomorphisms of subgraphs of action graphs, and several related notions. 
\begin{defn}
Let $X$ and $Y$ be $\Lambda$-sets and
let $f:X_{0}\rightarrow Y_{0}$ be a map between subsets $X_{0}\subseteq X$
and $Y_{0}\subseteq Y$.
\begin{enumerate}
\item For $s\in S$ and $x\in X_{0}$, we say that $f$ \emph{preserves}
the edge $x\overset{s}{\longrightarrow}$ if either $s\cdot x\notin X_{0}$
and $s\cdot f\left(x\right)\notin Y_{0}$, or $s\cdot x\in X_{0}$,
$s\cdot f\left(x\right)\in Y_{0}$ and $f\left(s\cdot x\right)=s\cdot f\left(x\right)$.
\item If $f$ is bijective, and preserves $x\overset{s}{\longrightarrow}$
for every $s\in S$ and $x\in X_{0}$, we say that $f$ is a \emph{subgraph
isomorphism} from $X_{0}$ to \emph{$Y_{0}$. }
\end{enumerate}
\end{defn}
For a $\Lambda$-set $X$, a point $x\in X$ and a subset $P\subseteq\Lambda$,
we write $P\cdot x=\left\{ p\cdot x\mid p\in P\right\} $. For subsets
$P_{1}$ and $P_{2}$ of $\Lambda$, we write $P_{1}\cdot P_{2}=\left\{ p_{1}\cdot p_{2}\mid p_{1}\in P_{1},p_{2}\in P_{2}\right\} $.
\begin{defn}
\label{def:F-map}Let $\left(X,x\right)$ and $\left(Y,y\right)$
be pointed $\Lambda$-sets and $P\subseteq\Lambda$, a subset. Assume
that 
\begin{equation}
\Stab_{\Lambda}\left(x\right)\cap\left(P^{-1}P\right)\subseteq\Stab_{\Lambda}\left(y\right)\cap\left(P^{-1}P\right)\,\,\text{.}\label{eq:stab_x-in-stab_y}
\end{equation}
Define the function $F_{P,x,y}:P\cdot x\rightarrow P\cdot y$ by $F_{P,x,y}\left(p\cdot x\right)=p\cdot y$
for each $p\in P$. Note that this function is well-defined since
if $p_{1},p_{2}\in P$ and $p_{1}\cdot x=p_{2}\cdot x$, then $p_{2}^{-1}p_{1}\in\Stab_{\Lambda}\left(x\right)$,
and so $p_{2}^{-1}p_{1}\in\Stab_{\Lambda}\left(y\right)$, hence $p_{1}\cdot y=p_{2}\cdot y$.
\end{defn}
In the notation of the above definition, the function $F_{P,x,y}$
is injective if and only if the inclusion in (\ref{eq:stab_x-in-stab_y})
is in fact an equality. The special case of Definition \ref{def:F-map}
where $X=\Lambda$ gives rise to the following definition:
\begin{defn}
\label{def:injects-bijects}Let $Y$ be a $\Lambda$-set, $y\in Y$
and $P\subseteq\Lambda$, a subset. We say that $P$ \emph{injects
into} $Y$ at $y$ if the map $F_{P,1_{\Lambda},y}:P\rightarrow P\cdot y$
is injective, or, equivalently, if $\Stab_{\Lambda}\left(y\right)\cap\left(P^{-1}P\right)=\left\{ 1_{\Lambda}\right\} $.
We say that $P$ \emph{bijects} onto $Y$ at $y$ if this map is bijective.
\end{defn}
Note that Definition \ref{def:injects-bijects} merely requires the
map $F_{P,1_{\Lambda},y}$ to be injective, but not necessarily a
subgraph isomorphism. We seek sufficient conditions which guarantee
that a given map of the form $F_{P,x,y}$, for general $\FF_{S}$-sets
$X,Y$ and points $x\in X,y\in Y$, preserves a given edge, or even
that it is a subgraph isomorphism. The following two lemmas provide
such conditions by considering short elements of $\Lambda$ and whether
or not they belong to $\Stab_{\Lambda}(x)$ and $\Stab_{\Lambda}(y)$.
\begin{lem}
\label{lem:edge-preservation}Let $\left(X,x\right)$ and $\left(Y,y\right)$
be pointed $\Lambda$-sets and $P\subseteq\Lambda$, a subset, such
that Condition (\ref{eq:stab_x-in-stab_y}) of Definition \ref{def:F-map}
is satisfied. Let $s\in S$, $p\in P$ and assume that 
\[
\Stab_{\Lambda}\left(x\right)\cap\left(P^{-1}\cdot sp\right)=\Stab_{\Lambda}\left(y\right)\cap\left(P^{-1}\cdot sp\right)\,\,\text{.}
\]
Then, the map $F=F_{P,x,y}:P\cdot x\rightarrow P\cdot y$ preserves
$p\cdot x\overset{s}{\longrightarrow}$.
\end{lem}
\begin{proof}
We first note that for $p_1\in P$,
\begin{equation} \label{eq:edgePreservationNote}
sp\cdot x = p_1 \cdot x \:\:\:\:\:\text{if and only if}\:\:\:\:\: sp\cdot y = p_1\cdot y\,\,\text{.}
\end{equation}
Indeed, $sp\cdot x=p_{1}\cdot x$ if and only if $p_{1}^{-1}sp\in\Stab_{\Lambda}\left(x\right)$
if and only if $p_{1}^{-1}sp\in\Stab_{\Lambda}\left(y\right)$ if
and only if $sp\cdot y=p_{1}\cdot y$. 

It follows from (\ref{eq:edgePreservationNote}) that $s\cdot\left(p\cdot x\right)$ belongs to $P\cdot x$
if and only if $s\cdot F\left(p\cdot x\right) = s\cdot \left(p\cdot y\right)$ belongs to $P\cdot y$. Assume
that these equivalent conditions hold. It remains to show that in
this case, $F\left(s\cdot\left(p\cdot x\right)\right)=s\cdot F\left(p\cdot x\right)$.
Write $sp\cdot x=p_{1}\cdot x$ for $p_{1}\in P$. By (\ref{eq:edgePreservationNote}),
$sp\cdot y=p_{1}\cdot y$. So, $F\left(s\cdot\left(p\cdot x\right)\right)=F\left(p_{1}\cdot x\right)=p_{1}\cdot y=sp\cdot y=s\cdot F\left(p\cdot x\right)$,
as required.
\end{proof}
\begin{lem}
\label{lem:subgraph-isomorphism} Let $\left(X,x\right)$ and $\left(Y,y\right)$
be pointed $\Lambda$-sets and $P\subseteq\Lambda$, a subset. Write
$S_{1}=S\cup\left\{ 1_{\Lambda}\right\} $ and assume that
\[
\Stab_{\Lambda}\left(x\right)\cap\left(P^{-1}\cdot S_{1}\cdot P\right)=\Stab_{\Lambda}\left(y\right)\cap\left(P^{-1}\cdot S_{1}\cdot P\right)\,\,.
\]
Then, $F_{P,x,y}:P\cdot x\rightarrow P\cdot y$ is well-defined and
is a subgraph isomorphism.
\end{lem}
\begin{proof}
Since $\Stab_{\Lambda}\left(x\right)\cap\left(P^{-1}\cdot P\right)=\Stab_{\Lambda}\left(y\right)\cap\left(P^{-1}\cdot P\right)$,
the map $F_{P,x,y}$ is well-defined and injective. Furthermore, for
each $s\in S$ and $x\in P\cdot x$, since $\Stab_{\Lambda}\left(x\right)\cap\left(P^{-1}\cdot sp\right)=\Stab_{\Lambda}\left(y\right)\cap\left(P^{-1}\cdot sp\right)$,
Lemma \ref{lem:edge-preservation} implies that $F_{P,x,y}$ preserves
$x\overset{s}{\longrightarrow}$.
\end{proof}
We proceed to define balls in $\FF_{S}$ and $\FF_{S}$-sets, and
give a useful corollary of Lemma \ref{lem:subgraph-isomorphism}.
The word norm $\left|\cdot\right|$ on $\FF_{S}$ in defined for a
word $w\in\FF_{S}$ as the length of $w$ when written as a reduced
word over $S^{\pm}$. Let $X$ be an $\FF_{S}$-set. The word-norm
induces a metric $d_{X}$ on $X$:
\[
\forall x_{1},x_{2}\in X\,\,\,\,\,\,\,\,d_{X}\left(x_{1},x_{2}\right)=\min\left\{ \left|w\right|\mid w\in\FF_S,\text{ }w\cdot x_{1}=x_{2}\right\} \,\,\text{.}
\]
Write $B_{X}(x,r)$ for the ball of radius $r\geq0$ centered at the
point $x\in X$ with respect to $d_{X}$. For $A\subseteq X$, let $B(A,r) = \bigcup_{x\in A} B(x,r)$. In the special case $X=\Lambda$,
we also write $B_{\Lambda}(r)$ for $B_{\Lambda}\left(1_{\Lambda},r\right)$.
This notation will be used often with either $\Lambda=\FF_{S}$ or
$\Lambda=\ZZ^{\left|S\right|}$. For $r\geq0$, plugging in $P=B_{\Lambda}\left(r\right)$
in Lemma \ref{lem:subgraph-isomorphism}, we deduce the following
corollary:
\begin{lem}
\label{lem:subgraph-isomorphism-balls} Let $\left(X,x\right)$ and
$\left(Y,y\right)$ be $\Lambda$-sets and $r\geq0$, an integer.
Assume that
\[
\Stab_{\Lambda}\left(x\right)\cap B_{\Lambda}\left(2r+1\right)=\Stab_{\Lambda}\left(y\right)\cap B_{\Lambda}\left(2r+1\right)\,\,.
\]
Then, the map $F_{B_{\Lambda}\left(r\right),x,y}:B_{X}\left(x,r\right)\rightarrow B_{Y}\left(y,r\right)$
is well-defined and is a subgraph isomorphism.
\end{lem}
Now, assume that $E\subseteq\FF_{S}$ is a finite set generating the
kernel of the surjection $\FF_{S}\rightarrow\Lambda$ as a normal
subgroup (hence $\FF_{S}/\lla E\rra\cong\Lambda$). We end this section
with two definitions and a basic lemma that allows us to bound the
global defect $G_{E}\left(X\right)$ of an $\FF_{S}$-set $X$ with
respect to $E$ (see Definition \ref{def:global-defect-of-action})
in terms of properties of a map between graphs of actions.
\begin{defn}
Let $X$ be a $\Lambda$-set and $X_{0}\subseteq X$ a subset. A point
$x\in X_{0}$ is \emph{internal} in\emph{ $X_{0}$} if $S\cdot x\subseteq X_{0}$.
\end{defn}
\begin{defn}
Let $X$ and $Y$ be $\Lambda$-sets and $X_{0}\subseteq X$, a subset.
Take a function $f:X_{0}\rightarrow Y$. Define the set $\Eq\left(f\right)\subseteq X$
of \emph{equivariance points} of $f$ as
\[
\Eq\left(f\right)=\left\{ x\in X_{0}\mid\forall s\in S\,\,\,\,s\cdot x\in X_{0}\,\,\text{and}\,\,f\left(s\cdot x\right)=s\cdot f\left(x\right)\right\} \,\,\text{.}
\]
That is, $\Eq(f)$ is the set of internal points $x\in X_0$, for which $f$ preserves $x\overset{s}{\longrightarrow}$
for all $s\in S$.
\end{defn}
\begin{lem}
\label{lem:equivariance-points}Let $Y$ be a $\Lambda$-set, $X$
an $\FF_{S}$-set and $f:Y\rightarrow X$ an injective map. Then,
$G_{E}\left(X\right)\leq\left|S\right|\cdot\left(1-\frac{1}{\left|X\right|}\cdot\left|\Eq\left(f\right)\right|\right)$.
\end{lem}
\begin{proof}
Let $Z$ be a trivial $\Lambda$-set of cardinality $\left|Y\right|-\left|X\right|$,
i.e., each $g\in\Lambda$ fixes each $z\in Z$. Fix a bijection
$f_{Z}:Z\rightarrow X\setminus\Ima\left(f\right)$. We consider the disjoint union $Y\coprod Z$ and the bijection $f\coprod f_Z:Y\coprod Z\to X$. Then,
\begin{align*}
G_{E}\left(X\right) & \leq d_{S}\left(Y\coprod Z,X\right)\\
 &\leq\|f\coprod f_{Z}\|_{S}\\ 
 &\le \frac{1}{\left|X\right|}\cdot \left(\left|X\right|\cdot \left|S\right| - \left|S\right|\cdot \left|\Eq\left(f\coprod f_Z\right)\right|\right)\\
 &\leq\frac{1}{\left|X\right|}\cdot\left(\left|X\right|\cdot\left|S\right|-\left|S\right|\cdot\left|\Eq\left(f\right)\right|-\left|S\right|\cdot\left|\Eq\left(f_{Z}\right)\right|\right)\\
 & \leq\left|S\right|\cdot\left(1-\frac{1}{\left|X\right|}\cdot\left|\Eq\left(f\right)\right|\right)\,\,\text{.}
\end{align*}
\end{proof}

\section{Abelian groups are polynomially stable\label{sec:AbelianGroupsArePolynomiallyStable}}

The aim of this section is to prove Theorem \ref{thm:AbelianGroupsArePolynomiallyStable},
our main theorem. We begin by defining several objects that shall
remain fixed throughout Section \ref{sec:AbelianGroupsArePolynomiallyStable}.

Let $\Gamma$ be a finitely-generated abelian group. Without loss
of generality we can realize $\Gamma$ as follows: Let $m\ge d\ge0$,
take a basis $\{e_{1},\ldots,e_{m}\}$ for $\mathbb{Z}^{m}$ and let
$2\le\beta_{m-d+1}\le\ldots\le\beta_{m}$ be integers. Define 
\begin{equation}
K=\left\langle \left\{ \beta_{i}\cdot e_{i}\right\} _{i=m-d+1}^{m}\right\rangle \le\mathbb{Z}^{m}\label{eq:K_Tor}
\end{equation}
 and write $\Gamma=\mathbb{Z}^{m}/K$. Let $\Tor(\Gamma)$ denote
the torsion subgroup of $\Gamma$. Let
\[
\beta_{E}=\begin{cases}
\beta_{m} & \text{if }m>d\\
1 & \text{if }m=d\,\,\text{.}
\end{cases}
\]

Theorem \ref{thm:AbelianGroupsArePolynomiallyStable} asserts that $\deg\left(\SR_\Gamma\right)$ is finite. We will, in fact, provide an explicit upper bound on $\deg\left(\SR_\Gamma\right)$. In the case $d=0$, i.e., if $\Gamma$ is finite, Proposition \ref{prop:finite-groups} says that $\deg\left(\SR_\Gamma\right)=1$. We proceed assuming that $d\ge 1$. We shall show that
\begin{equation}
\deg\left(\SR_{\Gamma}\right)\le C_{\text{bound}}(\Gamma) \label{eq:DBound}
\end{equation}
for 
\begin{equation} \label{eq:CBound}
C_{\text{bound}}(\Gamma)= O\left(2^{d}\cdot d\cdot\max\left\{ d\log d,\log\beta_{E},1\right\} \right)\,\,\text{,}
\end{equation}
where the implied constant of the $O\left(\cdot\right)$ notation
is an absolute constant. We note that, in order to minimize $C_{\text{bound}}(\Gamma)$, one may let $\beta_{m-d+1},\ldots, \beta_m$ correspond to the primary decomposition of $\Gamma$, thus making the constant $\beta_E$ the largest prime power in that decomposition.

Let $\FF_{m}$ be the free group on $S=\{\hat{e}_{1},\ldots,\hat{e}_{m}\}$,
and consider the surjection $\pi:\FF_{m}\rightarrow\ZZ^{m}$, sending
$\hat{e}_{i}\mapsto e_{i}$. We get a sequence of surjections
\[
\FF_{m}\overset{\pi}{\longrightarrow}\ZZ^{m}\longrightarrow\Gamma\,\,\text{.}
\]
We also fix a free group $\FF_{d}$, generated by $\left\{ \hat{e}_{1},\dotsc,\hat{e}_{d}\right\} $.
That is, we write $\FF_{d}$ for this fixed copy of a free group of
rank $d$ inside our fixed free group $\FF_{m}$.

By the definition, the stability rate $\SR_{\Gamma}$ of $\Gamma$
can be computed through any presentation of $\Gamma$ (see Proposition
\ref{prop:StabilityRateIsAGroupInvariant}). We proceed to choose
the equation-set $E$, defining $\Gamma$, with which we will work.
Let 
\[
E_{0}=\left\{ \left[\hat{e}_{i}^{\epsilon_{1}},\hat{e}_{j}^{\epsilon_{2}}\right]\in\FF_{m}\mid i,j\in\left[m\right],\,\,\,\,\epsilon_{1},\epsilon_{2}\in\left\{ +1,-1\right\} \right\} \subseteq\FF_{m}\,\,\text{,}
\]
where $[x,y]$ denotes the commutator $xyx^{-1}y^{-1}$, and define
\[
E=E_{0}\cup\left\{ \hat{e}_{i}^{\beta_{i}}\mid m-d+1\leq i\leq m\right\} \subseteq\mathbb{F}_{m}\,\,\text{.}
\]
Note that if $d=m$, then $E$ is equivalent to the equation-set $E_{\text{comm}}^{d}$
from the introduction. Since $\Gamma\cong\mathbb{F}_{m}/\lla E\rra$,
we have $\SR_{\Gamma}=\left[\SR_{E}\right]$, and so our goal is to
prove that 
$$
\deg\left(\SR_{E}\right)\le C_{\text{bound}}(\Gamma)\,\,\text{.}
$$

For future reference, we fix the following constants:
\begin{align*}
C_{d} & =\max\left\{ 3\cdot7^{d}\cdot d^{2d+2},\beta_{E}\right\} \\
t_{E} & =\max\left\{ d^{-1}\cdot\left(m-d\right)\cdot m\cdot\beta_{E},2\right\} \,\,\text{.}
\end{align*}
Additionally, we let 
\[
\hat{T}=\left\{ \prod_{i=m-d+1}^{m}\hat{e}_{i}^{\alpha_{i}}\in\FF_{m}\mid\forall i\,\,0\leq\alpha_{i}<\beta_{i}\right\} \,\,\text{,}
\]
and 
\[
T=\left\{ \sum_{i=m-d+1}^{m}\alpha_{i}\cdot e_{i}\in\ZZ^{m}\mid\forall i\,\,0\leq\alpha_{i}<\beta_{i}\right\} \,\,\text{.}
\]
In the rest of this section, the implied constants in the $O\left(\cdot\right)$
notation are allowed to depend on $m$, $d$ and $E$.

\subsection{Proof plan \label{subsec:ProofPlan}}

In this section we outline our proof of Theorem \ref{thm:AbelianGroupsArePolynomiallyStable},
as implemented in Sections \ref{subsec:GeometricDefinitions}--\ref{subsec:TilingAlgorithm}. 

Let $X$ be an $\FF_{m}$-set, and write $n=\left|X\right|$. By Proposition \ref{prop:stability-rate-in-terms-of-actions}, in order
to bound $\SR_{E}$, it suffices to bound $G_{E}(X)$ in terms of
$L_{E}(X)$. To this end, we algorithmically construct a
$\Gamma$-set $Y$ (Proposition \ref{prop:TilingAlgorithm}), together with a certain injection $f:Y\to X$
with many equivariance points, namely, $|\Eq(f)|\ge n\cdot\left(1-O\left(L_{E}(X)^{\frac{1}{C_{\text{bound}}(\Gamma)}}\right)\right)$.
Lemma \ref{lem:equivariance-points} then gives the bound $G_{E}(X)\le|S|\cdot\left(1-\frac{|\Eq(f)|}{n}\right)\le O\left(L_{E}(X)^{\frac{1}{C_{\text{bound}}(\Gamma)}}\right)$,
which yields the claim of Theorem \ref{thm:AbelianGroupsArePolynomiallyStable}.

We build $Y$ as the disjoint union of a collection of small $\Gamma$-sets
$\{Y_{x}\}_{x\in J}$, each equipped with an injection $f_{x}:Y_{x}\to X$.
The images of these injections are pairwise disjoint and $f:Y\to X$
is taken to be the disjoint union of the maps $f_{x}$. Clearly, due
to this construction, $|\Eq(f)|$ is at least $\sum_{x\in J}|\Eq(f_{x})|$,
so we wish to maximize the latter sum. Towards this end, it is desirable
that the images of the injections $f_{x}$ cover almost all of $X$,
and that each of the injections has a large fraction of equivariance
points. We manage to construct the injections $f_{x}$ so that the
equivariance points of $f_{x}$ are approximately those points in
$Y_{x}$ that are mapped to internal points of $\Ima(f_{x})\subseteq X$.
Hence, we think of the ratio $\frac{|Y_{x}\setminus\Eq(f_{x})|}{|Y_{x}|}$
as an \emph{isoperimetric ratio}, which we wish to minimize.

The reader may prefer to read both the proof and its outline under
the simplifying assumption that $\Gamma$ is torsion-free, i.e., $m=d$
and, accordingly, $\Gamma=\ZZ^{m}$. In fact, the torsion-free case of Theorem \ref{thm:AbelianGroupsArePolynomiallyStable} implies the general case. This follows from Proposition \ref{prop:quotient-by-fg}, since every finitely-generated abelian group is a quotient of a torsion-free finitely-generated abelian group by a finitely-generated subgroup. The simplified strategy, of starting with the torsion-free case and then using Proposition \ref{prop:quotient-by-fg}, comes at the price of a somewhat worse bound on the degree, compared to (\ref{eq:DBound}).

We turn to give an outline of our algorithm. 

\subsubsection{The algorithm constructing $Y$ and $f$\label{subsec:ProofPlanAlgorithmDescription}}

Our algorithm works iteratively as follows. We first initialize a
rather large number $t_{1}$ which depends on the ratio $\frac{|X_{E}|}{|X|}$,
which in turn is related to $L_{E}\left(X\right)$ (see Equation (\ref{eq:local-defect-vs-X_E})).
In the first iteration, we find, in a greedy manner, a collection
of $\Gamma$-sets $\{Y_{x}\}_{x\in J_{1}}$ and respective injections
$\{f_{x}:Y_{x}\to X\}{}_{x\in J_{1}}$ with pairwise disjoint
images, such that $\frac{|Y_{x}\setminus\Eq(f_{x})|}{|Y_{x}|}\le O(\frac{1}{t_{1}})$
for each $x\in J_{1}$. We think of the images of the injections
$\left\{ f_{x}\right\} _{x\in J_{1}}$ as ``tiles'' embedded in
$X$, and of $t_{1}$ as a parameter used in the construction of these
tiles. The set $J_{1}$ is maximal in the sense that we cannot add
more tiles with parameter $t_{1}$ without violating the constraint
that they be disjoint. We proceed to tile the remainder of $X$. We
define a new parameter $t_{2}<t_{1}$ which is equal to $t_{1}$ divided
by some constant, and repeat this process for another iteration, which
yields additional $\Gamma$-sets $\{Y_{x}\}_{x\in J_{2}}$ and corresponding
injections, perhaps with a worse isoperimetric ratio. We require that
the images of these injections be disjoint from each other, as well
as from the images obtained in the previous iteration. We proceed
in this manner, tiling a constant fraction of the remainder of $X$
in each iteration, until $t_{i}$ is below a certain threshold, at
which point the iterative algorithm halts. Finally, we set $Y$ to
be the disjoint union of the $\Gamma$-sets $\left\{ Y_{x}\right\} _{x\in J_{1}\cup\dots\cup J_{s}}$
constructed throughout the $s$ iterations, and define $f:Y\rightarrow X$
as the disjoint union of the maps $\left\{ f_{x}\right\} _{x\in J_{1}\cup\dots\cup J_{s}}$.

Note that, as the algorithm progresses, our injections become less
and less efficient, that is, their images have a larger isoperimetric
ratio. After developing the necessary machinery in Sections \ref{subsec:ToolC}--\ref{subsec:ToolA},
we conclude the proof of Theorem \ref{thm:AbelianGroupsArePolynomiallyStable}
in Section \ref{subsec:TilingAlgorithm} by defining the above
algorithm, and showing that it produces an injection $f$ with many
equivariance points. We turn to discuss the technique by which we
build the sets $Y_{x}$ and the corresponding injections $f_{x}$.

\subsubsection{Mapping a single $\Gamma$-set $Y_{x}$ into $X$ }

As mentioned, the $i$-th iteration of our algorithm injects a collection
of $\Gamma$-sets $\{Y_{x}\}_{x\in J_{i}}$ into $X$. The isoperimetric
ratio of each of these injections must be bounded by $O(\frac{1}{t_{i}}$),
where $t_{i}$ is the parameter introduced in Section \ref{subsec:ProofPlanAlgorithmDescription}.
We now focus on the main technical challenge, namely, building a single
finite $\Gamma$-set $Y_{x}$ and a corresponding injection $f_{x}:Y_{x}\rightarrow X$.
The efficiency of our construction is reflected in the fact that an
``accumulation'' of no more than $O\left(t_{i}^{d}\right)$ nearby
points of $X$ which belong to $X_{E}$ suffices to construct an injection
with isoperimetric constant bounded by $O\left(\frac{1}{t_{i}}\right)$.
We begin by describing two essential tools for the construction of
$Y_{x}$ and $f_{x}$:

\paragraph*{Tool A (Proposition \ref{prop:toolA}):}

This tool requires a point $x\in X_{E}$ with a large enough neighborhood
(called a ``box-neighborhood of side-length $t_{i}$'') which is
entirely contained in $X_{E}$. It provides a radius $r_A \ge \Omega\left(t_i\right)$ such that the ball $B_X\left(x,r_A\right)$ is contained in the aforementioned neighborhood. It also provides a new (usually infinite)
pointed $\Gamma$-set $\left(U_{A},u_{A}\right)$ and a subgraph isomorphism $f_A$:
\[
\xymatrix{B_{U_{A}}\left(u_{A},r_{A}\right)\myarr{r}{f_{A}}\ar@{}[d]|-*[@]{\subseteq} & B_{X}\left(x,r_{A}\right)\ar@{}[d]|-*[@]{\subseteq}\\
\left(U_{A},u_{A}\right) & X
}
\text{ }\text{ }
\]

\paragraph*{Tool C (Proposition \ref{prop:Pt-isoperimetric-properties}):}

Given a pointed $\Gamma$-set $\left(V,v\right)$ and $t_{C}\in\NN$,
this tool creates an injective map:
\[
\xymatrix{\left(Y,y\right)\myarr{r}{f_{C}} & \left(V,v\right)}
\]
where $\left(Y,y\right)$ is a new small \emph{finite} pointed $\Gamma$-set
and the isoperimetric ratio of $\Ima\left(f_{C}\right)$ is bounded
by $O\left(\frac{1}{t_{C}}\right)$.

\medskip{}

One may be tempted to try to construct $Y_{x}$ and the injective
map $f_{x}:Y_{x}\rightarrow X$ as follows: Locate a point $x\in X_{E}$
with the property required by Tool A, and use this tool to create:
\[
\xymatrix{B_{U_{A}}\left(u_{A},r_{A}\right)\myarr{r}{f_{A}} & X}
\]
as in the description of Tool A. Then, apply Tool C to $\left(V,v\right)=\left(U_{A},u_{A}\right)$
with some $t_{C}\geq\Omega\left(t_{i}\right)$, and create:
\[
\xymatrix{\left(Y_{x},y_{x}\right)\myarr{r}{f_{C}} & \left(U_{A},u_{A}\right)}
\]
as in the description of Tool C. Now, in the very fortunate case where
the image of $f_{C}$ is contained in the domain $B_{U_{A}}\left(u_{A},r_{A}\right)$
of $f_{A}$, we can define the map $f_{x}:Y_{x}\rightarrow X$
as the following composition:
\[
Y_{x}\overset{f_{C}}{\longrightarrow}B_{U_{A}}\left(u_{A},r_{A}\right)\overset{f_{A}}{\longrightarrow}X\,\,\text{.}
\]
When this works, the map $f_{x}$ is injective and its image has a
small isoperimetric ratio, as required, because these properties hold
for $f_{C}$ and since $f_{A}:B_{U_{A}}\left(u_{A},r_{A}\right)\rightarrow B_{X}\left(x,r_{A}\right)$
is a subgraph isomorphism. However, the image of the map $f_{C}$,
produced by Tool C, is usually too large to be contained in
the domain of $f_{A}$. We solve this issue by introducing yet another
$\Gamma$-set $U_{B}$ which sits between $Y_{x}$ and $B_{U_{A}}\left(u_{A},r_{A}\right)$
in the above diagram. As is the case for $U_{A}$, the set $U_{B}$
is usually infinite. It is generated by Tool B (see below), and has
a useful combination of properties: (I) locally, $U_{B}$ looks like
$U_{A}$ in a rather large radius, and (II) when Tool C is applied
to $U_{B}$, the image in $U_{B}$ of the resulting injection $f_{C}$
is relatively small.

\paragraph*{Tool B (Proposition \ref{prop:tile-construction}):}

Given a pointed $\Gamma$-set $\left(U_{A},u_{A}\right)$ and $t_{B}\in\NN$,
this tool creates a subgraph isomorphism:
\[
\xymatrix{\left(U_{B}^{0},u_{B}\right)\myarr{r}{f_{B}}\ar@{}[d]|-*[@]{\subseteq} & \left(U_{A}^{0},u_{A}\right)\ar@{}[d]|-*[@]{\subseteq}\\
\left(U_{B},u_{B}\right) & \left(U_{A},u_{A}\right)
}
\]
where $\left(U_{B},u_{B}\right)$ is a new pointed $\Gamma$-set and
$U_{A}^{0}$ and $U_{B}^{0}$ are finite sets. The tool guarantees
the following properties:
\begin{enumerate}[label=(\Roman*)]
\item $U_{A}^{0}\subseteq B_{U_{A}}\left(u_{A},O\left(t_{B}\right)\right)$,
and
\item The set $U_{B}^{0}$ is exactly the image of the map $f_{C}$ that
our implementation of Tool C provides when it is applied to $\left(V,v\right)=\left(U_{B},u_{B}\right)$
with $t_{C}=t_{B}$.
\end{enumerate}
\medskip{}

Using all three tools, we define $f_{x}$ as the composition of the
following chain of maps:

\[
\xymatrix{Y_{x}\myarr{r}{f_{C}} & \Ima\left(f_{C}\right)\myarr{r}{f_{B}}\ar@{}[d]|-*[@]{\subseteq} & B_{U_{A}}\left(u_{A},r_{A}\right)\myarr{r}{f_{A}}\ar@{}[d]|-*[@]{\subseteq} & X\\
 & U_{B} & U_{A}
}
\]
The objects and maps in the diagram above are created by first using
Tool A to create $U_{A}$ and $f_{A}$, then applying Tool B to $U_{A}$
to create $U_{B}$ and $f_{B}$, and finally applying Tool C to $U_{B}$
to create $Y_{x}$ and $f_{C}$. All three maps $f_{A}$, $f_{B}$
and $f_{C}$ are injective, and so the same is true for $f_{x}$.
Both $f_{A}$ and $f_{B}$ are subgraph isomorphisms onto their respective
images, and the image of $f_{C}$ has isoperimetric ratio at most
$O\left(\frac{1}{t_{i}}\right)$. Hence, the same isoperimetric property
is true for $f_{x}$, as required.

The proofs for Tools A and B use \emph{basis reduction theory} of
sublattices of $\ZZ^{m}$, and the proof for Tool C is also in a related
spirit. More specifically, a transitive $\Gamma$-set $V$ is isomorphic
to $\ZZ^{m}/H$ for some subgroup $H\leq\ZZ^{m}$, and we are able
to study $V$ by applying reduction theory to $H$, thought of as
a sublattice of $\ZZ^{m}$.

We note that the isoperimetric property possessed by each tile serves
two purposes in our proof. The ``local purpose'' is to ensure that
each injection $f_{x}$ has a large fraction of equivariance points
as described above. The ``global purpose'' is to ensure that we
can pack \emph{many} tiles into $X$.

\subsection{Geometric definitions\label{subsec:GeometricDefinitions}}

Section \ref{subsec:graphs-of-actions} introduced the word-norm on
a free group, the word-metric on sets which are acted on
by a free group, and defined balls with respect to the word-metric.
In addition to these, our proof will make use of various norms on
$\ZZ^{m}$ (specifically, $L^{1}$, $L^{2}$ and $L^{\infty}$), and
of ``boxes'' in free groups. These are described below:

\paragraph*{Geometry of $\mathbb{Z}^{m}$}

For $1\le p\le\infty$, let $\|\cdot\|_{p}$ denote the $L^{p}$ norm,
restricted to $\mathbb{Z}^{m}$. Let $B_{\mathbb{Z}^{m}}^{L^{p}}(x,r)\subseteq\mathbb{Z}^{m}$
be the closed ball, with respect to $\|\cdot\|_{p}$, of radius $r$,
centered at $x$. Again, we omit $x$ for a ball centered at $0_{\ZZ^{m}}$.
We note that in the case $p=1$, $\|\cdot\|_{1}$ coincides with the
word-metric on $\mathbb{Z}^{m}$ as an $\FF_{m}$-set, and so $B_{\mathbb{Z}^{m}}(x,r)=B_{\mathbb{Z}^{m}}^{L^{1}}(x,r)$. 

\paragraph*{\emph{Boxes} in $\mathbb{F}_{m}$ and $\protect\FF_{d}$}

Let $\FF_{k}$ be the free group on $\{\hat{e}_{1},\ldots,\hat{e}_{k}\}$
(we are interested in $k\in\left\{ m,d\right\} $). We say that a
word in $\mathbb{F}_{k}$ is \emph{sorted} if it is of the form $\prod_{i=1}^{k}\hat{e}_{i}^{a_{i}}$,
$a_{i}\in\ZZ$. Write $\pi_{k}:\FF_{k}\rightarrow\ZZ^{k}$ for the
surjection sending $\hat{e}_{i}\mapsto e_{i}$. For each $v\in\mathbb{Z}^{k}$
we define a canonical representative $\hat{v}\in\mathbb{F}_{k}$
of the set $\pi_{k}^{-1}(v)$, namely, $\hat{v}$ is the unique sorted
word such that $\pi_{k}(\hat{v})=v$. Reusing this notation, for $w\in\mathbb{F}_{k}$
let $\hat{w}$ denote $\widehat{\pi(w)}$, i.e, the sorted form of
$w$. The following definition introduces a key player in our proof:
\begin{defn}
For $t\in\NN$, let 
\[
\Box_{\FF_{k}}\left(t\right)=\left\{ \hat{v}\mid v\in B_{\mathbb{Z}^{k}}^{L^{\infty}}(t)\right\} =\left\{ \hat{v}\mid v\in\mathbb{Z}^{k},\,\,\|v\|_{\infty}\le t\right\} \subseteq\FF_{k}\,\,\text{.}
\]

Note that $w\in\Box_{\FF_{k}}\left(\left|w\right|\right)$ whenever
$w$ is sorted. Finally, for a subset $A$ of $\ZZ^{k}$ or $\FF_{k}$,
define $\hat{A}=\left\{ \hat{a}\mid a\in A\right\} $.
\end{defn}

\subsection{A review of reduction theory (of sublattices of $\protect\ZZ^{l}$)}

Let $l\ge1$. We use the term \emph{lattice} interchangeably with
\emph{a subgroup} $H\le\mathbb{Z}^{l}$. A \emph{basis }for $H$ is
a set of linearly independent vectors $\calD\subseteq\mathbb{Z}^{l}$
such that $H$ is the set of all integer combinations of elements
in $\calD$. It is well known that every lattice affords a basis,
and that different bases representing the same lattice all have the
same cardinality. Hence, we may define the \emph{rank }of $H$ by
$\rank H=|\calD|$. The goal of \emph{reduction theory} (See \cite{Lag95}
for a survey) is to represent $H$ via a \emph{reduced basis}, namely,
a basis consisting of relatively short vectors. In this section, we
adapt a certain result from reduction theory to our purposes.
\begin{defn}
Let $H\leq\mathbb{Z}^{l}$ be a lattice and $k=\rank H$. The \emph{successive
minima} sequence $\lambda_{1}\left(H\right),\dotsc,\lambda_{k}\left(H\right)$
is defined as follows: $\lambda_{i}\left(H\right)$ is the minimum
radius $r$ for which $\rank\left\langle \left\{ v\in H\mid\|v\|_{2}\leq r\right\} \right\rangle \ge i$.
\end{defn}
We note that for $l>4$, the vectors that yield the successive minima
are not necessarily a basis for $H$ (\cite{Mar02} p.\ 51). However,
as we now elaborate, $H$ does accommodate a basis consisting of vectors
within the same order of magnitude as the successive minima. It is
known that every lattice has a basis of a certain type called \emph{Korkin-Zolotarev
reduced \cite{Lag95}. }The following proposition about such bases
is a part of Theorem 2.1 in \cite{LLS90}.
\begin{prop}
\label{prop:Lenstra}If $H\le\mathbb{Z}^{l}$ is a lattice and $B=\{b_{1},\ldots,b_{k}\}$
is a Korkin-Zolotarev reduced basis for $H$, then $\|b_{i}\|_{2}\le\frac 12 \cdot \sqrt{i+3}\cdot\lambda_{i}(H)$
for each $1\le i\le k$. 
\end{prop}
Our use of Proposition \ref{prop:Lenstra} will always be mediated
through Proposition \ref{prop:KZ-reduced-basis-corollary}, below.
\begin{defn}
For a finite subset $A\subseteq\ZZ^{l}$, write $\|A\|_{1}=\sum_{v\in A}\|v\|_{1}$.
\end{defn}
\begin{prop}
\label{prop:KZ-reduced-basis-corollary}Let $H\le\mathbb{Z}^{l}$
be a lattice of rank $k$ and $t\in\NN$, such that $\rank\langle H\cap B_{\ZZ^{l}}^{L^{2}}\left(t\right)\rangle=k$.
Then, $H$ has a basis $\calD$ such that:
\begin{enumerate}
\item $\calD\subseteq B_{\ZZ^{l}}\left(l\cdot t\right)$.
\item $\|\calD\|_{1}\le l^{2}\cdot t$.
\end{enumerate}
\end{prop}
\begin{proof}
Let $\calD=\{b_{1},\dotsc,b_{k}\}$ be a Korkin-Zolotarev reduced
basis for $H$. By Proposition \ref{prop:Lenstra}, 
\[
\|b_{i}\|_{2}\le\frac{1}{2}\sqrt{i+3}\cdot\lambda_{i}\left(H\right)\le\frac{1}{2}\sqrt{l+3}\cdot t\le\frac{1}{2}\sqrt{4l}\cdot t=\sqrt{l}\cdot t\,\,,
\]
and so
\[
\|b_{i}\|_{1}\leq l\cdot t\,\,\text{,}
\]
which yields the first claim. The second claim follows since 
\[
\|\calD\|_{1}\le|\calD|\cdot\max_{1\le i\le k}\|b_{i}\|_{1}\le k\cdot l\cdot t\le l^{2}\cdot r\,\,.
\]
\end{proof}

\subsection{Tool C: The standard-completion isoperimetric method \label{subsec:ToolC}}

In this section we develop \emph{Tool C}. Let $(V,v_{0})$ be a pointed
$\Gamma$-set and $t\in\NN$. Our goal is to build a small finite
pointed $\Gamma$-set $(Y,y_{0})$, and an injective map $f_{C}:(Y,y_{0})\to(V,v_{0})$
such that 
\begin{equation}
\left|\Eq(f_{C})\right|\ge|Y|\cdot\left(1-O\left(\frac{1}{t}\right)\right)\,\,\text{.}\label{eq:ToolCEquivariance}
\end{equation}

We may assume that $V$ is transitive, since otherwise we may consider
only the component containing $v_{0}$. Hence, $V$ is realizable
as $\ZZ^{m}/H$ for some $K\le H\le\ZZ^{m}$ (see Equation (\ref{eq:K_Tor}))
, with $v_{0}$ corresponding to the coset $0+H$. Let ${\mathcal{D}}_{0}$
be a lattice basis for $H$. We first consider the simple case where
${\mathcal{D}}_{0}$ consists of axis-parallel vectors, namely, ${\mathcal{D}}_{0}=\left\{ \alpha_{i}\cdot e_{i}\mid i\in[m]\setminus I\right\} $
for some $I\subseteq[m]$, where $\{\alpha_{i}\}_{i\in[m]\setminus I}$
are positive integers. We complete ${\mathcal{D}}_{0}$ to a full-rank
basis ${\mathcal{T}}={\mathcal{D}}_{0}\cup\left\{ 2t\cdot e_{i}\right\} _{i\in I}$.
Define $\left(Y,y_{0}\right)=\left(\mathbb{Z}^{m}/\left\langle {\mathcal{T}}\right\rangle ,0+\left\langle {\mathcal{T}}\right\rangle \right)$,
and note that $Y$ is a $\Gamma$-set since $K\subseteq H\subseteq\left\langle {\mathcal{T}}\right\rangle $. 

Each point $y\in Y$ has a unique representation as
\[
y=\sum_{f\in\calD_{0}}a_{f}\cdot f+\sum_{i\in I}b_{i}\cdot e_{i}+\left\langle {\mathcal{T}}\right\rangle \,\,\,\,0\leq a_{f}<1,\,\,-t\leq b_{i}<t\,\,\text{.}
\]
Let $f_{C}:(Y,y_{0})\to\left(V,v_{0}\right)$ map such a point $y$ to the point
\[
\sum_{f\in\calD_{0}}a_{f}\cdot f+\sum_{i\in I}b_{i}\cdot e_{i}+H
\]
of $V$. This map $f_{C}$ is clearly injective. Now, consider a point
$y\in Y$ as above. The map $f_{C}$ necessarily preserves the edge
$y\overset{e_{i}}{\longrightarrow}$ for each $i\in[m]\setminus I$.
If $i\in I$, then $f_{C}$ preserves this edge \textbf{unless} $b_{i}=t-1$.
It follows that 
\[
\Eq\left(f_{C}\right)=\left\{ \sum_{f\in\calD_{0}}a_{f}\cdot f+\sum_{i\in I}b_{i}\cdot e_{i}+\left\langle {\mathcal{T}}\right\rangle \mid0\leq a_{f}<1\,\,\text{and}\,\,-t\leq b_{i}<t-1\right\} ,
\]
 so Equation (\ref{eq:ToolCEquivariance}) is satisfied. As mentioned,
it is desirable that $Y$ be small, i.e, that ${\mathcal{T}}$ be a short
basis. While we cannot control ${\mathcal{D}}_{0}$, we have chosen the
additional vectors $2t\cdot e_{i}$ to be as short as possible, that
is, just long enough to guarantee Equation (\ref{eq:ToolCEquivariance}).

We turn to the general case, where the basis ${\mathcal{D}}_{0}$ of $H$
might not consist of axis-parallel vectors. Again, we augment ${\mathcal{D}}_{0}$
with axis-parallel vectors to form a full-rank basis, namely, ${\mathcal{T}}={\mathcal{D}}_{0}\cup\left\{ 2t\cdot e_{i}\mid i\in I\right\} $
for some carefully chosen $I\subseteq[m]$, and define $\left(Y,y_{0}\right)$
and $f_{C}$ as above. However, it is now possible that an edge labeled
$e_{j}$ ($j\in[m]\setminus I$) is not preserved by $f_{C}$. In
order to analyze the behavior of the generator $e_{j}$, we apply
a linear transformation that maps ${\mathcal{T}}$ to an axis-parallel
basis. The image of $e_{j}$ under this transformation may be large
in its $I$ coordinates, which means that many $e_{j}$-labeled edges
are not preserved by $f_{C}$. In order to control the effect of these
generators $e_{j}$, we need to choose $I$ so that the vectors $\left\{ e_{i}\right\} _{i\in I}$
are \emph{nearly orthogonal }to ${\mathcal{D}}_{0}$. Consequently, ${\mathcal{T}}$
is already close to being an orthogonal basis, thereby bounding the
distortion of the linear transformation that ``fixes'' it to being
orthogonal. We proceed to formally define these notions.

Given a finite ordered set of vectors $A\subseteq\mathbb{R}^{m}$,
we write ${\mathcal{M}}_{A}$ for the matrix whose columns are the elements
of $A$ in the standard basis. Note that if $A$ is a basis for $\mathbb{R}^{m}$
then ${\mathcal{M}}_{A}^{-1}$ is the change of basis matrix transforming
a standard-coefficients vector to an $A$-coefficients vector (by
multiplying column vectors of coefficients from the left). Given a
set $I\subseteq[m]$, let ${\mathcal{M}}_{A}^{I}$ denote the matrix ${\mathcal{M}}_{A}$
with the $i$-th row stricken out for every $i\in I$. Notice that
if $\left|I\right|=m-\left|\calD_{0}\right|$, then
\begin{equation}
\left|\det{\mathcal{M}}_{{\mathcal{D}}_{0}}^{I}\right|=\left|\det{\mathcal{M}}_{{\mathcal{D}}_{0}\cup{\mathcal{D}}_{1}}\right|\,\,\text{,}\label{eq:DeterminantM_B_0^I}
\end{equation}
where ${\mathcal{D}}_{1}=\left\{ e_{i}\mid i\in I\right\} $. 
\begin{defn}
Let ${\mathcal{D}}_{0}\subseteq\ZZ^{m}$ be an ordered set of linearly
independent vectors.
\begin{enumerate}
\item Write $k=\left|{\mathcal{D}}_{0}\right|$. For a set of coordinates $I\subseteq\left[m\right]$
with $\left|I\right|=m-k$, the \emph{strength} of $I$ with respect
to ${\mathcal{D}}_{0}$ is $\left|\det\left({\mathcal{M}}_{{\mathcal{D}}_{0}}^{I}\right)\right|$,
that is, the unsigned volume of the parallelotope  generated by ${\mathcal{D}}_{0}\cup\left\{ e_{i}\right\} _{i\in I}$
(see Equation (\ref{eq:DeterminantM_B_0^I})). 
\item We say that $I$, as above, is a set of \emph{strongest coordinates}
for ${\mathcal{D}}_{0}$ if $\left|\det\left({\mathcal{M}}_{{\mathcal{D}}_{0}}^{I}\right)\right|\ge\left|\det\left({\mathcal{M}}_{{\mathcal{D}}_{0}}^{J}\right)\right|$
for every $J\subseteq\left[m\right]$ satisfying $\left|J\right|=m-k$.
Note that in this case, $\left|\det\left({\mathcal{M}}_{{\mathcal{D}}_{0}}^{I}\right)\right|$
is strictly positive.
\item For a set ${\mathcal{D}}_{1}\subseteq\ZZ^{m}$, we say the ${\mathcal{D}}_{1}$
is a \emph{standard complement} for ${\mathcal{D}}_{0}$ if ${\mathcal{D}}_{0}\cup{\mathcal{D}}_{1}$
is a basis for $\RR^{m}$ and ${\mathcal{D}}_{1}=\left\{ e_{i}\right\} _{i\in I}$
for some $I\subseteq\left[m\right]$. A standard complement ${\mathcal{D}}_{1}$
for ${\mathcal{D}}_{0}$ as above is \emph{strong} if $I$ is a set of
strongest coordinates with respect to ${\mathcal{D}}_{0}$.
\end{enumerate}
\end{defn}
Note that for every set of strongest coordinates $I\subseteq[m]$
for $\calD_{0}$, the set ${\mathcal{D}}_{1}=\left\{ e_{i}\right\} _{i\in I}$
is a strong standard complement for $\calD_{0}$. The following lemma
formulates the near orthogonality condition mentioned above.
\begin{lem}
\label{lem:strong-bounds}Let ${\mathcal{D}}_{0}\subseteq\ZZ^{m}$ be a
linearly independent set. Take a strong standard complement ${\mathcal{D}}_{1}$
for ${\mathcal{D}}_{0}$. Write ${\mathcal{D}}_{0}=\left\{ h_{1},\dotsc,h_{k}\right\} $
and ${\mathcal{D}}_{1}=\left\{ e_{i}\right\} _{i\in I}$, $I=\left\{ i_{1},\dotsc,i_{m-k}\right\} $,
and let ${\mathcal{D}}=\left(h_{1},\dotsc,h_{k},e_{i_{1}},\dotsc,e_{i_{m-k}}\right)$.
Then, all entries of the lower $(m-k)\times m$ block of the matrix
${\mathcal{M}}_{{\mathcal{D}}}^{-1}$ are in the range $[-1,1]$
\end{lem}
\begin{proof}
Permuting the rows of $\calM_{\calD}$ maintains the property that
the right $m-k$ columns of $\calM_{\calD}$ are a strong standard
complement for the left $k$ columns, and affects $\calM_{\calD}^{-1}$
by a permutation of its columns (in particular, it permutes the lower
$\left(m-k\right)\times m$ block of $\calM_{\calD}^{-1}$). Therefore,
we assume that $I=\left\{ k+1,\dotsc,m\right\} $ without loss of
generality. Accordingly, ${\mathcal{D}}=\left(h_{1},\dotsc,h_{k},e_{k+1},\dotsc,e_{m}\right)$.
Let $k+1\leq i\leq m$ and $1\leq j\leq m$. Our task is to show that
$\left|\left(\calM_{\calD}^{-1}\right)_{i,j}\right|\leq1$. Note that
the lower-right block of ${\mathcal{M}}_{{\mathcal{D}}}^{-1}$ is an $(m-k)\times(m-k)$
identity matrix. Thus, the claim holds whenever $j>k$. Henceforth,
assume that $1\leq j\leq k$. Let ${\mathcal{M}}_{{\mathcal{D}}}^{j,i}$ denote
the matrix ${\mathcal{M}}_{{\mathcal{D}}}$ with the $j$-th row and $i$-th
column removed. By Cramer's rule, the $(i,j)$-entry of ${\mathcal{M}}_{{\mathcal{D}}}^{-1}$
is either
\[
\frac{\det{\mathcal{M}}_{{\mathcal{D}}}^{j,i}}{\det{\mathcal{M}}_{{\mathcal{D}}}}
\]
or its negation. Hence, it suffices to show that $\left|\det{\mathcal{M}}_{{\mathcal{D}}}^{j,i}\right|\le\left|\det{\mathcal{M}}_{{\mathcal{D}}}\right|$.
By Equation (\ref{eq:DeterminantM_B_0^I}), $\left|\det{\mathcal{M}}_{{\mathcal{D}}}\right|=\left|\det{\mathcal{M}}_{{\mathcal{D}}_{0}}^{I}\right|$.
Similarly, $\left|\det{\mathcal{M}}_{{\mathcal{D}}}^{j,i}\right|=\left|\det{\mathcal{M}}_{{\mathcal{D}}_{0}}^{(I\setminus\{i\})\cup\{j\}}\right|$.
The claim follows since $I$ is a set of strongest coordinates for
${\mathcal{D}}_{0}$, and so $\left|\det\calM_{\calD_{0}}^{I}\right|\geq\left|\det\calM_{\calD_{0}}^{\left(I\setminus\left\{ i\right\} \right)\cup\left\{ j\right\} }\right|$.
\end{proof}
We need to fix, once and for all, a strong standard complement ${\mathcal{C}}({\mathcal{D}}_{0})$
for each linearly independent ${\mathcal{D}}_{0}\subseteq\ZZ^{m}$. For
our purposes, it does not matter how this is done. However, for concreteness,
we do it as follows:
\begin{defn}
For an ordered linearly independent
subset ${\mathcal{D}}_{0}\subseteq\ZZ^{m}$, let $\calC\left({\mathcal{D}}_{0}\right)\subseteq\ZZ^{m}$
be the lexicographically least among all strong standard complements
for ${\mathcal{D}}_{0}$. Let $I\left({\mathcal{D}}_{0}\right)$ denote the
subset of $\left[m\right]$ for which $\calC\left({\mathcal{D}}_{0}\right)=\left\{ e_{i}\right\} _{i\in I\left({\mathcal{D}}_{0}\right)}$. 
\end{defn}
\begin{defn}
\label{def:Pt}Let ${\mathcal{D}}_{0}\subseteq\ZZ^{m}$ be a linearly independent
set. Write $k=\left|\calD_{0}\right|$, and order $I\left({\mathcal{D}}_{0}\right)$
by writing $I\left({\mathcal{D}}_{0}\right)=\left\{ i_{k+1},\dotsc,i_{m}\right\} $,
$i_{k+1}<\dots<i_{m}$. Define an ordered basis ${\mathcal{D}}={\mathcal{D}}_{0}\cup{\mathcal{C}}\left({\mathcal{D}}_{0}\right)=\left(h_{1},\dotsc,h_{k},e_{i_{k+1}},\dotsc,e_{i_{m}}\right)$
for $\mathbb{R}^{m}$. For $t\in\NN$, define a \emph{discrete parallelotope
}
\[
P_{t}^{{\mathcal{D}}_{0}}=\left\{ x\in\ZZ^{m}\mid{\mathcal{M}}_{\calD}^{-1}(x)\in[0,1)^{k}\times\left[-t,t\right)^{m-k}\right\} \,\,,
\]
and write $\hat{P}_{t}^{{\mathcal{D}}_{0}}$ for its canonical lift to
$\FF_{m}$ (see Section \ref{subsec:GeometricDefinitions}), to wit,
\[
\hat{P}_{t}^{{\mathcal{D}}_{0}}=\left\{ \hat{x}\mid x\in P_{t}^{{\mathcal{D}}_{0}}\right\} \,\,\text{.}
\]
\end{defn}
\begin{lem}
\label{lem:Pt-basic-properties}Let $\calD_{0}\subseteq\ZZ^{m}$ be
a linearly independent set and $t\in\NN$. Then,
\begin{enumerate}
\item \label{enu:Pt-basic-BallAroundParallelopiped} $P_{t}^{{\mathcal{D}}_{0}}\subseteq B_{\ZZ^{m}}\left(\|{\mathcal{D}}_{0}\|_{1}+\left(m-\left|{\mathcal{D}}_{0}\right|\right)\cdot t\right)$.
\item \label{enu:Pt-basic-GrowthInt} $\left|P_{t}^{\calD_{0}}\right|=\left|P_{1}^{\calD_{0}}\right|\cdot t^{m-\left|\calD_{0}\right|}$.
\item \label{enu:Pt-plusminusParallelopiped}
Let $U=\ZZ^{m}/\langle\calD_{0}\rangle$ and $u_{0}\in U$. Then,
$\left(-P_{t}^{\calD_{0}}+P_{t}^{\calD_{0}}\right)\cdot u_{0}\subseteq P_{2t}^{\calD_{0}}\cdot u_{0}$.
\end{enumerate}
\end{lem}
\begin{proof}
The first statement follows from the triangle inequality of the $L^{1}$-norm in $\ZZ^{m}$. The second statement holds since $P_{t}^{\calD_{0}}$ is the disjoint union of $t^{m-\left|\calD_{0}\right|}$ translates of
$P_{1}^{\calD_{0}}$. We turn to prove the third statement.

Define $\calD=\calD_{0}\cup\calC\left(\calD_{0}\right)$
as in Definition \ref{def:Pt}. We have
\begin{align*}
\left(-P_{t}^{\calD_{0}}+P_{t}^{{\mathcal{D}}_{0}}\right)\cdot u_{0} & =\left\{ x\in\ZZ^{m}\mid{\mathcal{M}}_{{\mathcal{D}}}^{-1}(x)\in\left(-1,1\right)^{k}\times\left(-2t,2t\right)^{m-k}\right\} \cdot u_{0}\\
 & =\left\{ x\in\ZZ^{m}\mid{\mathcal{M}}_{{\mathcal{D}}}^{-1}(x)\in[0,1)^{k}\times\left(-2t,2t\right)^{m-k}\right\} \cdot u_{0}\\
 & \subseteq P_{2t}^{\calD}\cdot u_{0}\,\,\,\text{,}
\end{align*}
where the second equality above follows from the fact that $\calD_{0}\subseteq\Stab_{\ZZ^{m}}\left(u_{0}\right)$.
\end{proof}
The next proposition yields Tool C, namely, we show that for
a pointed $\Gamma$-set $(V,v_{0})$, there is an injection $f_{C}:(Y,y_{0})\to(V,v_{0})$,
where $Y$ is a finite $\Gamma$-set and $\Eq\left(f_{C}\right)$
is large. We also provide some control over word-metric neighborhoods
of $\Ima\left(f_{C}\right)$ in $V$.
\begin{prop}
\label{prop:Pt-isoperimetric-properties}Let $H\leq\ZZ^{m}$, $V=\ZZ^{m}/H$,
$v_{0}=0+H$ and $t\in\NN$. Take a basis ${\mathcal{D}}_{0}$ for $H$.
Then, there is a finite quotient $Y$ of $V$, and an injective map
$f_{C}:Y\rightarrow P_{t}^{\calD_{0}}\cdot v_{0}$, such that:
\begin{enumerate}
\item \label{enu:Pt-IsoperimetricEquivariance} $\left|\Eq\left(f_{C}\right)\right|\geq\left(1-\frac{m-\left|\calD_{0}\right|}{t}\right)\cdot\left|Y\right|$.
\item \label{enu:Pt-IsoperimetricBall}For every integer $\tilde{t}\geq0$, $\left|B_{V}\left(\Ima\left(f_{C}\right),\tilde{t}\right)\right|\leq\left(1+\frac{\tilde{t}}{t}\right)^{m-\left|\calD_{0}\right|}\cdot\left|\Ima\left(f_{C}\right)\right|$.
\end{enumerate}
\end{prop}
\begin{proof}
Let $k=\left|{\mathcal{D}}_{0}\right|$ and write $P_{r}=P_{r}^{{\mathcal{D}}_{0}}$ for every $r\in\NN$.  Let ${\mathcal{D}}={\mathcal{D}}_{0}\cup{\mathcal{C}}\left({\mathcal{D}}_{0}\right)$
and  ${\mathcal{T}}={\mathcal{D}}_{0}\cup2t\cdot{\mathcal{C}}({\mathcal{D}}_{0})$. Note that 
\begin{equation} \label{eq:ToolCH}
H=\left\langle {\mathcal{D}}_{0}\right\rangle =\left\{ x\in\mathbb{Z}^{m}\mid{\mathcal{M}}_{{\mathcal{D}}}^{-1}\left(x\right)\in\ZZ^{k}\times\left\{ 0\right\} ^{m-k}\right\} 
\end{equation}
 and 
\begin{equation} \label{eq:ToolCT}
\left\langle {\mathcal{T}}\right\rangle =\left\{ x\in\mathbb{Z}^{m}\mid{\mathcal{M}}_{{\mathcal{D}}}^{-1}\left(x\right)\in\ZZ^{k}\times\left(2t\ZZ\right)^{m-k}\right\} \,\,\text{.}
\end{equation}

Define $Y=\ZZ^{m}/\langle{\mathcal{T}}\rangle$ and $y_{0}=0+\langle{\mathcal{T}}\rangle$. Let $f_{C}=F_{P_{t},y_{0},v_{0}}:P_{t}\cdot y_{0}\rightarrow P_{t}\cdot v_{0}$.
First, we show that $f_{C}$ is well-defined and injective. By the
remark after Definition \ref{def:F-map}, it is enough to show that
$\langle{\mathcal{T}}\rangle\cap\left(-P_{t}+P_{t}\right)=\langle{\mathcal{D}}_{0}\rangle\cap\left(-P_{t}+P_{t}\right)$.
Clearly, the right-hand side is contained in the left-hand side, so
we only need to show the reverse inclusion. Let $x\in-P_{t}+P_{t}$.
Then, for $k+1\le i\le m$, we have $\left({\mathcal{M}}_{{\mathcal{D}}}^{-1}(x)\right)_{i}\in[-t,t)-[-t,t)=\left(-2t,2t\right)$.
If, further, $x\in\langle{\mathcal{T}}\rangle$, then $\left({\mathcal{M}}_{{\mathcal{D}}}^{-1}(x)\right)_{i}=0$,
and so $x\in\langle{\mathcal{D}}_{0}\rangle$, as required.

For $v\in V=\ZZ^m/H$ and $k+1\le i\le m$, define $\left({\mathcal{M}}_{{\mathcal{D}}}^{-1}(v)\right)_{i}=\left({\mathcal{M}}_{{\mathcal{D}}}^{-1}(x)\right)_{i}$
using any $x\in\ZZ^{m}$ such that $v=x+H$. This is well-defined due to Equation (\ref{eq:ToolCH}). By Lemma \ref{lem:strong-bounds},
for every $k+1\leq i\leq m$ and $1\leq j\leq m$,
\begin{equation}
\left|\left({\mathcal{M}}_{{\mathcal{D}}}^{-1}\left(e_{j}+x\right)-{\mathcal{M}}_{{\mathcal{D}}}^{-1}\left(x\right)\right)_{i}\right|\leq1\,\,\text{.}\label{eq:strong-slow-walk}
\end{equation}
We turn to proving \ref{enu:Pt-IsoperimetricEquivariance} and \ref{enu:Pt-IsoperimetricBall}.
\begin{enumerate}
\item First, we show that 
\begin{equation}
P_{t-1}\cdot y_{0}\subseteq\Eq\left(f_{C}\right)\,\,\text{.}\label{eq:P_(t-1)-in-Eq}
\end{equation}
Write $S=\left\{ e_{1},\ldots,e_{m}\right\} \subseteq \ZZ^m $. Since the domain $P_{t}\cdot y_{0}$
of $f_{C}$ is equal to the entire set $Y$, all of its points are
internal in $Y$. Therefore, a point $y\in P_{t-1}\cdot y_{0}$
belongs to $\Eq\left(f_{C}\right)$ if and only if $f_{C}$ preserves
$y\overset{s}{\longrightarrow}$ for every $s\in S$. Hence, by Lemma
\ref{lem:edge-preservation}, it is sufficient to prove that $\langle{\mathcal{T}}\rangle\cap\left(-P_{t}+S+P_{t-1}\right)=\langle{\mathcal{D}}_{0}\rangle\cap\left(-P_{t}+S+P_{t-1}\right)$.
Let $x\in-P_{t}+S+P_{t-1}$. By Equation (\ref{eq:strong-slow-walk}),
$\left({\mathcal{M}}_{{\mathcal{D}}}^{-1}(x)\right)_{i}\in(-2t,2t)$ for all $k+1\le i\le m$.
Using (\ref{eq:ToolCH}) and (\ref{eq:ToolCT}) as before, we see that $x\in\langle{\mathcal{D}}_{0}\rangle$
if and only if $x\in\langle\calT\rangle$, which proves Inclusion
(\ref{eq:P_(t-1)-in-Eq}). Finally, using Lemma \ref{lem:Pt-basic-properties}\ref{enu:Pt-basic-GrowthInt},
we get
\begin{align*}
\left|\Eq\left(f_{C}\right)\right| & \geq\left|P_{t-1}\cdot y_{0}\right|=\left|P_{t-1}\right|\\
&=\left|P_1\right| \cdot \left(t-1\right)^{m-\left|\calD_0\right|} \\
 & =\left|P_{t}\right|\cdot\left(1-\frac{1}{t}\right)^{m-\left|\calD_{0}\right|}\\
 & \geq\left|Y\right|\cdot\left(1-\frac{m-\left|\calD_{0}\right|}{t}\right)\,\,\text{.}
\end{align*}
\item Take an integer $\tilde{t}\geq0$. First, we show that $B_{V}\left(P_{t}\cdot v_{0},\tilde{t}\right)\subseteq P_{t+\tilde{t}}\cdot v_{0}$.
Note that $P_{r}\cdot v_{0}=\left\{ v\in V\mid\left({\mathcal{M}}_{{\mathcal{D}}}^{-1}(v)\right)_{i}\in[-r,r)\text{ for }k+1\le i\le m\right\} $
for every $r\in\NN$. Let $v\in B_{V}\left(P_{t}\cdot v_{0},\tilde{t}\right)$.
By Equation (\ref{eq:strong-slow-walk}), $\left({\mathcal{M}}_{{\mathcal{D}}}^{-1}(v)\right)_{i}\in\left[-(t+\tilde{t}),t+\tilde{t}\right)$,
so $v\in P_{t+\tilde{t}}\cdot v_{0}$. Finally, since $\Ima\left(f_{C}\right)=P_{t}\cdot v_{0}$,
using Lemma \ref{lem:Pt-basic-properties}\ref{enu:Pt-basic-GrowthInt} again, we get
\begin{align*}
\left|B_{V}\left(\Ima\left(f_{C}\right),\tilde{t}\right)\right| & \leq\left|P_{t+\tilde{t}}\cdot v_{0}\right|\\
 & =\left|P_{t}\right|\cdot\left(1+\frac{\tilde{t}}{t}\right)^{m-\left|\calD_{0}\right|}\\
 & =\left|\Ima\left(f_{C}\right)\right|\cdot\left(1+\frac{\tilde{t}}{t}\right)^{m-\left|\calD_{0}\right|}\,\,\text{.}
\end{align*}
\end{enumerate}
\end{proof}

\subsection{Tool B: The method of far-reaching short sets}

We turn to \emph{Tool B}, in which we are given a pointed, usually
infinite, $\Gamma$-set $\left(V,v_{0}\right)$ and a positive integer
$t$. Our goal is to create a pointed $\Gamma$-set $(U,u_{0})=\left(\ZZ^{m}/\left\langle {\mathcal{D}}\right\rangle ,0+\left\langle {\mathcal{D}}\right\rangle \right)$
and a subgraph isomorphism $f_{B}:\left(P_{t}^{\calD}\cdot u_{0},u_{0}\right)\to\left(P_{t}^{\calD}\cdot v_{0},v_{0}\right)$.
Here, ${\mathcal{D}}\subseteq\ZZ^{m}$ is a linearly independent set of
our choice, and $P_{t}^{\calD}\cdot v_{0}$ must be bounded within
a ball of radius $O(t)$. By Lemma \ref{lem:Pt-basic-properties}\ref{enu:Pt-basic-BallAroundParallelopiped},
a sufficient condition for this bound is that $\|{\mathcal{D}}\|_{1}\le O(t)$.
Throughout this section, the reader should keep the discussion in
the beginning of Section \ref{sec:AbelianGroupsArePolynomiallyStable}
in mind, and, in particular, the realization of $\Gamma$ as the quotient
$\ZZ^{m}/K$, and the constants $C_{d}$ and $t_{E}$. 

As before, we may assume that $V$ is transitive, so it can be regarded
as $\mathbb{Z}^{m}/H$ for some $K\le H\le\mathbb{Z}^{m}$, with $v_{0}=0+H$.
Our strategy is to find a linearly independent set ${\mathcal{D}}\subseteq H$
of\emph{ short vectors, }such that the restriction \emph{$f_{B}$}
of the quotient map $U=\mathbb{Z}^{m}/\left\langle {\mathcal{D}}\right\rangle \to\mathbb{Z}^{m}/H=V$
to $P_{t}^{{\mathcal{D}}}\cdot u_{0}$ is a subgraph-isomorphism. To show
that such a set ${\mathcal{D}}$ exists, we first prove Lemma \ref{lem:good-basis},
which asserts the existence of a subset ${\mathcal{D}}\subseteq H$, consisting
of short vectors, such that all vectors in $H$ up to a certain length
$R=R\left(\|{\mathcal{D}}\|_{1},t\right)$ are spanned by ${\mathcal{D}}$.
Proposition \ref{prop:tile-construction} then shows that this set
${\mathcal{D}}$ is suitable for our purposes. 
\begin{lem}
\label{lem:good-basis}Let $K\le H\le\ZZ^{m}$ and $t_{E}\le t\in\NN$.
Then, there is a linearly independent subset ${\mathcal{D}}\subseteq H$,
such that $K\le\left\langle {\mathcal{D}}\right\rangle $ and the following
holds:
\begin{enumerate}
\item \label{enu:good-basisNorm}$\|{\mathcal{D}}\|_{1}\le2\cdot7^{d}\cdot d^{2d+2}\cdot t$, 
\item \label{enu:good-basisSelfExplaining}$\langle{\mathcal{D}}\rangle\cap B_{\ZZ^{m}}\left(R\right)=H\cap B_{\ZZ^{m}}\left(R\right)$
for $R=2\cdot\left(\|{\mathcal{D}}\|_{1}+dt\right)+1$.
\end{enumerate}
\end{lem}
\begin{proof}
Let $\tau:\ZZ^{m}\to\ZZ^{d}\oplus\{0\}^{m-d}$ denote the projection
onto the first $d$ coordinates, and let Let $\bar{H}=\tau(H)$. For
each $0\leq i\leq d+1$, define $t_{i}=t\cdot\left(7d^{2}\right)^{i}$.
Consider the monotone nondecreasing sequence of integers
\[
\left\{ \rank\langle\bar{H}\cap B_{\ZZ^{m}}\left(t_{i}\right)\rangle\right\} _{i=0}^{d+1}\,\,\text{.}
\]
Since the last element is at most $d$, there must be an integer $0\leq i\leq d$
such that 
\[
\rank\langle\bar{H}\cap B_{\ZZ^{m}}\left(t_{i}\right)\rangle=\rank\langle\bar{H}\cap B_{\ZZ^{m}}\left(t_{i+1}\right)\rangle\,\,\text{.}
\]
Fix such an integer $i$, and define $G=\left\langle \bar{H}\cap B_{\ZZ^{m}}\left(t_{i+1}\right)\right\rangle $.
In particular, $G\cap B_{\ZZ^{m}}\left(t_{i}\right)=\bar{H}\cap B_{\ZZ^{m}}\left(t_{i}\right)$,
and so,
\[
\rank\left\langle G\cap B_{\ZZ^{m}}\left(t_{i}\right)\right\rangle =\rank\left\langle \bar{H}\cap B_{\ZZ^{m}}\left(t_{i}\right)\right\rangle =\rank\left\langle \bar{H}\cap B_{\ZZ^{m}}\left(t_{i+1}\right)\right\rangle =\rank G\,\,\text{.}
\]
Hence, $\rank\langle G\cap B^{L^2}_{\ZZ^{m}}\left(t_{i}\right)\rangle=\rank G$, and so Proposition \ref{prop:KZ-reduced-basis-corollary} yields a basis $\bar{{\mathcal{D}}}_{0}$ for $G$ satisfying 
\begin{equation}
\|{\bar{\calD}}_{0}\|_{1}\leq d^{2}\cdot t_{i}\,\,.\label{eq:GoodBasisShort}
\end{equation}
Recall the definition of the set $T\subseteq\ZZ^{m}$ from the beginning
of Section \ref{sec:AbelianGroupsArePolynomiallyStable}. Construct
a set ${\mathcal{D}}_{0}\subseteq H$ consisting of one preimage $h\in H$
for each $\bar{h}\in\bar{{\mathcal{D}}}_{0}$ (i.e., $\tau(h)=\bar{h}$), such that $h-\bar{h}\in T$ (this is possible since $K\leq H$). So, each
of the last $m-d$ coordinates of each $h\in\calD_{0}$ is in the
range $0,\dots,\beta_{m}-1$. Now, define $H_{T}=H\cap\left(\left\{ 0\right\} ^{d}\oplus\ZZ^{m-d}\right)$.
Then, $K\leq H_{T}$, and so $\rank H_{T}=m-d=\rank K$. Since $K$
has a basis $\beta_{d+1}\cdot e_{d+1},\dotsc,\beta_{m}\cdot e_{m}$,
which is contained in $B_{\left\{ 0\right\} ^{d}\oplus\ZZ^{m-d}}^{L^{2}}\left(\beta_{E}\right)$,
we conclude from Proposition \ref{prop:KZ-reduced-basis-corollary}
that $H_{T}$ has a basis ${\mathcal{D}}_{1}$, with $\|{\mathcal{D}}_{1}\|_{1}\le\left(m-d\right)^{2}\cdot\beta_{E}$.
Then, the set ${\mathcal{D}}={\mathcal{D}}_{0}\cup{\mathcal{D}}_{1}$ is linearly
independent, and 
\begin{align*}
\|{\mathcal{D}}\|_{1} & =\|{\mathcal{D}}_{0}\|_{1}+\|{\mathcal{D}}_{1}\|_{1}\\
 & \leq\sum_{x\in\bar{{\mathcal{D}}}_{0}}\left(\|x\|_{1}+\left(m-d\right)\cdot\beta_{E}\right)+\left(m-d\right)^{2}\cdot\beta_{E}\\
 & =\|\bar{{\mathcal{D}}}_{0}\|_{1}+\left(\left|{ \bar{\calD}}_{0}\right|\cdot(m-d)+(m-d)^{2}\right)\cdot\beta_{E}\\
 & \leq\|\bar{{\mathcal{D}}}_{0}\|_{1}+\left(d+(m-d)\right)\cdot\left(m-d\right)\cdot\beta_{E}\\
 & \leq d^{2}\cdot t_{i}+m\cdot(m-d)\cdot\beta_{E}\,\,\text{,}
\end{align*}
where the last inequality follows from Equation (\ref{eq:GoodBasisShort}).
Now, $t\geq t_{E}$ implies that $m\cdot\left(m-d\right)\cdot\beta_{E}\le d\cdot t$,
and so,
\[
\|{\mathcal{D}}\|_{1}\leq d^{2}\cdot t_{i}+d\cdot t\le d^{2}\cdot\left(7d^{2}\right)^{d}\cdot t+d\cdot t\le2\cdot7^{d}\cdot d^{2d+2}\cdot t\,\,\text{,}
\]
proving \ref{enu:good-basisNorm}. Furthermore, for $R=2\cdot\left(\|{\mathcal{D}}\|_{1}+dt\right)+1$,
\begin{align*}
R&\leq2\cdot\left(d^{2}\cdot t_{i}+m\cdot(m-d)\cdot\beta_{E}+dt\right)+1\\ &\leq 2\cdot\left(3\cdot d^2\cdot t_i\right)+1\leq7\cdot d^{2}\cdot t_{i}=t_{i+1}\,\,\text{.}
\end{align*}
Hence, to prove \ref{enu:good-basisSelfExplaining} it suffices to show that 
\[
\langle{\mathcal{D}}\rangle\cap B_{\ZZ^{m}}\left(t_{i+1}\right)=H\cap B_{\ZZ^{m}}\left(t_{i+1}\right)\,\,\text{.}
\]
Let $h\in H\cap B_{\ZZ^{m}}\left(t_{i+1}\right)$ and write $\bar{h}_{0}=\tau(h)$.
So,
\[
\bar{h}_{0}\in\bar{H}\cap B_{\ZZ^{m}}\left(t_{i+1}\right)\subseteq G=\langle\bar{\calD}_{0}\rangle\,\,\text{.}
\]
Hence, there exists $h_{0}\in\left\langle {\mathcal{D}}_{0}\right\rangle $
such that $\tau\left(h_{0}\right)=\bar{h}_{0}$, and so $h-h_{0}\in H_{T}=\langle\calD_{1}\rangle$.
Hence, $h=h_{0}+\left(h-h_{0}\right)\in\langle\calD\rangle$, proving
the claim.
\end{proof}
The following proposition yields Tool B.
\begin{prop}
\label{prop:tile-construction}Let $K\leq H\le\ZZ^{m}$ and $t_{E}\le t\in\NN$.
Write $V=\ZZ^{m}/H$ and $v_{0}=0+H$. Then, there is a linearly independent
set ${\mathcal{D}}\subseteq H$, $K\leq\langle{\mathcal{D}}\rangle$, satisfying
the following:
\begin{enumerate}
\item $P_{t}^{{\mathcal{D}}}\subseteq B_{\ZZ^{m}}\left(C_{d}\cdot t\right)$.
\item For $U=\ZZ^{m}/\langle{\mathcal{D}}\rangle$ and $u_{0}=0+\langle{\mathcal{D}}\rangle$,
there is a subgraph isomorphism $f_{B}:P_{t}^{{\mathcal{D}}}\cdot u_{0}\rightarrow P_{t}^{{\mathcal{D}}}\cdot v_{0}$
(given by $f_{B}=F_{P_{t}^{{\mathcal{D}}},u_{0},v_{0}}$).
\end{enumerate}
\end{prop}
\begin{proof}
Apply Lemma \ref{lem:good-basis} to $H$ and $t$ to obtain a linearly
independent set $\calD$ with $K\le\left\langle {\calD}\right\rangle $
such that 
\[
\|{\mathcal{D}}\|_{1}\le2\cdot7^{d}\cdot d^{2d+2}\cdot t
\]
and
\begin{equation}
\langle{\mathcal{D}}\rangle\cap B_{\ZZ^{m}}\left(R\right)=H\cap B_{\ZZ^{m}}\left(R\right)\,\,\text{,}\label{eq:TileConstructionPropertyTwo}
\end{equation}
where $R=2\cdot\left(\|{\mathcal{D}}\|_{1}+dt\right)+1$. Write $P=P_{t}^{{\calD}}$.
The first claim follows from  Lemma \ref{lem:Pt-basic-properties}\ref{enu:Pt-basic-BallAroundParallelopiped}
since
\begin{align}
P & \subseteq B_{\mathbb{Z}^{m}}\left(\|{\mathcal{D}}\|_{1}+(m-|\calD|)\cdot t\right)\subseteq B_{\mathbb{Z}^{m}}\left(\|{\mathcal{D}}\|_{1}+d\cdot t\right)\label{eq:TileConstructionParallelotopeInBall}
\end{align}
and 
\[
\|{\mathcal{D}}\|_{1}+d\cdot t\le2\cdot7^{d}\cdot d^{2d+2}\cdot t+d\cdot t\le C_{d}\cdot t\,\,\text{.}
\]

We turn to proving that $F_{P_{t}^{{{\mathcal{D}}}},u_{0},v_{0}}$ is a subgraph
isomorphism. By Lemma \ref{lem:subgraph-isomorphism}, it suffices
to show that, for $S_{1}=\left\{ e_{1},\dotsc,e_{m}\right\} \cup\left\{ 0\right\} \subseteq\ZZ^{m}$,
\[
H\cap\left(-P+S_{1}+P\right)=\langle{\mathcal{D}}\rangle\cap\left(-P+S_{1}+P\right)\,\,.
\]
Equation (\ref{eq:TileConstructionParallelotopeInBall}) yields $-P+S_{1}+P\subseteq B_{\ZZ^{m}}\left(2\cdot\left(\|{\mathcal{D}}\|_{1}+dt\right)+1\right)= B_{\ZZ^{m}}(R)$,
hence the claim follows from Equation (\ref{eq:TileConstructionPropertyTwo}). 
\end{proof}

\subsection{Tool A: The bounded-addition method} \label{subsec:ToolA}

We turn to {\em Tool A}. Here, we are given a pointed $\FF_{m}$-set
$(X,x_{0})$ and $t\in\NN$ such that 
\begin{equation}
\Box_{\FF_{m}}(t)\cdot x_{0}\subseteq X_{E}\,\,\text{,}\label{eq:X_EContainsBox}
\end{equation}
i.e., $X$ abides by the set of equations $E$ within a certain neighborhood
of $x_{0}$. Our goal is to construct a pointed $\Gamma$-set $\left(U,u_{0}\right)$
and a subgraph isomorphism $f_{A}:B_{U}\left(u_{0},r\right)\to B_{X}\left(x_{0},r\right)$
for some $r\ge\Omega(t)$. A key notion in our proof is ``bounded
addition'', formalized in the following definition:
\begin{defn}
\label{def:bounded-addition}Let ${\mathcal{D}}\subseteq B_{\ZZ^{m}}^{L^{\infty}}(R)$,
$R\in\NN$. Define $\left[{\mathcal{D}}\right]_{R}\subseteq B_{\ZZ^{m}}^{L^{\infty}}(R)$
as the minimal subset of $\ZZ^{m}$ satisfying the following conditions:
\begin{enumerate}[label=(\Roman*)]
\item ${\mathcal{D}}\cup\left\{ 0\right\} \subset\left[{\mathcal{D}}\right]_{R}$.
\item If $x\in\left[{\mathcal{D}}\right]_{R}$, then $-x\in\left[{\mathcal{D}}\right]_{R}$.
\item If $x_{1},x_{2}\in\left[{\mathcal{D}}\right]_{R}$ and $\left\Vert x_{1}+x_{2}\right\Vert _{\infty}\le R$,
then $x_{1}+x_{2}\in\left[{\mathcal{D}}\right]_R$.
\end{enumerate}
\end{defn}
Informally, the relevance of Definition \ref{def:bounded-addition}
to the problem Tool A aims to solve is the following: Inclusion (\ref{eq:X_EContainsBox})
yields a guarantee on the behavior of $X$ \emph{in a neighborhood
of $x_{0}$} (see Lemma \ref{lem:good-box-canonical-action}). Therefore,
we would like to know which elements of $\ZZ^{m}$ can be generated
from a given finite subset $\calD\subseteq\ZZ^{m}$ \emph{without straying
far from the origin}.

We proceed to develop \emph{bounded addition} in order to prove Lemma \ref{lem:effective-lattice-generation}
below, and then use it to provide Tool A. We begin with two immediate
observations about $\left[{\mathcal{D}}\right]_R$ without proof :
\begin{lem}
\label{lem:BoundedGenerationFacts}Let ${\mathcal{D}}\subseteq B_{\ZZ^{m}}^{L^{\infty}}(R)$,
$R\in\NN$. Then:
\begin{enumerate}
\item \label{enu:BoundedGenerationSequence}A vector $y\in\ZZ^{m}$ belongs to $\left[{\mathcal{D}}\right]_{R}$ if
and only if there is a sequence $\left\{ y_{i}\right\} _{i=1}^{k}\subseteq B_{\ZZ^{m}}^{L^{\infty}}(R)$
with $y=y_{k}$, in which every element $y_{i}$ satisfies at least
one of the following:
\begin{enumerate}
\item $y_{i}\in{\mathcal{D}}$, or
\item $y_{i}$ is the negation of a previous element, or
\item $y_{i}$ is the sum of two previous elements.
\end{enumerate}
\item \label{enu:BoundedGenerationClosed}For $\mathcal{E}\subseteq \ZZ^m$, if ${\mathcal{E}}\subseteq\left[{\mathcal{D}}\right]_{R}$ then $\left[{\mathcal{E}}\right]_{R}\subseteq\left[{\mathcal{D}}\right]_{R}$.
\end{enumerate}
\end{lem}
The sequence $\left\{ y_{i}\right\} _{i=1}^{k}$ in Lemma \ref{lem:BoundedGenerationFacts}\ref{enu:BoundedGenerationSequence}
is called a \emph{$\left({\mathcal{D}},R\right)$-generation-sequence }for
$y$. 

Lemma \ref{lem:bader-ofer} provides a technical result, which is
then used by Lemmas \ref{lem:bader-ofer-cor}, \ref{lem:bader-ofer-basis}
and \ref{lem:effective-lattice-generation} to give sufficient conditions
for membership in $\left[{\mathcal{D}}\right]_{R}$. The proof of Lemma
\ref{lem:bader-ofer} is inspired by \emph{Thomas Jefferson's method
}for proportional allocation of seats in the United States House of
Representatives \textendash{} a method used in many countries to this
day. Recall that ${\mathcal{D}}^{\pm}$ is the set of elements of ${\mathcal{D}}$
and their negations.
\begin{lem}
\label{lem:bader-ofer}Let ${\mathcal{D}}\subseteq B_{\ZZ^{m}}^{L^{\infty}}(R)$,
$R\in\NN$, and let $y\in\langle{\mathcal{D}}\rangle$. Then, there is
a finite sequence $x_{1},\dotsc,x_{N}$ of elements of ${\mathcal{D}}^{\pm}$,
such that $y=\sum_{j=1}^{N}x_{j}$, and for every $0\leq k\leq N$,
\begin{equation}
\left\Vert \sum_{j=1}^{k}x_{j}-\frac{k}{N}\cdot y\right\Vert _{\infty}\leq2\cdot\left|{\mathcal{D}}\right|\cdot R\,\,\text{.}\label{eq:Bader_oferLemCondition}
\end{equation}
\end{lem}
\begin{proof}
Let $M$ be a multiset of elements of ${\mathcal{D}}^{\pm}$ whose sum
is $y$. Assume, without loss of generality, that $M$ does not contain both
$x$ and $-x$ for any $x\in\calD$. Write $\bar{x}_{1},\dotsc,\bar{x}_{p}$
($p\le\left|{\mathcal{D}}\right|)$ for the distinct elements of $M$,
and $a_{i}$ for the multiplicity of $\bar{x}_{i}$ in $M$. Thus,
$y=\sum_{i=1}^{p}a_{i}\cdot\bar{x}_{i}$. Let $N=|M|=\sum_{i=1}^{p}a_{i}$.
We seek to order the elements of $M$ in a sequence $x_{1},\ldots,x_{N}$
satisfying the claim.

Let 
\[
F=\left\{ \left(N\cdot\frac{b}{a_{i}},i\right)\in\left[0,N\right]\times[p]\mid i\in[p]\text{ and }1\leq b\leq a_{i}\right\} \,\,\text{.}
\]
It is helpful to think of an element $(z,i)\in F$ as a flag of color
$i$, located at $z\in\left[0,N\right]$. So, for every $i\in\left[p\right]$,
we have $a_{i}$ flags of color $i$, positioned evenly from $\frac{N}{a_{i}}$
to $N$. Let $\left(z_{1},i_{1}\right),\dotsc,\left(z_{N},i_{N}\right)$
be a sequence, consisting of all elements of $F$, ordered so that
$z_{1}\leq z_{2}\leq\dots\leq z_{N}$. For $1\leq j\leq N$, we set
$x_{j}=\bar{x}_{i_{j}}$. Note that $y=\sum_{j=1}^{N}x_{j}$,
so we only need to show that Equation (\ref{eq:Bader_oferLemCondition})
holds.

First, we would like to show that
\begin{equation}
\forall j\in\left[N\right]\,\,\,\,\,\,\,\,z_{j}\geq j\,\,\text{.}\label{eq:BaderOferj<z_j}
\end{equation}
For $z\in[0,N]$, let $w_{z}=\left|\left\{ j\in[N]\mid z_{j}\le z\right\} \right|$,
i.e., the number of flags up to location $z$. We claim that $w_{z}\le z$.
Indeed, $\left|\left\{ j\in[N]\mid z_{j}\le z\text{ and }i_{j}=i\right\} \right|=\left\lfloor z\cdot\frac{a_{i}}{N}\right\rfloor $
for $i\in[p]$. Summing over $i$ yields 
\[
w_{z}=\sum_{i=1}^{p}\left\lfloor z\cdot\frac{a_{i}}{N}\right\rfloor \le\sum_{i=1}^{p}z\cdot\frac{a_{i}}{N}=z.
\]
In particular, taking $z=z_{j}$, we have $w_{z_{j}}\leq z_{j}$,
but clearly $w_{z_{j}}\geq j$, and so Inequality (\ref{eq:BaderOferj<z_j})
follows.

For $i\in\left[p\right]$ and $l\in [N]$, define $b_{i,l}=\left|\left\{ 1\le j\le l\mid i_{j}=i\right\} \right|$,
that is, the number of $i$-colored flags among the first $l$ flags. Let $k\in[N]$. We seek to bound the $L^{1}$-distance $\Delta$ between the vectors
\begin{equation}
\left(b_{i,k}\right)_{i=1}^{p}\in\mathbb{R}^{p}\,\,\,\,\text{and}\,\,\,\,\left(\frac{k}{N}\cdot a_{i}\right)_{i=1}^{p}\in\mathbb{R}^{p}\,\,\text{.}\label{eq:ofer-bader-vectors}
\end{equation}
Let $i\in\left[p\right]$. We would like to show that 
\begin{equation}
b_{i,k}\geq\frac{k}{N}\cdot a_{i}-1\,\,\text{.}\label{eq:BaderOferb_i,k}
\end{equation}
This clearly holds if $k<\frac{N}{a_{i}}$. Assume that $k\ge\frac{N}{a_{i}}$
and define $k'=\left\lfloor \frac{k}{N}\cdot a_{i}\right\rfloor $.
Let $j$ be minimal such that $b_{i,j}=k'$. In other words, the $k'$-th
$i$-colored flag is $j$-th among all flags. Note that $z_{j}=\frac{N}{a_{i}}\cdot k'$.
By Inequality (\ref{eq:BaderOferj<z_j}), $j\le z_{j}=\frac{N}{a_{i}}\cdot k'\le k$.
Equation (\ref{eq:BaderOferb_i,k}) follows as 
\[
b_{i,k}\ge b_{i,j}=k'\ge\frac{k}{N}\cdot a_{i}-1\,\,\text{.}
\]

Since the two vectors in (\ref{eq:ofer-bader-vectors}) have the same
sum $k$, the distance between them is
\[
\Delta=\sum_{i=1}^{p}\left|\frac{k}{N}\cdot a_{i}-b_{i,k}\right|=2\cdot\sum_{i=1}^{p}\max\left\{ 0,\frac{k}{N}\cdot a_{i}-b_{i,k}\right\} \,\,\text{,}
\]
so $\Delta\le2p$ due to Equation (\ref{eq:BaderOferb_i,k}). Now,
\begin{align*}
\left\Vert \sum_{j=1}^{k}x_{j}-\frac{k}{N}\cdot y\right\Vert _{\infty} & =\left\Vert \sum_{i=1}^{p}b_{i,k}\cdot x_{i}-\frac{k}{N}\cdot\sum_{i=1}^{p}a_{i}\cdot x_{i}\right\Vert _{\infty}=\left\Vert \sum_{i=1}^{p}\left(b_{i,k}-\frac{k}{N}\cdot a_{i}\right)\cdot x_{i}\right\Vert _{\infty}\\
 & \le\sum_{i=1}^{p}\left|b_{i,k}-\frac{k}{N}\cdot a_{i}\right|\cdot R\\
 & =\Delta\cdot R\le2p\cdot R\le2\cdot|{\mathcal{D}}|\cdot R\,\,\text{.}
\end{align*}
\end{proof}
\begin{lem}
\label{lem:bader-ofer-cor} Let ${\mathcal{D}}\subseteq B_{\ZZ^{m}}^{L^{\infty}}(R)$,
$R\in\NN$, and $y\in\left\langle {\mathcal{D}}\right\rangle $.
Then, $y\in\left[{\mathcal{D}}\right]_{2\cdot\left|{\mathcal{D}}\right|\cdot R+\|y\|_{\infty}}$.
\end{lem}
\begin{proof}
By Lemma \ref{lem:bader-ofer}, there is a sequence $x_{1},\dotsc,x_{m}$
of elements of ${\mathcal{D}}^{\pm}$, such that $y=\sum_{j=1}^{m}x_{j}$,
and for every $0\leq k\leq m$, we have $\|\sum_{j=1}^{k}x_{j}-\frac{k}{m}\cdot y\|_{\infty}\leq2\cdot\left|{\mathcal{D}}\right|\cdot R$.
Consequently, $\|\sum_{j=1}^{k}x_{j}\|_{\infty}\leq2\cdot\left|{\mathcal{D}}\right|\cdot R+\|y\|_{\infty}$,
so by induction on $k$, each of the partial sums $\sum_{j=1}^{k}x_{j}$
belongs to $\left[{\mathcal{D}}\right]_{2\cdot\left|{\mathcal{D}}\right|\cdot R+\|y\|_{\infty}}$.
In particular, $k=m$ yields the claim.
\end{proof}
Proposition \ref{prop:KZ-reduced-basis-corollary} says that a lattice
in $\ZZ^{m}$ which is generated by short elements has a short basis.
The following lemma is an analogue statement in the context of bounded addition.
\begin{lem}
\label{lem:bader-ofer-basis}Let ${\mathcal{D}}\subseteq B_{\ZZ^{m}}^{L^{\infty}}\left(R\right)$,
$R\in\NN$. The lattice $\langle{\mathcal{D}}\rangle$ admits a basis contained
in $\left[{\mathcal{D}}\right]_{5m^{2}\cdot R}\cap B_{\ZZ^{m}}^{L^{\infty}}\left(m\cdot R\right)$.
\end{lem}
\begin{proof}
Write ${\mathcal{D}}=\left\{ x_{1},\dotsc,x_{p}\right\} $. For $0\leq i\leq p$,
let ${\mathcal{D}}_{i}=\left\{ x_{1},\dotsc,x_{i}\right\} $. We prove
by induction on $i$ that $\left\langle {\mathcal{D}}_{i}\right\rangle $
has a basis ${\mathcal{E}}_{i}$ contained in $\left[{\mathcal{D}}\right]_{5m^{2}\cdot R}\cap B_{\ZZ^{m}}^{L^{\infty}}\left(m\cdot R\right)$.
Taking $ {\mathcal{E}}_{0}=\emptyset$, the base case
$i=0$ is immediate. Assume that $i\ge1$ and that $\left\langle {\mathcal{D}}_{i-1}\right\rangle $
has a basis ${\mathcal{E}}_{i-1}$ contained in $\left[{\mathcal{D}}\right]_{5m^{2}\cdot R}\cap B_{\ZZ^{m}}^{L^{\infty}}\left(m\cdot R\right)$.
We consider two generating sets for $\left\langle {\mathcal{D}}_{i}\right\rangle $:
\begin{enumerate}[label=(\Roman*)]
\item The set ${\mathcal{D}}_{i}$, which is contained in $B_{\ZZ^{m}}^{L^{\infty}}\left(R\right)$.
\item The set ${\mathcal{E}}_{i-1}\cup\left\{ x_{i}\right\} $, which is contained
in $B_{\ZZ^{m}}^{L^{\infty}}\left(m\cdot R\right)$, and has at most
$m+1$ elements.
\end{enumerate}
By virtue of the former generating set and by Proposition \ref{prop:KZ-reduced-basis-corollary},
$\left\langle {\mathcal{D}}_{i}\right\rangle $ has a basis ${\mathcal{E}}_{i}\subseteq B_{\ZZ^{m}}\left(m\cdot R\right)\subseteq B_{\ZZ^{m}}^{L^{\infty}}\left(m\cdot R\right)$.
It is now enough to show that ${\mathcal{E}}_{i}\subseteq\left[{\mathcal{D}}\right]_{5m^{2}\cdot R}$.
Let $y\in{\mathcal{E}}_{i}$. Then, $\|y\|_{\infty}\leq m\cdot R$, and
since $y\in\calE_{i}\subseteq\langle\calD_{i}\rangle$, we have $y\in\langle\calE_{i-1}\cup\left\{ x_{i}\right\} \rangle$.
As $\langle\calE_{i}\cup\left\{ x_{i}\right\} \rangle\subseteq B_{\ZZ^{m}}^{L^{\infty}}\left(m\cdot R\right)$
and $\left|\calE_{i-1}\cup\left\{ x_{i}\right\} \right|\leq m+1$,
Lemma \ref{lem:bader-ofer-cor} implies that
\[
y\in\left[\calE_{i-1}\cup\left\{ x_{i}\right\} \right]_{2\left(m+1\right)\cdot mR+mR}\subseteq\left[\calE_{i-1}\cup\left\{ x_{i}\right\} \right]_{5m^{2}\cdot R}\,\,\text{.}
\]
As ${\mathcal{E}}_{i-1}\subseteq\left[{\mathcal{D}}\right]_{5m^{2}\cdot R}$
by the induction hypothesis, and since $\left\{ x_{i}\right\} \subseteq\left[\calD\right]_{5m^{2}\cdot R}$,
it follows from Lemma \ref{lem:BoundedGenerationFacts}\ref{enu:BoundedGenerationClosed} that $y\in\left[{\mathcal{D}}\right]_{5m^{2}\cdot R}$. 
\end{proof}
\begin{lem}
\label{lem:effective-lattice-generation}Let ${\mathcal{D}}\subseteq B_{\ZZ^{m}}^{L^{\infty}}\left(R\right)$,
$R\in\NN$. Then, $\langle{\mathcal{D}}\rangle\cap B_{\ZZ^{m}}^{L^{\infty}}\left(R\right)=\left[{\mathcal{D}}\right]_{5m^{2}\cdot R}\cap B_{\ZZ^{m}}^{L^{\infty}}\left(R\right)$.
\end{lem}
\begin{proof}
Given $x\in\langle{\mathcal{D}}\rangle\cap B_{\ZZ^{m}}^{L^{\infty}}\left(R\right)$,
we need to show that $x\in\left[{\mathcal{D}}\right]_{5m^{2}\cdot R}$. Lemma
\ref{lem:bader-ofer-basis} yields a basis ${\mathcal{E}}\subseteq\left[{\mathcal{D}}\right]_{5m^{2}\cdot R}\cap B_{\ZZ^{m}}^{L^{\infty}}\left(m\cdot R\right)$
for $\left\langle {\mathcal{D}}\right\rangle $. Since $x\in \langle \mathcal{E} \rangle$, Lemma \ref{lem:bader-ofer-cor} implies that
$x\in\left[{\mathcal{E}}\right]_{2\left|{\mathcal{E}}\right|\cdot mR+R}\subseteq\left[{\mathcal{E}}\right]_{5m^{2}\cdot R}$.
Lemma \ref{lem:BoundedGenerationFacts}\ref{enu:BoundedGenerationClosed} yields the claim, since
${\mathcal{E}}\subseteq\left[{\mathcal{D}}\right]_{5m^{2}\cdot R}$.
\end{proof}
At this point, we have established the required groundwork concerning
bounded addition. Before proving our main proposition about Tool A,
we also need the following lemma. It states that, within $X_{E}$,
an $\FF_{m}$-set $X$ behaves in a certain sense like a $\ZZ^{m}$-set.
The reader should recall, from Section \ref{subsec:GeometricDefinitions},
the definition of a \emph{sorted word} and the notation $\hat{w}$
for a given $w\in\FF_{m}$.
\begin{lem}
\label{lem:good-box-canonical-action}Let $X$ be an $\FF_{m}$-set,
$x\in X$ and $R\in\NN$. Assume that $\Box_{\FF_{m}}\left(R\right)\cdot x\subseteq X_{E}$.
Then, $w\cdot x=\hat{w}\cdot x$ for every $w\in B_{\FF_{m}}\left(R\right)$.
\end{lem}
\begin{proof}
First, we fix some notation: For a word $w\in\FF_{m}$, whose reduced
form is $w=\hat{e}_{i_{1}}^{\epsilon_{1}}\cdots\hat{e}_{i_{k}}^{\epsilon_{k}}$,
$i_{1},\dotsc,i_{k}\in\left[m\right]$ and $\epsilon_{1},\dotsc,\epsilon_{k}\in\left\{ +1,-1\right\} $,
write $\iota\left(w\right)$ for the number of \emph{inversions} in
$w$, namely, 
\[
\iota\left(w\right)=\left|\left\{ \left(j_{1},j_{2}\right)\mid1\leq j_{1}<j_{2}\leq k\text{ and }i_{j_{1}}>i_{j_{2}}\right\} \right|\,\,\text{.}
\]

Let $w\in B_{\FF_{m}}(R)$ and write $w=\hat{e}_{i_{1}}^{\epsilon_{1}}\cdots\hat{e}_{i_{k}}^{\epsilon_{k}}$
as above ($k\leq R$). If $\iota\left(w\right)=0$, then $w$ is sorted,
i.e., $w=\hat{w}$, and we are done. Otherwise, take the maximal $1\leq l\leq k-1$
for which $i_{l}>i_{l+1}$. Let $w_{l+1}$ denote the suffix $\hat{e}_{i_{l+1}}^{\epsilon_{l+1}}\cdots\hat{e}_{i_{k}}^{\epsilon_{k}}$
of $w$. Then $w_{l+1}$ is a sorted word by the definition of $l$,
and so $w_{l+1}\in\Box_{\FF_{m}}\left(R\right)$, implying that $w_{l+1}\cdot x\in X_{E}$.
Hence, $\left[\hat{e}_{i_{l}}^{-\epsilon_{l}},\hat{e}_{i_{l+1}}^{\epsilon_{l+1}}\right]\cdot w_{l+1}\cdot x=w_{l+1}\cdot x$.
Define $w'\in\FF_{m}$ by $w'=\hat{e}_{i_{1}}^{\epsilon_{1}}\cdots\hat{e}_{i_{l}}^{\epsilon_{l}}\cdot\left[\hat{e}_{i_{l}}^{-\epsilon_{l}},\hat{e}_{i_{l+1}}^{\epsilon_{l+1}}\right]\cdot w_{l+1}=\hat{e}_{i_{1}}^{\epsilon_{1}}\cdots\hat{e}_{i_{l-1}}^{\epsilon_{l-1}}\cdot\hat{e}_{i_{l+1}}^{\epsilon_{l+1}}\cdot\hat{e}_{i_{l}}^{\epsilon_{l}}\cdot\hat{e}_{i_{l+2}}^{\epsilon_{l+2}}\cdots\hat{e}_{k}^{\epsilon_{k}}$.
Then, $\widehat{w'}=\hat{w}$, $\iota\left(w'\right)<\iota\left(w\right)$
and 
\begin{align*}
w'\cdot x & =\hat{e}_{i_{1}}^{\epsilon_{1}}\cdots\hat{e}_{i_{l}}^{\epsilon_{l}}\cdot\left(\left[\hat{e}_{i_{l}}^{-\epsilon_{l}},\hat{e}_{i_{l+1}}^{\epsilon_{l+1}}\right]\cdot w_{l+1}\cdot x\right)\\
 & =\hat{e}_{i_{1}}^{\epsilon_{1}}\cdots\hat{e}_{i_{l}}^{\epsilon_{l}}\cdot\left(w_{l+1}\cdot x\right)\\
 & =w\cdot x\,\,\text{.}
\end{align*}
Therefore, the claim follows by induction.
\end{proof}
We turn to our main statement in this section, which yields Tool A.
\begin{prop}
\label{prop:toolA} Let $X$ be an $\FF_{m}$-set, $x_{0}\in X$ and
$r\in\NN$. Assume that $\Box_{\FF_{m}}\left(30m^{3}\cdot r\right)\cdot x_{0}\subseteq X_{E}$.
Let $H=\langle\pi\left(\Stab_{\FF_{m}}\left(x_{0}\right)\cap B_{\FF_{m}}\left(2r+1\right)\right)\rangle\leq\ZZ^{m}$
and define the pointed $\ZZ^{m}$-set $\left(U,u_{0}\right)=\left(\ZZ^{m}/H,0+H\right)$.
Then, there is a subgraph isomorphism $f_{A}:B_{U}\left(u_{0},r\right)\to B_{X}\left(x_{0},r\right)$
(given by $f_{A}=F_{B_{\FF_{m}}(r),u_{0},x_{0}}$).
\end{prop}
\begin{proof}
Write $R=2r+1$, and note that 
\[
\Box_{\FF_{m}}(10m^{3}\cdot R)\cdot x_{0}\subseteq X_{E}\,\,\text{.}
\]
Recall the natural surjection $\pi:\FF_{m}\to\ZZ^{m}$, defined in
the beginning of Section \ref{sec:AbelianGroupsArePolynomiallyStable}.
Note that $\Stab_{\FF_{m}}(u_{0})=\pi^{-1}\left(H\right)$. Hence,
by Lemma \ref{lem:subgraph-isomorphism-balls}, it suffices to show
that
\begin{equation}
\Stab_{\FF_{m}}\left(x_{0}\right)\cap B_{\FF_{m}}\left(R\right)=\pi^{-1}\left(H\right)\cap B_{\FF_{m}}\left(R\right)\,\,\text{.}\label{eq:stabilizer-locally-factors-through-commutator}
\end{equation}
 The $\subseteq$ inclusion in Equation (\ref{eq:stabilizer-locally-factors-through-commutator})
is clear from the definition of $H$. We proceed to prove the $\supseteq$ inclusion. 

Define ${\mathcal{D}}=\pi\left(\Stab_{\FF_{m}}\left(x_{0}\right)\cap B_{\FF_{m}}\left(R\right)\right)$,
implying $H=\langle{\mathcal{D}}\rangle\leq\ZZ^{m}$. Note that 
\begin{equation}
{\mathcal{D}}\subseteq B_{\ZZ^{m}}^{L^{\infty}}\left(R\right)\,\text{.}\label{eq:D-in-what}
\end{equation}
Let $w\in\pi^{-1}\left(H\right)\cap B_{\FF_{m}}\left(R\right)$. It
suffices to prove that $w\cdot x_{0}=x_{0}$. Now, 
\begin{equation}
\pi\left(w\right)\in H\cap B_{\ZZ^{m}}^{L^{\infty}}\left(R\right)=\langle{\mathcal{D}}\rangle\cap B_{\ZZ^{m}}^{L^{\infty}}\left(R\right)\label{eq:pi(w)-in-what}\,\,\text{.}
\end{equation}
By Lemma \ref{lem:effective-lattice-generation}, Equations (\ref{eq:D-in-what})
and (\ref{eq:pi(w)-in-what}) imply that $\pi\left(w\right)\in\left[{\mathcal{D}}\right]_{5m^{2}\cdot R}$.
Let $v_{1},\dotsc,v_{k}\in B_{\ZZ^{m}}^{L^{\infty}}\left(5m^{2}\cdot R\right)$
be a $\left({\mathcal{D}},5m^{2}\cdot R\right)$-generation-sequence for
$\pi\left(w\right)$. In particular, $v_{k}=\pi\left(w\right)$. We
shall construct, inductively, a sequence of \emph{sorted} words $w_{1},\dotsc,w_{k}\in\FF_{m}$,
such that for each $1\leq i\leq k$, 
\begin{equation}
w_{i}\in B_{\FF_{m}}\left(5m^{3}\cdot R\right),\,\,\,\,\pi\left(w_{i}\right)=v_{i}\,\,\,\,\text{ and }\,\,\,w_{i}\cdot x_{0}=x_{0}\,\text{.}\label{eq:properties-of-w_i}
\end{equation}
Let $1\leq i\leq k$, and assume that a sequence of sorted words $w_{1},\dotsc,w_{i-1}\in\FF_{m}$,
satisfying (\ref{eq:properties-of-w_i}), has been constructed. We
define $w_{i}$ by considering three separate cases:\renewcommand{\labelenumi}{\Roman{enumi})}
\begin{enumerate}[label=(\Roman*)]
\item Assume that $v_{i}\in{\mathcal{D}}$. By the definition of ${\mathcal{D}}$,
there is $u\in\Stab_{\FF_{m}}\left(x_{0}\right)\cap B_{\FF_{m}}\left(R\right)$
satisfying $\pi\left(u\right)=v_{i}$. Since $\Box_{\FF_{m}}\left(R\right)\cdot x_{0}\subseteq X_{E}$,
Lemma \ref{lem:good-box-canonical-action} implies that $\hat{u}\cdot x_{0}=u\cdot x_{0}$.
Define $w_{i}=\hat{u}$. Then, $w_{i}\cdot x_{0}=x_{0}$ since $u\in\Stab_{\FF_{m}}\left(x_{0}\right)$,
and $w_{i}\in B_{\FF_{m}}\left(R\right)$ since $u\in B_{\FF_{m}}\left(R\right)$.
\item Otherwise, assume that there is $1\leq j<i$ for which $v_{i}=-v_{j}$.
By the induction hypothesis, $w_{j}\in B_{\FF_{m}}\left(5m^{3}\cdot R\right)$
and $w_{j}\cdot x_{0}=x_{0}$. So, the same holds for $w_{j}^{-1}$.
Since $\Box_{\FF_{m}}\left(5m^{3}\cdot R\right)\cdot x_{0}\subseteq X_{E}$,
Lemma \ref{lem:good-box-canonical-action} implies that $\widehat{w_{j}^{-1}}\cdot x_{0}=w_{j}^{-1}\cdot x_{0}=x_{0}$.
Define $w_{i}=\widehat{w_{j}^{-1}}$.
\item Otherwise, there must be integers $j_{1},j_{2}\in\left[i-1\right]$
such that $v_{i}=v_{j_{1}}+v_{j_{2}}$. Define $u=w_{j_{1}}\cdot w_{j_{2}}\in B_{\FF_{m}}\left(10m^{3}\cdot R\right)$.
Then, 
\[
\left\Vert \pi\left(u\right)\right\Vert _{1}=\left\Vert v_{j_{1}}+v_{j_{2}}\right\Vert _{1}=\left\Vert v_{i}\right\Vert _{1}\le m\cdot\left\Vert v_{i}\right\Vert _{\infty}\le5m^{3}\cdot R\,\,\text{.}
\]
This means that, although the length of $u\in\FF_{m}$ is merely bounded
by $10m^{3}\cdot R$, the length of the sorted word $\hat{u}\in\FF_{m}$
is bounded by $5m^{3}\cdot R$. Define $w_{i}=\hat{u}$. Since $\Box_{\FF_{m}}\left(10m^{3}\cdot R\right)\cdot x_{0}\subseteq X_{E}$,
Lemma \ref{lem:good-box-canonical-action} implies that, $w_{i}\cdot x_{0}=u\cdot x_{0}=w_{j_{1}}\cdot w_{j_{2}}\cdot x_{0}=x_{0}$.
\end{enumerate}
Finally, $\hat{w}=w_{k}$ since $w_{k}$ is sorted and $\pi\left(w_{k}\right)=v_{k}=\pi\left(w\right)$.
As $w\in B_{\FF_{m}}\left(R\right)$ and $\Box_{\FF_{m}}\left(R\right)\cdot x_{0}\subseteq X_{E}$,
Lemma \ref{lem:good-box-canonical-action} implies that $w\cdot x_{0}=\hat{w}\cdot x_{0}$.
Thus, $w\cdot x_{0}=w_{k}\cdot x_{0}=x_{0}$, as claimed.
\end{proof}

\subsection{The tiling algorithm - proof of the main theorem \label{subsec:TilingAlgorithm}}

We can now implement the algorithm outlined in Section \ref{subsec:ProofPlanAlgorithmDescription}.
The reader should recall the definitions and objects fixed in the
beginning of Section \ref{sec:AbelianGroupsArePolynomiallyStable}.
In particular we shall refer to the equation-set $E$, the constants
$C_{d}$ and $t_{E}$, the integers $\left\{ \beta_{i}\right\} _{i=d+1}^{m}$, the generators $\left\{ \hat{e}_{1},\dotsc,\hat{e}_{m}\right\} $
of $\FF_{m}$ and the set $\hat T\subseteq \FF_m$. Given an $\FF_{m}$-set $X$, we first discuss the
injection of a single tile into $X$ by means of Tools A, B and C. 
\begin{defn}
\label{def:AdmitsParallelotope}Let $X$ be an $\FF_{m}$-set, $x\in X$,
and $t\in\NN$.
\begin{enumerate}
\item We say that $x$ \emph{admits a $t$-parallelotope} if there is a
linearly independent set ${\mathcal{D}}\subseteq\ZZ^{m}$, where $K\le\left\langle {\mathcal{D}}\right\rangle$, such that the pointed subset $\left(P_{t}^{{\mathcal{D}}}\cdot u_{0},u_{0}\right)$
of $\left(U,u_{0}\right)=\left(\ZZ^{m}/\left\langle {\mathcal{D}}\right\rangle ,0+\left\langle {\mathcal{D}}\right\rangle \right)$
is subgraph isomorphic to $\left(\hat{P}_{t}^{{\mathcal{D}}}\cdot x,x\right)$,
and
\begin{equation}
\hat{P}_{t}^{\mathcal{D}} \subseteq B_{\FF_m}\left(C_{d}\cdot t\right)\,\,\text{.}\label{eq:tAdmitsParalelotopeInBall}
\end{equation}
More elaborately, we also say that\emph{ $x$ admits the $t$-parallelotope
$P_{t}^{\calD}$}.\emph{\label{enu:AdmitsParallelotope}}
\item Fix some arbitrary well-ordering $\prec$ on finite subsets of $\ZZ^{m}$,
e.g., let $\prec$ order $\ZZ^{m}$-subsets lexicographically with
regard to some well-ordering on $\ZZ^{m}$ itself. If $x$ admits
a $t$-parallelotope, we denote by ${\mathcal{D}}_{x,t}$ the $\prec$-minimal
linearly independent set ${\mathcal{D}}$ for which $x$ admits the $t$-parallelotope
$P_{t}^{\calD}$.
\end{enumerate}
\end{defn}
Due to Equation (\ref{eq:tAdmitsParalelotopeInBall}), the collection
of sets ${\mathcal{D}}$ satisfying condition $\ref{enu:AdmitsParallelotope}$ above depends only on the ball of radius $C_{d}\cdot t$
centered at $x$. Since we take $D_{x,t}$ to be the $\prec$-minimum
of this collection, the following simple fact follows.
\begin{lem}
\label{lem:AdmitSameParallelotope}Let $\left(X,x\right)$ and $\left(\tilde{X},\tilde{x}\right)$
be pointed $\FF_{m}$-sets. Let $t\in\NN$, and assume that $B_{X}\left(x,C_{d}\cdot t\right)$
and $B_{\tilde{X}}\left(\tilde{x},C_{d}\cdot t\right)$ are subgraph
isomorphic as pointed sets, and that $x$ admits a $t$-parallelotope.
Then, $\tilde{x}$ also admits a $t$-parallelotope and ${\mathcal{D}}_{x,t}={\mathcal{D}}_{\tilde{x},t}$.
\end{lem}
In light of Definition \ref{def:AdmitsParallelotope}, Proposition
\ref{prop:tile-construction} can be rephrased as ``For every $t\ge t_{E}$,
every point in a $\Gamma$-set admits a $t$-parallelotope''.
\begin{defn}
Let $X$ be an $\FF_{m}$-set and $t\in\NN$. A \emph{t-tile} in $X$
is a pair $\left(x,f\right)$ such that:
\begin{enumerate}
\item $x\in X$ admits a $2t$-parallelotope, and
\item $f$ is a bijection from some finite $\Gamma$-set onto $\hat{P}_{t}^{{\mathcal{D}}}\cdot x\subseteq X$,
where ${\mathcal{D}}={\mathcal{D}}_{x,2t}$.
\end{enumerate}
\end{defn}
Note that, in the above definition we take ${\mathcal{D}}$ to be ${\mathcal{D}}_{x,2t}$,
rather than ${\mathcal{D}}_{x,t}$. This ``extra length'' will be useful
later in controlling the amount of interference between tiles.

As explained in Section \ref{subsec:ProofPlanAlgorithmDescription},
for a tile $\left(x,f\right)$, we want the set $\Eq(f)$ to be large.
Also, we are interested in choosing, from among all tiles, a large
collection of \emph{pairwise disjoint }tiles. To this end, we seek
to minimize the interference between tiles. The sets defined below
are used to measure this interference:
\begin{defn}
Let $A\subseteq X$ for some $\FF_{m}$-set $X$. For $t\in\NN$,
let
\[
\eta_{t}\left(A\right)=A\cup\left\{ x\in X\mid\text{There exists a \ensuremath{t}-tile }\left(x,f\right)\text{ such that }\Ima f\cap A\ne\emptyset\right\} \,\,\text{.}
\]
\end{defn}
We turn to prove the existence of tiles with good parameters (Proposition \ref{prop:SingleTile}). We require the following observation.
\begin{lem} \label{lem:mBoxindBoxT}
Let $X$ be an $\FF_m$-set, $x\in X$ and $t\in \NN$. Assume that 
\begin{equation} \label{eq:Box_DTinX_E}
\Box_{\FF_d}(t) \cdot \hat T \cdot x\subseteq X_E\,\,\text{.}
\end{equation}
Then, $$\Box_{\FF_m}(t) \subseteq X_E\,\,\text{.}$$
\end{lem}
\begin{proof}
For $d+1\le i\le m+1$, let 
\begin{align*}
\hat T_i &= \left\{\prod_{j=i}^m\hat e_j^{\alpha_j} \mid 0\le \alpha_j< \beta_j \right\} \\
\hat T_i^\infty &= \left\{\prod_{j=i}^m\hat e_j^{\alpha_j} \mid \alpha_j\in \ZZ\right\}
\end{align*}

Clearly, $\hat T_i\cdot x \subseteq \hat T_i^\infty \cdot x$ for every $d+1\le i\le m+1$. We show, by induction on $i=m+1,\ldots,d+1$, that this is in fact an equality. The lemma then follows since
$$\Box_{\FF_m}(t)\cdot x \subseteq \Box_{\FF_d}(t)\cdot \hat T_{d+1}^\infty\cdot x = \Box_{\FF_d}(t)\cdot \hat T_{d+1}\cdot x = \Box_{\FF_d}(t)\cdot \hat T\cdot x\,\,\text{.}$$ 

The base case $i=m+1$ is trivial, as $\hat T_{m+1} = \hat T_{m+1}^\infty = \left\{1_{\FF_m}\right\}$. Let $d+1\le i \le m$ and assume that $\hat T_{i+1}\cdot x = \hat T_{i+1}^\infty \cdot x$.  Let $w\in \hat T_i^\infty$. Note that we can write $w = \hat e_i^{\alpha_i} \cdot v$, where  $\alpha_i\in \ZZ$ and $v\in \hat T_{i+1}^\infty$. By the induction hypothesis $v\cdot x\in \hat T_{i+1}\cdot x$. Since $\hat T_{i+1}\subseteq \hat T$, it follows from Equation (\ref{eq:Box_DTinX_E}) that $v\cdot x\in X_E$. In particular, $\hat e_i^{\beta _i} \cdot (v\cdot x) = v\cdot x$. Hence, 
$$w\cdot x = \hat e_{i}^{\alpha_i}\cdot (v\cdot x) =\hat e_{i}^{\bar \alpha_i}\cdot (v\cdot x) \in \hat e_i^{\bar \alpha_i}\cdot \hat T_{i+1}\cdot x$$ 
where $0 \le \bar{\alpha}_i < \beta_i$. As $\hat e_i^{\bar{\alpha}_{i}}\cdot  \hat T_{i+1}\subseteq \hat T_i$, we have $w\cdot x\in \hat T_i\cdot x$. 
\end{proof}
Define the constant 
\[
C_{\Box}=180\cdot m^{3}\cdot C_{d}\,\,\text{.}
\]

\begin{prop}
\label{prop:SingleTile}Let $t\ge t_{E}$ be an integer. Let $X$
be an $\FF_{m}$-set, and $x\in X$ such that 
\begin{equation}
\Box_{\FF_{d}}(C_{\Box}\cdot t)\cdot \hat T\cdot x\subseteq X_{E}\,\,\text{.}\label{eq:SingleTileBoxAssumption}
\end{equation}
Then, there is a $t$-tile $(x,f)$ with 
\begin{equation}
\left|\Eq(f)\right|\ge\left(1-\frac{d}{t}\right)\cdot\left|\Ima f\right|\label{eq:SingleTileEqBound}
\end{equation}
 and for every $t_{E}\le\tilde{t}\le t$,
\begin{equation}
\left|\eta_{\tilde{t}}\left(\Ima f\right)\right|\le\left(1+2C_d\cdot \frac{\tilde{t}}{t}\right)^{d}\cdot\left|\Ima f\right|\,\,\text{.}\label{eq:etaBoundGeneralt}
\end{equation}
Furthermore, 
\begin{equation}
\left|\eta_{t}\left(\Ima f\right)\right|\le2^{d}\cdot\left|\Ima f\right|\,\,\text{.}\label{eq:etaBoundSamet}
\end{equation}
\end{prop}
\begin{proof}
By Assumption (\ref{eq:SingleTileBoxAssumption}) and Lemma \ref{lem:mBoxindBoxT}, we have
\begin{equation} \label{eq:SingleTileBoxAssumptionSimplified}
\Box_{\FF_m}\left(C_{\Box}\cdot t\right)\cdot x\subseteq X_E\,\,\text{.}
\end{equation}
Define 
\begin{align*}
H & =\left\langle \pi\left(\Stab_{\FF_{m}}(x)\cap B_{\FF_{m}}(12\cdot C_{d}\cdot t+1)\right)\right\rangle \,\text{, and}\\
\left(U_{A},u_{A}\right) & =\left(\ZZ^{m}/H,0+H\right)\,\,\text{.}
\end{align*}
Note that $\left\{ \hat{e}_{d+1}^{\beta_{d+1}},\ldots,\hat{e}_{m}^{\beta_{m}}\right\} \subseteq\Stab_{\FF_{m}}(x)\cap B_{\FF_{m}}(12\cdot C_{d}\cdot t+1)$,
since $x\in X_{E}$ and $12\cdot C_{d}\cdot t+1\ge C_{d}\ge\beta_{E}$.
Thus, $K\le H$, so $U_{A}$ is a $\Gamma$-set. By virtue of Inclusion
(\ref{eq:SingleTileBoxAssumptionSimplified}), Proposition \ref{prop:toolA}
(Tool A) yields a subgraph isomorphism 
\[
f_{A}:B_{U_{A}}\left(u_{A},6C_{d}\cdot t\right)\to B_{X}\left(x,6C_{d}\cdot t\right)\,\,\text{.}
\]

By applying Proposition \ref{prop:tile-construction} (Tool B) to
$H$ and $2t$, it follows that $u_{A}$ admits the $2t$-parallelotope
$P_{2t}^{{\mathcal{D}}}$ where ${\mathcal{D}}={\mathcal{D}}_{u_{A},2t}$ and $K\le\left\langle {\mathcal{D}}\right\rangle $.
Hence, for
\[
\left(U_{B},u_{B}\right)=\left(\ZZ^{m}/\left\langle {\mathcal{D}}\right\rangle ,0+\left\langle {\mathcal{D}}\right\rangle \right)\,\,\text{,}
\]
we have a subgraph isomorphism $f_{B}:P_{2t}^{{\mathcal{D}}}\cdot u_{B}\to P_{2t}^{{\mathcal{D}}}\cdot u_{A}$. Since $P_{2t}^{{\mathcal{D}}}\cdot u_{A}\subseteq B_{U_{A}}\left(u_{A},C_{d}\cdot2t\right)$
(due to Equation (\ref{eq:tAdmitsParalelotopeInBall})), we can define $f_{AB}:P_{2t}^{{\mathcal{D}}}\cdot u_{B}\to\hat{P}_{2t}^{{\mathcal{D}}}\cdot x$ by $f_{AB}=f_{A}\circ f_{B}$, and $f_{AB}$ is a subgraph isomorphism since both $f_A$ and $f_B$ are. 
In particular, $x$ admits a $2t$-parallelotope. By virtue of the
subgraph isomorphism $f_{A}$, the balls $B_{U_{A}}\left(u_{A},C_{d}\cdot2t\right)$
and $B_{X}\left(x,C_{d}\cdot2t\right)$ are isomorphic, so Lemma \ref{lem:AdmitSameParallelotope}
implies that ${\mathcal{D}}_{u_{A},2t}={\mathcal{D}}={\mathcal{D}}_{x,2t}$. 

We now apply Proposition \ref{prop:Pt-isoperimetric-properties} (Tool
C) to $\left\langle {\mathcal{D}}\right\rangle $ with the basis ${\mathcal{D}}$,
and to the parameter $t$. This yields a finite $\Gamma$-set $Y$
and a bijection $f_{C}$ from $Y$ onto $P_{t}^{{\mathcal{D}}}\cdot u_{B}$,
with 
\[
\left|\Eq\left(f_{C}\right)\right|\ge\left(1-\frac{m-\left|{\mathcal{D}}\right|}{t}\right)\cdot\left|Y\right|\ge\left(1-\frac{d}{t}\right)\cdot\left|Y\right|\,\,\text{.}
\]
Note that the restriction of $f_{AB}$ to $\Ima f_{C}=P_{t}^{{\mathcal{D}}}\cdot u_{B}$
is a subgraph isomorphism onto $\hat{P}_{t}^{{\mathcal{D}}}\cdot x$. Thus,
$f=f_{AB}\circ f_{C}$ is a bijection from $Y$ onto $\hat{P}_{t}^{{\mathcal{D}}}\cdot x$,
so $\left(x,f\right)$ is a $t$-tile. Also, $\Eq\left(f\right)=\left|\Eq\left(f_{C}\right)\right|\ge\left(1-\frac{d}{t}\right)\cdot\left|\Ima f\right|$,
proving Equation (\ref{eq:SingleTileEqBound}). We turn to prove Equations
(\ref{eq:etaBoundGeneralt}) and (\ref{eq:etaBoundSamet}).

Let $t_{E}\le\tilde{t}\le t$, and consider a $\tilde{t}$-tile $\left(\tilde{x},\tilde{f}\right)$
in $X$. Denote $\tilde{{\mathcal{D}}}={\mathcal{D}}_{\tilde{x},2\tilde{t}}$.
Assume that $\Ima\tilde{f}\cap\Ima f\ne\emptyset$, i.e., $\hat{P}_{\tilde{t}}^{\tilde{{\mathcal{D}}}}\cdot x\cap \hat{P}_{t}^{{\mathcal{D}}}\cdot x\ne \emptyset$. Then, $\tilde{x}\in W\cdot x$,
where 
\[
W=\left(\hat{P}_{\tilde{t}}^{\tilde{{\mathcal{D}}}}\right)^{-1}\cdot\hat{P}_{t}^{{\mathcal{D}}}\,\,\text{.}
\]
In other words, $\eta_{\tilde{t}}\left(\Ima f\right)\subseteq W\cdot x$. Since $W\subseteq \left(\hat{P}_{2\tilde{t}}^{\tilde{{\mathcal{D}}}}\right)^{-1}\cdot\hat{P}_{2t}^{{\mathcal{D}}}$, Equation (\ref{eq:tAdmitsParalelotopeInBall}) implies that 
$$W\subseteq  B_{\FF_{m}}\left(2C_{d}\cdot {\tilde t}+2C_{d}\cdot t\right)\subseteq B_{\FF_{m}}\left(4C_{d}\cdot t\right)\,\,\text{.}$$
Hence, we may consider the restriction of $f_{A}:B_{U_{A}}\left(u_{A},6C_{d}\cdot t\right)\to B_{X}\left(x,6C_{d}\cdot t\right)$
to $W\cdot u_{A}$, which is a subgraph isomorphism $W\cdot u_A\to W\cdot x$. Since $\left|\eta_{\tilde{t}}\left(\Ima f\right)\right|\le\left|W\cdot x\right|$, it follows
that 
\[
\left|\eta_{\tilde{t}}\left(\Ima f\right)\right|\le \left|W\cdot u_{A}\right|\,\,\text{.}
\]
Now, 
\begin{align*}
\left|W\cdot u_A\right|&=\left|\left(P_{\tilde{t}}^{ \tilde{\calD}}\right)^{-1}\cdot P_{t}^{{\mathcal{D}}}\cdot u_A\right|\le\left|\left(P_{2\tilde{t}}^{ \tilde{\calD}}\right)^{-1}\cdot P_{t}^{{\mathcal{D}}}\cdot u_A\right|\\ &\le \left|B_{U_A}\left(P_{t}^{{\mathcal{D}}}\cdot u_A,2C_{d}\cdot\tilde{t}\right)\right| = 
\left|B_{U_B}\left(P_{t}^{{\mathcal{D}}}\cdot u_B,2C_{d}\cdot\tilde{t}\right)\right|\\ &\le\left(1+\frac{2C_{d}\cdot\tilde{t}}{t}\right)^{d}\cdot\left|\Ima\left(f_C\right)\right|\,\,\text{, }
\end{align*}
where the second inequality follows from Equation (\ref{eq:tAdmitsParalelotopeInBall}),
and the third from Proposition \ref{prop:Pt-isoperimetric-properties}\ref{enu:Pt-IsoperimetricBall}.
Equation (\ref{eq:etaBoundGeneralt}) follows since $\left|\Ima\left(f_C\right)\right|=\left|\Ima f\right|$. 

We turn to prove Equation (\ref{eq:etaBoundSamet}). Observe that Equation (\ref{eq:etaBoundSamet})
is a tighter version of (\ref{eq:etaBoundGeneralt}) in the special
case $\tilde{t}=t$. Hence, we continue with the existing notation,
and assume further that $\tilde{t}=t$. Note that, since
$\Gamma$ is abelian, all transitive $\Gamma$-sets are Cayley graphs
of quotients of $\Gamma$, and so balls of the same radius in the
$\Gamma$-set $U_{A}$ are isomorphic.

We have 
\[
B_{X}\left(\tilde{x},2C_{d}\cdot t\right)\subseteq B_{X}\left(W\cdot x,2C_{d}\cdot t\right)\subseteq B_{X}\left(x,6C_{d}\cdot t\right)=\Ima f_{A}\,\,\text{.}
\]
Thus, by virtue of the subgraph isomorphism $f_{A}$, the balls of
radius $2C_{d}\cdot t$ centered at $x$ and at $\tilde{x}$ are both
isomorphic, as subgraphs of $X$, to balls of the same radius in the $\Gamma$-set $U_{A}$.
Hence, these two balls are also subgraph isomorphic to each other. Consequently,
${\mathcal{D}}=\tilde{\mathcal{D}}$ due to Lemma \ref{lem:AdmitSameParallelotope},
and so $W=\left(\hat{P}_{t}^{{\mathcal{D}}}\right)^{-1}\cdot\hat{P}_{t}^{{\mathcal{D}}}$.

Now, Lemma \ref{lem:Pt-basic-properties}\ref{enu:Pt-plusminusParallelopiped} implies that
$W\cdot u_{B}=\left(\hat{P}_{t}^{{\mathcal{D}}}\right)^{-1}\cdot \hat{P}_{t}^{{\mathcal{D}}}\cdot u_{B}\subseteq \hat{P}_{2t}^{{\mathcal{D}}}\cdot u_{B}$.
Hence, the subgraph isomorphism $f_{AB}:P_{2t}^{{\mathcal{D}}}\cdot u_{B}\to\hat{P}_{2t}^{{\mathcal{D}}}\cdot x$ restricts to a subgraph isomorphism $W\cdot u_B \to W\cdot x$. Thus,
\[
\left|\eta_{t}\left(\Ima f\right)\right|\le\left|W\cdot x\right|=\left|W\cdot u_{B}\right|\le\left|\hat{P}_{2t}^{{\mathcal{D}}}\cdot u_{B}\right|\le\left|P_{2t}^{{\mathcal{D}}}\right|\le2^{m-\left|{\mathcal{D}}\right|}\cdot\left|P_{t}^{{\mathcal{D}}}\right|\,\,\text{,}
\]
where the last inequality follows from Lemma \ref{lem:Pt-basic-properties}\ref{enu:Pt-basic-GrowthInt}.
As $\left|{\mathcal{D}}\right|\ge m-d$ and $\left|\Ima f\right|=\left|P_{t}^{{\mathcal{D}}}\right|$,
Equation (\ref{eq:etaBoundSamet}) follows. 
\end{proof}
We turn to discuss an iteration of our algorithm. We require the following
observation.
\begin{lem}
\label{lem:maximal-disjoint-subcollection}Let $\calC=\left(A_{i}\right)_{i\in I}$
be a finite collection of finite sets. Let $c>0$, and assume that
for each $i\in I$, at most $c\cdot\left|A_{i}\right|$ sets $A_{j}$
($j\in I$) intersect $A_{i}$ (including $A_{i}$ itself). Say that
$J\subseteq I$ is \emph{intersection-free} if\emph{ }$A_{j_{1}}\cap A_{j_{2}}=\emptyset$
for all distinct $j_{1},j_{2}\in J$. Then, I has an intersection-free
subset $J$ such that $\left|\bigcup_{j\in J}A_{j}\right|\ge\frac{|I|}{c}$.
\end{lem}
\begin{proof}
Let $J$ be a maximal intersection-free subset of $I$, and let $M=\bigcup_{j\in J}A_{j}$.
Note that $M$ intersects at most $c\cdot|M|$ sets out of $\left(A_{i}\right)_{i\in I}$.
By maximality, each of the $|I|$ elements of $\calC$ intersects $M$, and so $\left|I\right|\le c\cdot\left|M\right|$.
\end{proof}

In the following proposition, which describes a single iteration of the tiling algorithm,
we think of the set $A$ as the image of the tiles already injected
into $X$ in previous iterations.
\begin{prop}
\label{prop:SingleIteration}Let $X$ be a finite $\FF_{m}$-set,
$A\subseteq X$ and $t_{E}\le t\in\NN$. Then, there is a finite
$\Gamma$-set $Y$ and an injection $f:Y\to X\setminus A$, with the
following properties:
\begin{enumerate}
\item $\left|\Ima f\right|\geq\frac{1}{2^{d}}\cdot\left(|X|-\left(3C_{\Box}\cdot t\right)^{d}\cdot\left|\Tor(\Gamma)\right|\cdot\left|X\setminus X_{E}\right|-|\eta_{t}\left(A\right)|\right)$.\label{enu:SingleIterationImageSize}
\item $\left|\Eq f\right|\ge\left(1-\frac{d}{t}\right)\cdot\left|\Ima f\right|$.\label{enu:SingleIterationEquivariance}
\item $\left|\eta_{\tilde{t}}\left(\Ima f\right)\right|\le\left(1+2C_d\cdot \frac{\tilde{t}}{t}\right)^{d}\cdot\left|\Ima f\right|$
for every $t_{E}\le\tilde{t}\le t\cdot$\label{enu:SingleIterationEta}
\end{enumerate}
\end{prop}
\begin{proof}
Let $$M=\left(\Box_{\FF_{d}}\left(C_{\Box}\cdot t\right)\cdot \hat T\right)^{-1}\cdot\left(X\setminus X_{E}\right)$$
and $\tilde{X}=X\setminus\left(M\cup\eta_{t}(A)\right)$. Note that
$\Box_{\FF_{d}}\left(C_{\Box}\cdot t\right)\cdot \hat T\cdot\tilde{X}\subseteq X_{E}$, so Proposition \ref{prop:SingleTile}, applied to each $x\in \tilde X$ separately, yields a set of $t-$tiles
$\left(x,f_{x}\right)_{x\in\tilde{X}}$ . 

Let $x\in \tilde X$. By Equation (\ref{eq:etaBoundSamet}),
$\eta_{t}\left(\Ima f_{x}\right)\le2^{d}\cdot\left|\Ima f_{x}\right|$,
so $\Ima f_{x}$ intersects at most $2^{d}\cdot\left|\Ima f_{x}\right|$
of the sets $\Ima f_{x'}$ ($x'\in\tilde{X}$). In other words, the
collection $\left(\Ima f_{x}\right)_{x\in\tilde{X}}$ satisfies the
requirements of Lemma \ref{lem:maximal-disjoint-subcollection}, with
$c=2^{d}$. Hence, there is a subset $J\subseteq\tilde{X}$ such that
the sets $\left(\Ima f_{x}\right)_{x\in J}$ are pairwise disjoint, and
their union is of size at least $2^{-d}\cdot\left|\tilde{X}\right|$.
We create a $\Gamma$-set $Y=\coprod_{x\in J}\domain\left(f_{x}\right)$
and an injection $f=\coprod_{x\in J}f_{x}:Y\to X$. By the definition of $\tilde X$, we have $\tilde{X}\cap\eta_{t}(A)=\emptyset$.
This means that if $\left(x,f_{x}\right)$ is a $t$-tile and $x\in\tilde{X}$,
then $\Ima f_{x}\subseteq X\setminus A$. By the definition of $f$, it follows that $\Ima f\subseteq X\setminus A$. We proceed to prove
that $f$ has the stated properties. First,
\[
\left|\Ima f\right|\ge2^{-d}\cdot\left|\tilde{X}\right|\ge2^{-d}\left(|X|-\left|M\right|-\left|\eta_{t}\left(A\right)\right|\right)\,\,\text{,}
\]
and so Property \ref{enu:SingleIterationImageSize} follows since
\begin{align*}
\left|M\right| & \le\left|\Box_{\FF_{d}}\left(C_{\Box}\cdot t\right)\cdot \hat T\right|\cdot\left|X\setminus X_{E}\right|=\left(2C_{\Box}\cdot t+1\right)^{d}\cdot\left|\Tor(\Gamma)\right| \cdot\left|X\setminus X_{E}\right|\\
 & \le\left(3C_{\Box}\cdot t\right)^{d}\cdot\left|\Tor(\Gamma)\right|\cdot\left|X\setminus X_{E}\right|\,\,\text{.}
\end{align*}
Property \ref{enu:SingleIterationEquivariance} holds since 
\[
\left|\Eq f\right|\ge\sum_{x\in J}\left|\Eq f_{x}\right|\ge\sum_{x\in J}\left(1-\frac{d}{t}\right)\cdot\left|\Ima f_{x}\right|=\left(1-\frac{d}{t}\right)\cdot\left|\Ima f\right|\,\,\text{,}
\]
where the second inequality is guaranteed by Equation (\ref{eq:SingleTileEqBound}).

Property \ref{enu:SingleIterationEta} follows since, for $t_{E}\le\tilde{t}\le t$, Equation (\ref{eq:etaBoundGeneralt}) yields
\[
\left|\eta_{\tilde{t}}\left(\Ima f\right)\right|\le\sum_{x\in J}\left|\eta_{\tilde{t}}\left(\Ima f_{x}\right)\right|\le\sum_{x\in J}\left(1+2C_{d}\cdot\frac{\tilde{t}}{t}\right)^{d}\cdot\left|\Ima f_{x}\right|=\left(1+2C_{d}\cdot\frac{\tilde{t}}{t}\right)^{d}\cdot\left|\Ima f\right|\,\,\text{.}
\]
\end{proof}
We turn to describe and analyze the tiling algorithm itself.
\begin{prop}
\label{prop:TilingAlgorithm}Let $X$ be an $\FF_{m}$-set, and denote
$n=\left|X\right|$. Let $\delta>0$, and assume that $\frac{\left|X\setminus X_{E}\right|}{n}\le\delta$.
Then, there is a $\Gamma$-set $Y$ and an injection $f:Y\to X$,
with 
\begin{equation}
\left|\Eq\left(f\right)\right|\ge\left(1-C\cdot\delta^{\frac{1}{Q}}\right)\cdot n\,\,\text{,}\label{eq:TilingAlgorithmEquivariance}
\end{equation}
where $C>0$ is a constant which depends only on the equation-set $E$, and $Q\le O\left(2^{d}\cdot d\cdot\max\left\{ d\log d,\log\beta_{E},1\right\} \right)$
with an absolute implied constant.
\end{prop}
\begin{proof}
We inductively define a sequence $f_{1},\ldots,f_{s}$ of injections into $X$. The domain of each $f_i$ is a finite $\Gamma$-set. 

Let $i\ge1$. Assume that the injections $f_{1},\ldots,f_{i-1}$ have
already been defined. Let $A_{i}=\bigcup_{j=1}^{i-1}\Ima f_{j}$, and
define
\begin{align*}
h & =16d\cdot C_{d}+1\\
H_{i} & =24C_{\Box}\cdot\left|\Tor(\Gamma)\right|^{\frac 1d}\cdot\delta^{\frac{1}{d}} \cdot n\cdot h^{i-1}\\
b_{i} & =n-\left|A_{i}\right|\\
t_{i} & =\frac{b_{i}}{H_{i}}\,\,\text{.}
\end{align*}
It is helpful to remember that $H_1,H_2,\ldots $
is a geometric sequence. To generate the injection $f_{i}$, we apply
Proposition \ref{prop:SingleIteration} to the set $X$, with $t=\left\lfloor t_{i}\right\rfloor $
and $A=A_{i}$. We continue this
process as long as $t_{i}\ge 2t_{E}$ (hence, $\lfloor t_{i}\rfloor\ge t_{E}$), and denote the obtained injections by $f_1,\ldots,f_s$.

Proposition \ref{prop:SingleIteration} guarantees that $\Ima f_{i}\cap A_{i}=\emptyset$,
so the sets $\left(\Ima f_{i}\right)_{i=1}^{s}$ are pairwise disjoint.
Hence, we may take $Y=\coprod_{i=1}^{s} \domain\left(f_{i}\right)$
and $f=\coprod_{i=1}^{s}f_{i}$, yielding an injection from the finite
$\Gamma$-set $Y$ into $X$. We turn to prove that $f$ satisfies
Equation (\ref{eq:TilingAlgorithmEquivariance}).

Let $1\le j\le s$. We seek to control the interference of an injection $f_{j}$ with
subsequent iterations. Let $j<i\le s$. By Proposition \ref{prop:SingleIteration},
\[
\left|\eta_{\left\lfloor t_{i}\right\rfloor }\left(\Ima f_{j}\right)\setminus\Ima f_{j}\right|\le\left(\left(1+\frac{2C_{d}\cdot\left\lfloor t_{i}\right\rfloor }{\left\lfloor t_{j}\right\rfloor }\right)^{d}-1\right)\cdot\left|\Ima f_{j}\right|\,\,\text{.}
\]
Note that, as $t_{j}\ge t_{E}\ge2$, we have $\left\lfloor t_{j}\right\rfloor \ge\frac{t_{j}}{2}$.
Hence,
\[
2\cdot\frac{C_{d}\cdot\left\lfloor t_{i}\right\rfloor }{\left\lfloor t_{j}\right\rfloor }\le4\cdot\frac{C_{d}\cdot t_{i}}{t_{j}}=4C_{d}\cdot\frac{b_{i}H_{j}}{b_{j}H_{i}}\le4C_{d}\cdot\frac{H_{j}}{H_{i}}\le4 C_d\cdot \frac{1}{h}\le\frac{1}{4d}\,\,.
\]
In general, $\left(1+x\right)^{d}\le1+2d\cdot x$ for $0\le x\le\frac{1}{2d}$,
and so, 
\begin{align}
\left|\eta_{\left\lfloor t_{i}\right\rfloor}\left(\Ima f_{j}\right)\setminus\Ima f_{j}\right| & \le\frac{4d\cdot C_{d}\cdot\left\lfloor t_{i}\right\rfloor }{\left\lfloor t_{j}\right\rfloor }\cdot\left|\Ima f_{j}\right|\le\frac{8d\cdot C_{d}\cdot t_{i}}{t_{j}}\cdot\left|\Ima f_{j}\right| \nonumber
\\ &\le\frac{8d\cdot C_{d}\cdot t_{i}}{t_{j}}\cdot b_{j}=8d\cdot C_{d}\cdot t_{i}\cdot H_{j}\,\,\text{.} \label{eq:TilingAlgorithmEtaBetweenRoundsBound}
\end{align}

We next give a lower bound on the number of points covered in the
$i$-th iteration. For $1\le i\le s$, Equation (\ref{eq:TilingAlgorithmEtaBetweenRoundsBound}) yields
\begin{align}
\left|\eta_{\left\lfloor t_{i}\right\rfloor}\left(A_{i}\right)\setminus A_{i}\right| & \le\sum_{j=1}^{i-1}\left|\eta_{\left\lfloor t_{i}\right\rfloor}\left(\Ima f_{j}\right)\setminus\Ima f_{j}\right|\le8d\cdot C_{d}\cdot t_{i}\cdot\sum_{j=1}^{i-1}H_{j}\nonumber \\
 & = 8d\cdot C_{d}\cdot t_{i}\cdot\frac{H_{i}-H_1}{h-1}\le \frac{8d\cdot C_d}{h-1}\cdot t_i\cdot H_i \le\frac{t_{i}\cdot H_{i}}{2}=\frac{b_{i}}{2}\,\,.\label{eq:TilingAlgorithmEtaBound}
\end{align}
By Proposition \ref{prop:SingleIteration}\ref{enu:SingleIterationImageSize},
\[
\left|\Ima f_{i}\right|\ge2^{-d}\cdot\left(n-\left(3C_{\Box}\cdot\lfloor t_{i}\rfloor\right)^{d}\cdot\left|\Tor(\Gamma)\right|\cdot\delta n-\left|\eta_{\lfloor t_{i}\rfloor}\left(A_{i}\right)\right|\right)\,\,\text{.}
\]
Equation (\ref{eq:TilingAlgorithmEtaBound}) yields 
\begin{align*}
\left|\Ima f_{i}\right| & \ge2^{-d}\cdot\left(n-\left(3C_{\Box}\cdot\lfloor t_{i}\rfloor\right)^{d}\cdot\left|\Tor(\Gamma)\right|\cdot\delta n-\frac{b_{i}}{2}-\left|A_{i}\right|\right)\\
 & =2^{-d}\cdot\left(b_{i}-\left(3C_{\Box}\cdot\lfloor t_{i}\rfloor\right)^{d}\cdot\left|\Tor(\Gamma)\right|\cdot\delta n-\frac{b_{i}}{2}\right)\\
 & =2^{-d}\cdot\left(\frac{b_{i}}{2}-\left(3C_{\Box}\cdot\lfloor t_{i}\rfloor\right)^{d}\cdot\left|\Tor(\Gamma)\right|\cdot\delta n\right)\\ &\ge2^{-d}\cdot\left(\frac{b_{i}}{2}-\left(3C_{\Box}\cdot t_{i}\right)^{d}\cdot\left|\Tor(\Gamma)\right|\cdot\delta n\right)\,\,\text{.}
\end{align*}
Now, 
\begin{align*}
\left(3C_{\Box}\cdot t_{i}\right)^{d}\cdot\left|\Tor(\Gamma)\right|\cdot\delta n & =\left(\frac{3C_{\Box}\cdot b_{i}}{H_{i}}\right)^{d}\cdot\left|\Tor(\Gamma)\right|\cdot\delta n\\
&\le\left(\frac{3C_{\Box}\cdot b_{i}}{H_{1}}\right)^{d}\cdot\left|\Tor(\Gamma)\right|\cdot\delta n\\
 & =\left(\frac{b_{i}}{8n}\right)^{d}\cdot n\le\frac{b_{i}}{8n}\cdot n=\frac{b_{i}}{8}\le \frac{b_i}{4}\,\,\text{,}
\end{align*}
and consequently,
$$
\left|\Ima f_{i}\right|\ge2^{-\left(d+2\right)}\cdot b_{i}\,\,\text{.}\
$$
Let $\gamma=1-2^{-\left(d+2\right)}$. Then, $b_{i}=b_{i-1}-\left|\Ima f_{i-1}\right|\le b_{i-1}\cdot\gamma$
for every $2\le i\le s$, so $b_{i}\le n\cdot\gamma^{i-1}$.

Finally, we turn to bound $\left|\Eq f\right|$ from below. Proposition
\ref{prop:SingleIteration}\ref{enu:SingleIterationEquivariance}
gives
\begin{align}
\left|\Ima f\right|-\left|\Eq f\right| & =\left(\sum_{i=1}^{s}\left|\Ima f_{i}\right|\right)-\left|\Eq f\right|\le\sum_{i=1}^{s}\left(\left|\Ima f_{i}\right|-\left|\Eq f_{i}\right|\right)\nonumber \\
 & \le\sum_{i=1}^{s}\frac{d}{\left\lfloor t_{i}\right\rfloor }\cdot\left|\Ima f_{i}\right|\le\sum_{i=1}^{s}\frac{2d}{t_{i}}\cdot\left|\Ima f_{i}\right|=\sum_{i=1}^{s}2d\cdot\frac{\left|\Ima f_{i}\right|}{b_{i}}\cdot H_{i}\nonumber \\
 & \le\sum_{i=1}^{s}2d\cdot H_{i}\le2d\cdot H_{s+1}\,\,\text{.} \label{eq:TilingAlgorithmImfMinusEqf}
\end{align}
By the termination condition of the algorithm $t_{s+1}\le 2t_{E}$, but $t_{s+1}=\frac{b_{s+1}}{H_{s+1}}$ and $b_{s+1} = n-\left|\Ima f\right|$. Consequently, $\left|\Ima f\right|\ge n-2t_{E}\cdot H_{s+1}$. Together with Equation (\ref{eq:TilingAlgorithmImfMinusEqf}), this implies
\begin{equation}
\left|\Eq f\right|=\left|\Ima f\right|-\left|\Ima f\setminus\Eq f\right|\ge n-2\left(t_{E}+d\right)\cdot H_{s+1}\,\,\text{.}\label{eq:TilingAlgorithmEqfFirstBound}
\end{equation}
In order to bound $H_{s+1}$, we first need to bound $s$. Appealing
again to the termination condition, we have
\[
2t_{E}\le t_{s}=\frac{b_{s}}{H_{s}}\le\frac{n\cdot\gamma^{s-1}}{H_{1}\cdot h^{s-1}}\,\,\text{, }
\]
so 
\[
s\le \log_{\gamma/h}\left(\frac {2t_E\cdot H_1}{n}\right)+1 \le \frac{\log\left(48t_{E}\cdot C_{\Box}\cdot\left|\Tor(\Gamma)\right|^{\frac 1d}\cdot\delta^{\frac{1}{d}}\right)}{\log\gamma-\log h}+1\,\,\text{.}
\]
Hence, 
\[
h^{s}\le C'\cdot\delta^{\frac{\log h}{d\left(\log\gamma-\log h\right)}}
\]
for some $C'>0$ which depends only on the equation-set $E$. Thus,
\begin{equation}
H_{s+1}=24C_{\Box}\cdot\left|\Tor(\Gamma)\right|^{\frac 1d}\cdot n\cdot\delta^{\frac{1}{d}}\cdot h^{s}\le C''\cdot n\cdot\delta^{\frac{1}{d}+\frac{\log h}{d\left(\log\gamma-\log h\right)}}=C''\cdot n\cdot\delta^{\frac{1}{Q}}\,\,\text{.}\label{eq:TilingAlgorithmH_s Bound}
\end{equation}
where $C''=24C_{\Box}\cdot\left|\Tor(\Gamma)\right|^{\frac 1d}\cdot C'$ and $Q=d\cdot\left(1-\frac{\log h}{\log\gamma}\right)$.
Now, 
\begin{align*}
Q & =d\cdot\left(1-\frac{\log h}{\log\gamma}\right)\le d\cdot\left(1+\frac{\log h}{2^{-(d+2)}}\right)\le O\left(2^{d}\cdot d\cdot\log \left(d\cdot C_d\right)\right)\\
 & \le O\left(2^{d}\cdot d\cdot\max\left\{ d\log d,\log\beta_{E},1\right\} \right)\,\,,
\end{align*}
where the implied constant is absolute. Therefore, the proposition
follows from Equations (\ref{eq:TilingAlgorithmEqfFirstBound}) and
(\ref{eq:TilingAlgorithmH_s Bound}). 
\end{proof}
Finally, we turn to proving the main theorem.
\begin{proof}[Proof of Theorem \ref{thm:AbelianGroupsArePolynomiallyStable}] By Lemma \ref{lem:equivariance-points} and
Proposition \ref{prop:TilingAlgorithm}, for any finite $\Gamma$-set
$X$ we have
\[
G_{E}\left(X\right)\le\left|S\right|\cdot C\delta^{\frac{1}{Q}}\,\,\text{,}
\]
where $C$ and $Q$ are as in Proposition \ref{prop:TilingAlgorithm}.
This yields Equation (\ref{eq:DBound}) from the beginning of Section \ref{sec:AbelianGroupsArePolynomiallyStable}, which is just a more elaborate
form of Theorem \ref{thm:AbelianGroupsArePolynomiallyStable}.
\end{proof}

\section{A lower bound on the polynomial stability degree of $\protect\ZZ^{d}$} \label{sec:LowerBound}

We turn to prove Theorem \ref{thm:DegreeLowerBound}. We rely here on the formulation of stability in terms of group actions and labeled graphs, which was introduced in Section \ref{sec:stability-and-graphs}. The current section 
is independent of Section \ref{sec:AbelianGroupsArePolynomiallyStable},
except for the notion of a \emph{sorted word }and the corresponding
notation $\hat{r}$ for $r\in\ZZ^{d}$ (see Section \ref{subsec:GeometricDefinitions}). 

Fix a constant $d\in\NN$ and fix generator sets $S=\left\{ e_{1},\ldots,e_{d}\right\} $
and $\hat{S}=\left\{ \hat{e}_{1},\ldots,\hat{e}_{d}\right\} $ for
$\ZZ^{d}$ and $\FF_{d}$, respectively. Note that we have a natural
homomorphism $\text{\ensuremath{\pi}:}\FF_{d}\to\ZZ^{d}$ which maps
$\hat{e}_{i}$ to $e_{i}$. Let $E$ denote the \emph{commutator equation-set
}$E_{\text{comm}}^{d}=\left\{ \hat{e}_{i}\hat{e}_{j}\hat{e}_{i}^{-1}\hat{e}_{j}^{-1}\mid1\le i<j\le d\right\} $
from Equation (\ref{eq:E_comm}) and note that $\FF_{d}/\lla E_{\text{comm}}^{d}\rra\cong\ZZ^{d}$.
We seek to construct an infinite sequence of $\FF_{d}$-sets $\left\{ X_{t}\right\} _{t=1}^{\infty}$,
with $\lim_{t\to\infty}L_{E}\left(X_{t}\right)=0$ and 
\begin{equation}
G_{E}\left(X_{t}\right)\ge\Omega_{t\to\infty}\left(L_{E}\left(X_{t}\right)^{\frac{1}{d}}\right)\,\,\text{.}\label{eq:LowerBoundGlobalDefect}
\end{equation}
By virtue of Proposition \ref{prop:stability-rate-in-terms-of-actions}, Inequality (\ref{eq:LowerBoundGlobalDefect}) implies
that $\deg\left(\SR_{E}\right)\ge d$, yielding Theorem \ref{thm:DegreeLowerBound}.

Fix a positive integer $t$. Given $x\in\ZZ^{d}$, let $\left[x\right]$ denote its coset in the quotient group $\ZZ^{d}/\left(t\cdot\ZZ^{d}\right)$.
We build the set $X_{t}$ by taking the natural action of $\FF_{d}$
on $\ZZ^{d}/\left(t\cdot\ZZ^{d}\right)$, removing a single point,
and slightly fixing the result so that it remains an $\FF_{d}$-set.
Concretely, the ground-set for $X_{t}$ is $\left(\ZZ^{d}/\left(t\cdot\ZZ^{d}\right)\right)\setminus\left\{ \left[0\right]\right\} $.
Each generator $\hat{e}_{i}$ acts by taking $\left[(t-1)\cdot e_{i}\right]$
to $\left[e_{i}\right]$, and mapping any other $\left[\sum_{j=1}^{d}a_{j}\cdot e_{j}\right]$
to $\left[\left(\sum_{j=1}^{d}a_{j}\cdot e_{j}\right)+e_{i}\right]$,
as usual. 

We wish to compute $L_{E}\left(X_{t}\right)$. Let $1\le i<j\le d$,
and consider the set of points in which the relation $\left[\hat{e}_{i},\hat{e}_{j}\right]$
is violated, namely $\left\{ x\in X_{t}\mid\hat{e}_{i}\hat{e}_{j}\hat{e}_{i}^{-1}\hat{e}_{j}^{-1}\cdot x\ne x\right\} $.
It is not hard to verify that these are exactly the three points 
\[
\left\{ \left[e_{i}\right],\left[e_{j}\right],\left[e_{i}+e_{j}\right]\right\} \,\,\text{.}
\]
Summing over the $\binom d2$ elements of $E$, Definition \ref{def:LocalDefectOfAction} gives 
\[
L_{E}\left(X_{t}\right)=\frac{1}{\left|X_{t}\right|}\cdot\binom{d}{2}\cdot3\,\,\text{.}
\]
Since $\left|X_{t}\right|=t^{d}-1$, it follows that $L_{E}\left(X_{t}\right)\le O_{t\to\infty}\left(t^{-d}\right)$.
Hence, the following proposition implies Equation (\ref{eq:LowerBoundGlobalDefect}),
which yields Theorem \ref{thm:DegreeLowerBound}.
\begin{prop}
$G_{E}\left(X_{t}\right)\ge\Omega_{t\to\infty}\left(t^{-1}\right)$
\end{prop}
\begin{proof}
Let $Y$ be a $\ZZ^{d}$-set of size $t^{d}-1$ and let $f:Y\to X_{t}$
be a bijection. We need to prove that $\left\Vert f\right\Vert _{S}\ge\Omega_{t\to\infty}\left(\frac{1}{t}\right)$. 

Define the set 
\[
X^{0}=X_{t}\setminus\left(\bigcup_{i=1}^{d}\left\{ \left[k\cdot e_{i}\right]\mid0\le k\le t-1\right\} \right)\,\,\text{,}
\]
and let $Y^{0}=f^{-1}\left(X^{0}\right)$. 

Let $U$ be a connected component of $Y$, i.e., an $\FF_{d}$-orbit,
 and let $u\in U$. Note that the subset $A=\left\{ \sum_{i=1}^{d}\alpha_{i}\cdot e_{i}\mid0\le\alpha_{i}\le t-1\text{ for every }1\le i\le d\right\} $
of $\ZZ^{d}$ does not inject into $U$ at $u$, since $|A|=t^{d}>|Y|\ge|U|$.
Hence, there are $p\ne q\in A$ such that $p\cdot u=q\cdot u$. Let
$0\ne r=p-q$, so $r\in\Stab_{\ZZ^{d}}(u)$. Hence, $r\in\Stab_{\ZZ^{d}}\left(y\right)$
for every $y\in U$. Write $r=\sum_{i=1}^{d}\alpha_{i}\cdot e_{i}$ where
$-(t-1)\le\alpha_{i}\le t-1$ and define $w\in\FF_{d}$ by $w=\hat{r}$
(see Section \ref{subsec:GeometricDefinitions}). We also write $w=w_{1}\cdots w_{k}$
as a reduced word, with $w_{i}\in\hat{S}^{\pm}$. Note that $w\cdot y=y$
for every $y\in U$. However, it is not hard to see that $w\cdot x \ne x$
for each $x\in X^{0}$. Indeed, a non-trivial sorted word in $\Stab_{\FF_{d}}(x)$
must contain some generator or its inverse at least $t$ times. 

For $y\in U\cap Y^0$, let $P_{y}$ denote the set of edges in the
path that starts at $y$ and proceeds as directed by $w$. Namely,
the first such edge is $y\overset{w_{k}}{\longrightarrow}w_{k}\cdot y$,
then $w_{k}\cdot y\overset{w_{k-1}}{\longrightarrow}w_{k-1}w_{k}\cdot y$,
and so on (if $w_{i}$ is the inverse of a generator, we walk along
an edge backwards). Assume that $f$ preserves all of the edges in
$P_{y}$. Then $f(w\cdot y)=w\cdot f(y)$, but this is absurd, since
$w\cdot y=y$, while $w\cdot f(y)\ne f(y)$, as $f(y)\in X^0$. Hence, $P_{y}$ must
contain a non-preserved edge. 

We claim that each edge of $Y$ is contained in at most $t-1$ of the paths $\left\{P_y\right\}_{y\in U\cap Y^0}$. Note that for $1\le j\le k$, the map that takes $y\in U\cap Y^{0}$ to the $j$-th edge in the path $P_{y}$ is an injection. Thus, an
edge labeled $\hat{e}_{i}$ ($i\in [d]$) is contained in at most $\left|\alpha_{i}\right|\le t-1$
of these paths, yielding the claim. Since every
such path contains an edge which is not preserved by $f$, it follows
that $U$ contains at least $\frac{\left|U\cap Y^{0}\right|}{t-1}$
non-preserved edges. Summing over all connected components $U$ yields
\[
\left\Vert f\right\Vert _{S}\ge\frac{1}{|Y|}\sum_{U}\frac{\left|U\cap Y^{0}\right|}{t-1}=\frac{1}{t^d-1}\cdot \frac{\left|Y^{0}\right|}{t-1}=\frac{t^{d}-1-d(t-1)}{\left(t^{d}-1\right)\cdot(t-1)}\ge\Omega_{t\to\infty}\left(\frac{1}{t}\right)\,\,\text{.}
\]
\end{proof}

\section{Discussion}
In this work we proved that finitely generated abelian groups are polynomially stable, and bounded their degree of polynomial stability. For the free abelian group $\ZZ^d$, our lower and upper bounds on the degree are, respectively, $d$ and an exponential expression in $d$. It would be interesting to close this gap. We note that the exponential term in our upper bound comes from the $2^d$-factor in the right-hand side of Equation (\ref{eq:etaBoundSamet}). More precisely, replacing this factor in Equation (\ref{eq:etaBoundSamet}) by some smaller term, would yield the same replacement in Equation (\ref{eq:CBound}).

In particular, when $d$ is small, some of our lemmas can be simplified. For instance, in a lattice of degree $\le 4$, the vectors yielding the local minima form a basis (see p.\ 51 of \cite{Mar02}, cf.\ Proposition \ref{prop:Lenstra}). This fact may be helpful in computing the exact degree of polynomial stability for certain constant values of $d$, starting with $d=2$. 

Another open question that suggests itself is whether polynomial stability holds for a larger class of groups, for example, groups of polynomial growth.

One may also consider a more flexible notion of stability (see Section 4 of \cite{BLpropertyT}): In our definition of the global defect of an $\FF_m$-set $X$, we do not allow adding points to $X$. It is also natural to consider a model where adding points is allowed. More precisely, given two finite $\FF_m$-sets $X$ and $Y$, $|X|\le |Y|$, we allow making $X$ isomorphic to $Y$ by first adding $|Y|-|X|$ points, and then adding and modifying edges. We set the cost of edge addition and edge modification to $\frac 1{|X|}$ per edge. This generalizes Definition \ref{def:DistanceBetweenFSets}. 

Note that our proof of Theorem \ref{thm:DegreeLowerBound} does not hold under the above model, since one can transform $X_t$ to a $\ZZ^d$-set by augmenting it with a single point, and then changing a constant number of edges. We do have a proof (not included in the present paper) applicable for this model, that the degree of polynomial stability of $\ZZ^d$ ($d\ge 2$) is at least $2$. We do not know of a better bound.

In regard to the applications to property testing (Section \ref{subsec:PropertyTesting}), our observation that an equation-set is polynomially stable if and only if its \emph{canonical} tester is efficient, raises the question of which sets of equations admit \emph{any} efficient tester. The subject of \cite{BLtesting} is a similar question in the non-quantitative setting.

Finally, we note that the proof of Proposition \ref{prop:TilingAlgorithm} gives, in fact, a stronger statement: Each orbit of the constructed set $Y$ is of size at most $O(\frac{1}{\delta})$. So, for every finite $\FF_m$-set $X$ and
$\delta>\frac{\left|X\setminus X_E\right|}{\left|X\right|}$,
the set $X$ is $O(\delta^{1/Q})$-close to a $\Gamma$-set $Y$ whose orbits are of size at most $O(\frac{1}{\delta})$.

\section*{Acknowledgments}
We, the authors, would like express our gratitude to Alex Lubotzky for his patient and devoted support, for his encouragement, and for many important suggestions regarding the presentation of the ideas in this work. We are also indebted to Nati Linial for his support and help. 

\appendix
\section{Stability of certain quotients and finite groups}

The first goal of this appendix is to prove Proposition \ref{prop:quotient-by-fg},
which relates the stability rate of a group $\Gamma$ and some of
its quotients. This proposition is of independent interest, and also
enables a somewhat simplified approach to the proof of our main theorem
(Theorem \ref{thm:AbelianGroupsArePolynomiallyStable}), as described
in the beginning of Section \ref{sec:AbelianGroupsArePolynomiallyStable}.
\begin{lem}
\label{lem:local-defect-of-close-actions}Let $\FF$ be a free group
on a finite set $S$, and $E\subseteq\FF$ a finite subset. Let $X$
and $Y$ be finite $\FF$-sets, $\left|X\right|=\left|Y\right|$.
Then,
\begin{equation}
L_{E}\left(Y\right)\leq L_{E}\left(X\right)+\left(\sum_{w\in E}\left|w\right|\right)\cdot d_{S}\left(X,Y\right)\,\,\text{.}\label{eq:appendix-local-defect}
\end{equation}
Consequently,
\begin{equation}
L_{E}\left(Y\right)\leq\left(\sum_{w\in E}\left|w\right|\right)\cdot G_{E}\left(Y\right)\,\,\text{.}\label{eq:local-and-global-general-bound}
\end{equation}
\end{lem}
\begin{proof}
Write $n=\left|X\right|=\left|Y\right|$. Recalling the notation of Section \ref{subsec:stability-in-terms-of-actions}, let $\Phi_{X},\Phi_{Y}:S\rightarrow\Sym\left(n\right)$
be $S$-assignments such that $\FF\left(\Phi_{X}\right)$ is isomorphic
to $X$, $\FF\left(\Phi_{Y}\right)$ is isomorphic to $Y$, and $d_{S}\left(X,Y\right)=d_{n}\left(\Phi_{X},\Phi_{Y}\right)$.
Then, using the triangle inequality and Lemma \ref{lem:d_n(Phi(w),Psi(w))},
we see that
\begin{align*}
L_{E}\left(Y\right) & =L_{E}\left(\Phi_{Y}\right)\\
 & =\sum_{w\in E}d_{n}\left(\Phi_{Y}\left(w\right),1\right)\\
 & \leq\sum_{w\in E}d_{n}\left(\Phi_{Y}\left(w\right),\Phi_{X}\left(w\right)\right)+\sum_{w\in E}d_{n}\left(\Phi_{X}\left(w\right),1\right)\\
 & \leq\sum_{w\in E}\left|w\right|\cdot d_{n}\left(\Phi_{Y},\Phi_{X}\right)+L_{E}\left(\Phi_{X}\right)\\
 & =\left(\sum_{w\in E}\left|w\right|\right)\cdot d_{S}\left(X,Y\right)+L_{E}\left(X\right)\,\,\text{.}
\end{align*}

We turn to the last assertion (\ref{eq:local-and-global-general-bound}).
Assume, for the sake of contradiction, that $L_{E}\left(Y\right)>\left(\sum_{w\in E}\left|w\right|\right)\cdot G_{E}\left(Y\right)$.
Write $\Gamma=\FF/\lla E\rra$. Then, there is a $\Gamma$-set $Z$
such that $L_{E}\left(Y\right)>\left(\sum_{w\in E}\left|w\right|\right)\cdot d_{S}\left(Z,Y\right)$.
But $L_{E}\left(Z\right)=0$, and so, applying Inequality (\ref{eq:appendix-local-defect})
with $X=Z$, we get a contradiction.
\end{proof}
\begin{defn}
Let $F_{1},F_{2}:(0,\left|E\right|]\to[0,\infty)$ be monotone nondecreasing
functions. Write $F_{1}\precsim F_{2}$ if $F_{1}(\delta)\le F_{2}(C\delta)+C\delta$
for some $C>0$. Recalling the equivalence relation $\sim$ from Definition
\ref{def:FunctionEquivalence}, we have $F_{1}\sim F_{2}$ if and
only if $F_{1}\precsim F_{2}$ and $F_{2}\precsim F_{1}$. The partial
order $\precsim$ on functions enables us to define a partial order
on $\sim$-equivalence classes, namely, $\left[F_{1}\right]\precsim\left[F_{2}\right]$
if and only if $F_{1}\precsim F_{2}$.
\end{defn}
The following proposition relates the stability rate (see Definition
\ref{def:GroupStabilityRate}) of a group $\Gamma$ to the stability
rate of certain quotients $\Gamma/N$.
\begin{prop}
\label{prop:quotient-by-fg}Let $\Gamma$ be a finitely-presented
group. Let $N$ be a normal subgroup of $\Gamma$. Assume that $N$
is a finitely-generated group. Then, $\SR_{\Gamma/N}\precsim C\cdot\SR_{\Gamma}$
for some $C=C\left(\Gamma,N\right)>0$. In particular: (i) if $\Gamma$
is stable, then so is $\Gamma/N$, and (ii) $\deg(\SR_{\Gamma/N})\leq\deg\left(\SR_{\Gamma}\right)$.
\end{prop}
\begin{proof}
Let $\pi:\FF\rightarrow\Gamma$ be a presentation of $\Gamma$ as
a quotient of a finitely-generated free group $\FF=\FF_{S}$, $\left|S\right|<\infty$.
Then, $\pi^{-1}\left(N\right)$ is a normal subgroup of $\FF$, $\pi^{-1}\left(N\right)/\Ker\pi\cong N$
is finitely-generated and $\FF/\pi^{-1}\left(N\right)\cong\Gamma/N$.
Let $E_{1}\subseteq\Ker\pi$ be a finite set which generates $\Ker\pi$
as a normal subgroup of $\FF$. Let $E_{2}\subseteq\pi^{-1}\left(N\right)$
be a finite set whose image in $\pi^{-1}\left(N\right)/\Ker\pi$ generates
the group $\pi^{-1}\left(N\right)/\Ker\pi$. Then, $E=E_{1}\cup E_{2}$
generates $\pi^{-1}\left(N\right)$ as a normal subgroup of $\FF$.
So, $\FF/\lla E_{1}\rra\cong\Gamma$ and $\FF/\lla E\rra\cong\Gamma/N$.
Thus, it suffices to show that $\SR_{E}\left(\delta\right)\leq O\left(\SR_{E_{1}}\left(\delta\right)+\delta\right)$.

Let $X$ be a finite $\FF$-set. Write $n=\left|X\right|$ and $\delta=L_{E}\left(X\right)$.
We need to show that $G_{E}\left(X\right)\leq O\left(\SR_{E_{1}}\left(\delta\right)+\delta\right)$.
Let $Z$ be a $\Gamma$-set, $\left|Z\right|=\left|X\right|=n$, for
which $d_{S}\left(X,Z\right)\leq\SR_{E_{1}}\left(L_{E_{1}}\left(X\right)\right)$.
Since $E_{1}\subseteq E$, we have $L_{E_{1}}\left(X\right)\leq L_{E}\left(X\right)$,
and so 
\[
d_{S}\left(X,Z\right)\leq\SR_{E_{1}}\left(L_{E}\left(X\right)\right)=\SR_{E_{1}}\left(\delta\right)\,\,\text{.}
\]
Using Lemma \ref{lem:local-defect-of-close-actions}, we deduce
\begin{align*}
L_{E_{2}}\left(Z\right)\leq L_{E}\left(Z\right) & \leq L_{E}\left(X\right)+\left(\sum_{w\in E}\left|w\right|\right)\cdot d_{S}\left(X,Z\right)\\
 & \le\delta+\left(\sum_{w\in E}\left|w\right|\right)\cdot\SR_{E_{1}}\left(\delta\right)\,\,\text{.}
\end{align*}
This allows use to bound the size of $Z_{E_{2}}$ (see Definition
\ref{def:E-abiding-points}) from below. Indeed, there are $n\cdot L_{E_{2}}\left(Z\right)$
pairs $\left(z,w\right)\in Z\times E$ for which $w\cdot z\neq z$.
Clearly, the number of distinct values of $z$ in those pairs is at
most $n\cdot L_{E_{2}}\left(Z\right)$, and so
\[
\left|Z_{E_{2}}\right|\geq n-n\cdot L_{E_{2}}\left(Z\right)\geq n-n\cdot\left(\delta+\left(\sum_{w\in E}\left|w\right|\right)\cdot\SR_{E_{1}}\left(\delta\right)\right)\,\,\text{.}
\]
Now, as $N=\langle \pi\left(E_2\right)\rangle$, each $z\in Z_{E_{2}}$ is fixed
by $N$. In fact, since $N$ is normal in $\Gamma$, every
$z\in Z_{E_{2}}$ belongs to an orbit $\Gamma\cdot z\subseteq Z$ where
all points are fixed by $N$. Therefore, $Z_{E_{2}}$ is a union of
$\Gamma$-orbits of $Z$, and the action of $\Gamma$ on $Z_{E_{2}}$
factors through $\Gamma/N$. Consider the inclusion map $\iota:Z_{E_{2}}\rightarrow Z$.
All points of $Z_{E_{2}}$ are equivariance points of $\iota$. So,
applying Lemma \ref{lem:equivariance-points} to $\iota$, we get
\begin{align*}
G_{E}\left(Z\right) & \leq\left|S\right|\cdot\left(1-\frac{1}{n}\cdot\left|Z_{E_{2}}\right|\right)\\
 & \leq\left|S\right|\cdot\left(\delta+\left(\sum_{w\in E}\left|w\right|\right)\cdot\SR_{E_{1}}\left(\delta\right)\right)\\
 & \leq O\left(\SR_{E_{1}}\left(\delta\right)+\delta\right)\,\,\text{.}
\end{align*}
Hence,
\begin{align*}
G_{E}\left(X\right) & \leq G_{E}\left(Z\right)+d_{S}\left(X,Z\right)\\
 & \leq O\left(\SR_{E_{1}}\left(\delta\right)+\delta\right)+\SR_{E_{1}}\left(\delta\right)\\
 & \le O\left(\SR_{E_{1}}\left(\delta\right)+\delta\right)\,\,,
\end{align*}
as required.
\end{proof}
Theorem 2 of \cite{GR09} states that all finite groups are stable.
The following is the quantitative version:
\begin{prop}
\label{prop:finite-groups}Let $\Gamma$ be a finite group. Then,
the degree of polynomial stability $\deg\left(\SR_{\Gamma}\right)$
of $\Gamma$ is $1$. In particular, $\Gamma$ is stable.
\end{prop}
\begin{proof}
Write $\delta\mapsto\delta$ for the inclusion map $(0,\left|E\right|]\rightarrow(0,\infty]$,
and $\left[\delta\mapsto\delta\right]$ for its $\sim$-class (Definition
\ref{def:FunctionEquivalence}). Our task is to show that $\SR_{\Gamma}=\left[\delta\mapsto\delta\right]$.
By the definition of the relation $\precsim$, we have $\left[\delta\mapsto\delta\right]\precsim\SR_{\Gamma}$,
and we need to prove the reverse inequality. Let $\pi:\FF\rightarrow\Gamma$
be a presentation of $\Gamma$ as a quotient of a finitely-generated
free group $\FF$. Then, $\left[\FF:\Ker\pi\right]<\infty$, and so
$\Ker\pi$ itself is a finitely-generated group. Therefore, by Proposition
\ref{prop:quotient-by-fg}, for some $C>0$,
\[
\SR_{\Gamma}\precsim C\cdot\SR_{\FF}\sim\left[\delta\mapsto\delta\right]\,\,\text{.}
\]
\end{proof}
 The following proposition shows that for every equation-set, except for the trivial equation-sets $\emptyset$ and $\left\{1\right\}$, the stability rate is at least linear. This motivates the requirement $k\ge 1$ in Definition \ref{def:degree}. 
\begin{prop}
\label{prop:SR-at-least linear}Let $\FF$ be a free group on a finite
set $S$, and $E\subseteq\FF$ a finite subset. Assume that $E\neq\emptyset$
and $E\neq\left\{ 1_{\FF}\right\} $. Then, 
\begin{equation}
\SR_{E}\left(\delta\right)\geq C\cdot\delta\label{eq:SR-at-least-linear}
\end{equation}
 for some $C>0$.
\end{prop}
\begin{proof}
Since $\SR_{E}$ is monotone non-decreasing, it suffices to prove
Inequality (\ref{eq:SR-at-least-linear}) for $\delta\in(0,\delta_{0}]\cap\QQ$
for some $\delta_{0}>0$. As $E$ contains a non-trivial element of
$\FF$, there is a finite $\FF$-set $X$ such that $\delta_{0}\overset{\defname}{=}L_{E}\left(X\right)$
is positive (and rational). Write $\delta_{0}=\frac{m_{0}}{n_{0}}$
for integers $0< m_{0}\le n_{0}$. Take $\delta\in(0,\delta_{0}]\cap\QQ$,
and write $\delta=\frac{p}{q}\cdot\delta_{0}$, where $0< p\leq q$
are integers. Define an $\FF$-set $Y$ as the disjoint union
of $p$ copies of $X$ and $\left(q-p\right)\cdot\left|X\right|$
additional fixed points. Then,
\[
L_{E}\left(Y\right)=\frac{p\cdot\left|X\right|\cdot L_{E}\left(X\right)}{p\cdot\left|X\right|+\left(q-p\right)\cdot\left|X\right|}=\frac{p}{q}\cdot L_{E}\left(X\right)=\delta\,\,\text{.}
\]
Finally, by Inequality (\ref{eq:local-and-global-general-bound})
in Lemma \ref{lem:local-defect-of-close-actions}, $G_{E}\left(Y\right)\geq\left(\sum_{w\in E}\left|w\right|\right)^{-1}\cdot L_{E}\left(Y\right)$,
and so $\SR_{E}\left(\delta\right)\geq\left(\sum_{w\in E}\left|w\right|\right)^{-1}\cdot\delta$,
as required.
\end{proof}

\end{document}